\documentclass[12pt]{article}
\usepackage{fullpage, amsmath, mathrsfs, amssymb, amsfonts, amsthm, stmaryrd, url, hyperref}
\usepackage[all]{xy}
\usepackage{relsize}
\usepackage[bbgreekl]{mathbbol}
\usepackage{amsfonts}
\usepackage{appendix}

\newtheorem{thm}[equation]{{Theorem}}
\newtheorem{cor}[equation]{{Corollary}}
\newtheorem{lem}[equation]{{Lemma}}
\newtheorem{prop}[equation]{{Proposition}}

\newtheorem{conv}[equation]{{Convention}}

\theoremstyle{definition}
\newtheorem{defn}[equation]{Definition}
\newtheorem{rem}[equation]{{Remark}}

\newcommand{\oldbullet}{{[-]}}

\usepackage{tikz}
\usepackage{color}

\newcommand{\quash}[1]{}
\newcommand{\Zp}{{\mathbb{Z}_p}}
\newcommand{\Fp}{{\mathbb {F}_p}}
\newcommand{\OK}{{\mathcal{O}_K}}
\newcommand{\pK}{{{\varpi}_K}}
\newcommand{\eK}{{e_K}}

\newcommand{\EK}{{E_K}}
\newcommand{\Sz}{{\Sk[z]}}
\newcommand{\Szz}{{\Sk[z_0,z_1]}}
\newcommand{\ef}{h}

\renewcommand{\k}{{k}}
\newcommand{\Sk}{{\mathbb{S}_{W(\k)}}}
\newcommand{\thh}{{\mathrm{THH}}}
\newcommand{\tp}{{\mathrm{TP}}}
\newcommand{\tc}{{\mathrm{TC}}}
\newcommand{\can}{{\mathrm{can}}}

\usepackage{scalerel}
\newcommand*{\biglimit}{{\mathrm{lim}}}
\newcommand*{\bigcolimit}{{\mathrm{colim}}}

\DeclareSymbolFontAlphabet{\mathbb}{AMSb} %to ensure that the meaning of \mathbb does not change
\DeclareSymbolFontAlphabet{\mathbbl}{bbold}

\numberwithin{equation}{section}

\usepackage{authblk}
\title{Topological Cyclic Homology of Local Fields}
\author{Ruochuan Liu \thanks{R. Liu is partially supported by the National Natural Science Foundation of China under agreement No. NSFC-11725101 and the Tencent Foundation.} and Guozhen Wang  \thanks{G. Wang is partially supported by the Shanghai Rising-Star Program under agreement No. 20QA1401600 and Shanghai Pilot Program for Basic Research-FuDan University 21TQ1400100(21TQ002) and the National Natural Science Foundation of China under agreement No. NSFC-11801082.}}
\date{}
\begin{document}

\maketitle
\thispagestyle{empty}

\abstract{We introduce a new approach to determining the structure of topological cyclic homology by means of a descent spectral sequence. We carry out the computation for a $p$-adic local field with $\Fp$-coefficients, including the case $p=2$ which was only covered by motivic methods except in the totally unramified case.}
\tableofcontents

%\large

\section{Introduction}

%\section{Main results}

We fix a prime number $p$. Let $K$ be a $p$-adic local field, i.e. a finite extension of $\mathbb{Q}_p$. In this paper, we introduce the \emph{descent spectral sequence} to determine the structure of topological cyclic homology groups $\tc_*(\OK)$ and its variants $\tc^-_*(\OK)$ and $\tp_*(\OK)$. In contrast to all earlier methods, we do not make any additional assumption on $p$ or $K$. In fact, we carry out the computation in the modulo $p$ case, and obtain the structure of $\tc_*(\OK;\mathbb{F}_p)$, which in turn determines the mod $p$ algebraic $K$-theory of $\OK$ by the cyclotomic trace map. Moreover, our computation implies that the natural map
\[
\tc(\OK;\mathbb{Z}_p)\rightarrow \mathrm{L}_{\text{K}(1)}\tc(\OK;\mathbb{Z}_p), 
\] 
where  $\mathrm{L}_{\text{K}(1)}$ denotes the Bousfield localization functor (cf. \cite[\S3.1]{handbook}), is $0$-truncated (see Remark \ref{R:0-trun} for more details). Consequently,  one may completely determine the homotopy type of $\tc(\OK;\mathbb{Z}_p)$ because the homotopy type of $\mathrm{L}_{\text{K}(1)}\tc(\OK;\mathbb{Z}_p)$ is well-understood.

In fact, our approach for determining the topological cyclic homology may apply to more general setup. In a forthcoming paper \cite{LW}, we will treat the case of locally complete intersection schemes over $\mathbb{Z}_p$. 
%Let $p = char(\k)$ be the residual characteristic.

%We construct the descent spectral sequence using relative topological Hochschild homology through an Adams type resolution of the base ring.  

We fix a uniformizer $\pK$ of $K$ throughout. Let $\k$ be the residue field of $\OK$, and let $\Sk$ be the corresponding spherical Witt vectors (cf. \cite[\S5.2]{Lur}).
%We assume that $p = char(\k)$ is odd.
%We assume $\eK>1$. The unramified case will be treated in another paper.
%\begin{rem}
%The normalization of $\EK$ is very important when when $(\EK,p-1)>1$.
%\end{rem}
%Let $\mathbb{N}$ be the additive monoid of natural numbers, and 
%let $\Sz$ be the $E_\infty$-algebra in spectra
%$\Sk \otimes_{\mathbb{S}} \Sigma^\infty\mathbb{N}_+.$
We consider  $\OK$ as an $E_\infty$-algebra over the spherical polynomial algebra $\Sz$
via the composition 
\[
\Sz\to W(k)[z]\to \OK,  
\] 
where the second map is the $W(k)$-algebra morphism sending $z$ to $\pK$. The topological Hochschild homology $\thh(\OK/\Sz)$ of $\OK$ over $\Sz$, which has the structure of an $E_\infty$-algebra in cyclotomic spectra, is introduced by Bhatt-Morrow-Scholze \cite[\S11]{bms2}. Using $\thh(\OK/\Sz)$, one may further define the periodic topological cyclic homology $\tp(\OK/\Sz)$ and the negative topological cyclic homology $\tc^-(\OK/\Sz)$ (See $\S2$ for more details). 

The descent spectral sequence for $\thh(\OK)$ is obtained via descent along the base change map
in cyclotomic spectra
\begin{equation}\label{descent}
\thh(\OK/\Sk)\to \thh(\OK/\Sz).
\end{equation}
%More precisely, we take the Adams type resolution 
%\[
%\Sk \rightarrow \Sz \rightarrow \Szz \rightarrow \dots\qquad{(\star)}
%\]
%for $\Sk$. 
It turns out that the resulting augmented cosimplicial $E_\infty$-algebra in cyclotomic spectra 
\[
\thh(\OK/\Sk) \rightarrow  \thh(\OK/\Sz^{\otimes\oldbullet}), 
\] 
where the tensor products in the target are taken relative to $\Sk$, is a limit diagram
\[
\thh(\OK/\Sk) \simeq \biglimit_{[n]\in \Delta}\thh(\OK/\Sz^{\otimes[n]}).
\] 
This in turn gives rise to the descent spectral sequence for $\thh(\OK)$, which is of homology type and in the second quadrant, 
\[
E^1_{i, j}(\thh(\OK))=\thh_{j}(\OK/\Sz^{\otimes [-i]}) \Rightarrow \thh_{i+j}(\OK/\Sk).
\]
Similarly, we get the descent spectral sequences 
\[
E^1_{i,j}(\tc^{-}(\OK))=\tc^-_{j}(\OK/\Sz^{\otimes [-i]}) \Rightarrow \tc^-_{i+j}(\OK/\Sk)
\]
and
\[
E^1_{i,j}(\tp(\OK))=\tp_{j}(\OK/\Sz^{\otimes [-i]}) \Rightarrow \tp_{i+j}(\OK/\Sk).
\]
Combining these two spectral sequences we obtain the descent spectral sequence for $\tc(\OK)$:
\[
E^2_{i,j}(\tc(\OK)) \Rightarrow \tc_{i+j}(\OK/\Sk),
\]
where $E^2_{i,j}(\tc^-(\OK)), E^2_{i,j}(\tp(\OK))$ and $E^2_{i,j}(\tc(\OK))$ are related by a long exact sequence induced from the fiber sequence
\[
\tc(\OK/\Sk) \rightarrow \tc^-(\OK/\Sk) \xrightarrow{\can-\varphi} \tp(\OK/\Sk).
\]
(See \S \ref{s5} for more details). 

The descent spectral sequence is an analogue of the Adams spectral sequence in the  $\infty$-category of cyclotomic spectra. Indeed, as in the case of standard Adams spectral sequence,  the $E^2$-term of the descent spectral sequence for $\thh(\OK)$ (resp. $\tp(\OK)$) may be identified with the cobar complex for $\thh_*(\OK/\Sz)$ (resp. $\tp_*(\OK/\Sz)$) with respect to the Hopf algebroid 
\[
(\thh_*(\OK/\Sz), \thh_*(\OK/\Szz)) 
\]
(resp. $(\tp_0(\OK/\Sz), \tp_0(\OK/\Szz))$). 

In terms of geometric language, the simplicial scheme $$[n]\mapsto \textrm{Spec}(\tp_0(\OK/\Sz^{\otimes[n]}))$$ forms a groupoid scheme, and there is the associated stack
\begin{equation}\label{X}
\mathcal{X} = \bigcolimit_{[n]\in \Delta^{op}}\mathrm{Spec}(\tp_0(\OK/\Sz^{\otimes[n]})).
\end{equation}
It follows that $E^2_{i,j}(\tp(\OK))$ may be identified with the coherent cohomology $\mathrm{H}^{j-i}(\mathcal{X}, \mathcal{O}_{\mathcal{X}}(\frac{j}{2}))$ with the understanding that this group is zero for $j$ odd, and the aforementioned cobar complex is nothing but the \v{C}ech complex with respect to the covering $$\mathrm{Spec}(\tp_0(\OK/\Sz))\rightarrow \mathcal{X}$$ for determining the said cohomology.
%In the following we will often refer a groupoid object by the tuple of the objects representing its objects and morphisms.

To understand the structure of these Hopf algebroids, we make use of the theory of $\delta$-rings. 
Recall that a $\delta$-ring structure on a $p$-torsion free commutative ring $A$ is equivalent to the datum of a ring map $\varphi: A\to A$ lifting the Frobenius on $A/p$;
the corresponding $\delta$-structure is given by 
\[
\delta(x) = \frac{\varphi(x) - x^p}{p}.
\]
Note that there is a Frobenius lift $\varphi$ on $W(\k)[z_0, \dots, z_n]$ which is the Frobenius on $W(\k)$ and sends $z_i$ to $z_i^p$. This makes $W(\k)[z_0,\dots, z_n]$ into a $\delta$-ring. 

By \cite[\S11]{bms2}, we know that the cyclotomic Frobenius on $\tp_0(\OK/\Sz)$ is a Frobeninus lift. Moreover, as a $\delta$-ring, $\tp_0(\OK/\Sz)$ is isomorphic to the $\EK(z)$-adic completion of $W(\k)[z]$; here $\EK(z)$ is a minimal polynomial of $\pK$ over $W(\k)$, normalized such that $E_K(0)=p$.

It remains to determine $\thh_0(\OK/\Szz)$ and $\tp_0(\OK/\Szz)$, which constitutes a key step of the paper. Let $D_{W(\k)[z_0,z_1]}((\EK(z_0),z_0-z_1))$ be the divided power envelope of $(\EK(z_0),z_0-z_1)$ in $W(\k)[z_0,z_1]$, and equip it with the topology induced by the Nygaard filtration $\mathcal{N}^{\geq\bullet}$. In \S3, we will prove the following result.

\begin{thm}

%It turns out that 
The cyclotomic Frobenius on $\tp_0(\OK/\Szz)$ is a Frobenius lift. Moreover, $\tp_0(\OK/\Szz)$  is isomorphic to the closure of the subring of $D_{W(\k)[z_0,z_1]}((\EK(z_0),z_0-z_1))^\wedge_\mathcal{N}$ generated by $W(\k)[z_0,z_1]$ and $\{\delta^k( {\varphi(z_0-z_1)\over \varphi(\EK(z_0))})\}_{k\geq0}$.
\end{thm}
\noindent In the course of the proof, we also establish a variant of the classical Hochschild-Kostant-Rosenberg theorem (Appendix A),  which might be of  some independent interest. Moreover, we prove that for all $n\geq1$, the Tate spectral sequence for $\tp(\OK/\Sz^{\otimes[n]})$ collapses at the $E^2$-term. Therefore, $(\thh_*(\OK/\Sz),\thh_*(\OK/\Szz))$ is canonically isomorphic to the associated graded Hopf algebroid of $$(\tp_0(\OK/\Sz),\tp_0(\OK/\Szz))$$
with respect to the Nygaard filtration. (See \S \ref{s3}, \S \ref{s4} for more details).

The explicit description of $(\thh_*(\OK/\Sz),\thh_*(\OK/\Szz))$ allows us to determine the $E^2$-term of the descent spectral sequence for $\thh(\OK)$ by the standard relative injective resolution; this was carry out in \S \ref{s5}. 
%This gives a new proof of the main result of \cite{LM}. Then we use the Nygaard filtration on the cobar complex for $\tp_0(\OK/\Sz)$ to construct the algebraic Tate spectral sequence. 

Note that $E^1(\tp(\OK))$ is naturally endowed with the Nygaard filtration, which is inherited from the Nygaard filtration on $\tp(\OK/\Sz^{\otimes\oldbullet})$. Moreover, the associated graded is isomorphic to $E^1(\thh(\OK))[\sigma^\pm]$ by the fact that the Tate spectral sequence for $\tp(\OK/\Sz^{\otimes[n]})$ collapses at the $E^2$-term. Hence the spectral sequence associated with the filtered chain complex $E^1(\tp(\OK))$ takes the form
\[
E^2(\thh(\OK))[\sigma^\pm] \Rightarrow E^2(\tp(\OK)).
\]
We call this spectral sequence the \emph{algebraic Tate spectral sequence}.
Similarly, we construct the \emph{algebraic homotopy fixed points spectral sequence}
\[
E^2(\thh(\OK))[v] \Rightarrow E^2(\tc^-(\OK)),
\]
which is the spectral sequence  associated with the Nygaard filtration on $E^1(\tc^-(\OK))$.
%$\tc^-_*(\OK/\Sz[z]^{\otimes\bullet})$.

Using these two spectral sequences, we conclude that both the descent spectral sequences for $\tc^-(\OK)$ and $\tp(\OK)$ collapse at the $E^2$-term for degree reasons. 

However, extra complication occurs when applying the above approach to determining the $E^2$-terms of the mod $p$ descent spectral sequences.  To remedy this, we introduce a refinement of the Nygaard filtration on $\tp_*(\OK/\Sz^{\otimes \oldbullet};\Fp)$, and determine all the differentials of the \emph{refined algebraic Tate spectral sequence} in \S\ref{s6}. Using the explicit description of refined algebraic Tate differentials, we determine the $E^2$-terms of the descent spectral sequences for $\tc_*^-(\OK;\mathbb{F}_p)$ and $\tp_*(\OK;\mathbb{F}_p)$ in \S\ref{S:e2tc-tp}. In \S\ref{S:e2tc},  we determine the $E^2$-term of the descent spectral sequence for $\tc_*(\OK;\mathbb{F}_p)$. Moreover, we show that all mod $p$ descent spectral sequences collapse at the $E^2$-term.
\begin{rem}
In terms of geometric language, the collapse of the descent spectral sequences for $\tp(\OK)$ and $\tc^-(\OK)$ at the $E^2$-term is a formal consequence of the fact that the stack $\mathcal{X}$ has coherent cohomology in degrees $0$ and $1$ only. Hence the problem of determining the homotopy groups $\tp_*(\OK/\Sk)$ is reduced to the purely algebraic problem of determining the coherent cohomology of the stack $\mathcal{X}$. This is the reason that we can avoid the problems inherent in homotopy theory concerning reduction mod $p$ and introduce the refined Nygaard filtration. Note that the latter has no topological analogue.
\end{rem}

\quash{

In \S\ref{S:bott}, applying the Bott elements in $\text{K}_2(\mathbb{Q}_p(\zeta_{p^n}))$ and the descent spectral sequences for $\tp(\OK; \mathbb{Z}/p^{n})$ and $\tc^-(\OK; \mathbb{Z}/p^{n})$ for all $n\geq1$, we conclude that, under the normalization specified in Theorem \ref{bokperiodicity}, the constant term of 
$E_K(z)$ is equal to $p$.

\begin{rem}
 It might be possible to determine the constant term of 
$E_K(z)$ directly using results on 
$\tp(\mathbb{F}_p)$ and  $\tc^-(\mathbb{F}_p)$, but we do not know how to do that.
\end{rem}
}

Finally, we conclude our main result

\begin{thm}\label{T:main}
Let $d=[K(\zeta_p):K]$ and $f_K = [\k:\mathbb{F}_p]$.
Then we explicitly construct 
\[
\beta \in E^2_{0,2d}(\tc(\OK);\Fp),\quad \lambda \in E^2_{-1,0}(\tc(\OK);\Fp), \quad \gamma\in E^2_{-1,2d+2}(\tc(\OK);\Fp), 
\]
and 
\[
\alpha_{i,l}^{(j)} \in E^2_{-1,2j}(\OK;\mathbb{F}_p), \quad 1\leq i\leq e_K, \quad1\leq j\leq d, \quad 1\leq l\leq f_K
\] 
such that as $\Fp[\beta]$-modules,
\[
E^2_{0,*}(\tc(\OK);\Fp)= \Fp[\beta], \quad E^2_{-2,*}(\tc(\OK);\Fp)= \Fp[\beta]\{\lambda\gamma\},
\]
\[
E^2_{-1,*}(\tc(\OK);\Fp)= \Fp[\beta]\{\lambda, \gamma\} \oplus \mathbb{F}_p[\beta]\{\alpha^{(j)}_{i,l}|1\leq i\leq \eK, 1\leq j\leq d, 1\leq l\leq f_K\},
\]
and
\[
E^2_{i,*}(\tc(\OK);\Fp)=0
\]
for $i\neq0,-1,-2$. Moreover,  both $E^2_{0,*}(\tc(\OK);\Fp)$ and $E^2_{-2,*}(\tc(\OK);\Fp)$ are concentrated in even degrees. 
Consequently, the descent spectral sequence converging to $\tc(\OK;\mathbb{F}_p)$ collapses at the $E^2$-term.
\end{thm}

\begin{rem}
Using Theorem \ref{T:main}, we may completely determine $\tc_*(\OK; \Fp)$. In fact, it turns out that the collapsing descent spectral sequence for $\tc(\OK;\mathbb{F}_p)$
does not have hidden additive extensions unless $p=2$ and $[K:\mathbb{Q}_2]$ is odd. See Theorems \ref{T:oddp} and \ref{T:evenp} for more details. 
\end{rem}

\begin{rem}\label{R:0-trun}
In fact, after base change to $K(\zeta_p)$,  (up to a scalar and up to higher filtrations) $\beta$ is equal to the $d$-th power of a Bott element (see Remark \ref{R:bott}). It follows that $\mathrm{L}_{\text{K}(1)}\tc(\OK;\mathbb{F}_p)$ may be identified with $\tc(\OK;\mathbb{F}_p)[\beta^{-1}]$. Thus Theorem \ref{T:main} implies that 
\[
\tc(\OK;\mathbb{F}_p) \rightarrow \mathrm{L}_{\text{K}(1)}\tc(\OK;\mathbb{F}_p)
 \]  
is $0$-truncated. Consequently, both 
\[
\tc(\OK;\mathbb{Z}_p) \rightarrow \mathrm{L}_{\text{K}(1)}\tc(\OK;\mathbb{Z}_p)
\] 
and 
\[
\text{K}(\OK;\mathbb{Z}_p) \rightarrow \mathrm{L}_{\text{K}(1)}\text{K}(\OK;\mathbb{Z}_p)
\] 
are $0$-truncated as well. Note that the last statement is one formulation of the Lichtenbaum-Quillen conjecture for $K$ (cf. \cite{Tho}).
\end{rem}
\begin{rem}
The $E^2$-terms of the standard Adams spectral sequences for $\mathbb{S}$, which are obtained via the descent along $\mathbb{S}\to \mathrm{H}\mathbb{F}_p$ or $\mathbb{S}\to \mathrm{MU}$, are far from being understood, let alone the abutments. In contrast, we are able to completely determine the $E^2$-term of the descent spectral sequence converging to $\tc_*(\OK;\mathbb{F}_p)$, and in addition, all the $E^2$-differentials are zero.
\end{rem}

\begin{rem}
In this paper we only consider the case of $p$-adic local fields. For the case of local fields of characteristic $p$,  our approach applies equally well. Moreover, one may apply the results in \cite[\S8]{bms2} in place of our constructions in Section \ref{s3} to simplify the process. 
\end{rem}
In \S\ref{S:motivic}, we observe that the descent spectral sequence converging to $\tc_*(\OK;\mathbb{F}_p)$ is reminiscent of the motivic spectral sequence converging to $\text{K}_*(K,\Fp)$. We expect that the descent spectral sequence will provide some incarnation of the motivic spectral sequence in the $p$-adic setting. 
\\

Topological cyclic homology is an important tool for understanding algebraic $K$-theory. The case of $p$-adic local fields has been extensively studied by many people. For example, the case $p$ odd and $\eK=1$ is determined in \cite{BM} and \cite{Tsa}; the case $p$ odd and $\eK$ arbitrary is determined in \cite{HM}. The case $p=2$ and $\eK=1$ is determined in \cite{Rog}. The case $p=2$ and $\eK$ arbitrary follows from the corresponding computation for $2$-completed algebraic $K$-theory \cite{2-adic}, which in turn is based on the work on the $2$-adic Lichtenbaum-Quillen conjecture \cite{RW}. 

These prior works \cite{BM}, \cite{Tsa}, \cite{HM}, \cite{Rog} all adopt a common strategy, which is different form ours; the difference may be summarized by the following diagram 
$$
\xymatrix{
& E^2(\thh(\OK))[\sigma^\pm] \ar@{=>}[ld]_{\text{descent }} \ar@{=>}[rd]^{\text{\; \; \text{algebraic Tate}}} & \\
\thh_*(\OK)[\sigma^\pm] \ar@{=>}[rd]_{\text{Tate }} && E^2(\tp(\OK)). \ar@{=>}[ld]^{\text{descent}}\\
& \tp_*(\OK) &
}
$$
More precisely, in these works one starts with a descent style spectral sequence for $\thh(\OK)$, which collapses at the $E^2$-term in consideration of degrees. Then one applies the Tate spectral sequence to determine the structure of $\tp_*(\OK)$. The hard part is to determine the Tate differentials, and the main technique for doing this is to inductively determine the structures of the Tate spectral sequences for all finite subgroups of the circle group $\mathbb{T}$.

Our approach proceeds in another direction. We first run the (mod $p$) algebraic Tate spectral sequence to determine the $E^2$-term of the descent spectral sequence for $\tp(\OK)$. Since the cobar complex can be described explicitly thanks to the determination of the Hopf algebroid $(\tp_0(\OK/\Sz),\tp_0(\OK/\Szz))$,  the computation of algebraic Tate differentials is purely algebraic. It follows that the descent spectral sequence for $\tp(\OK)$ collapses at the $E^2$-term in consideration of degrees. Indeed, it turns out that the structure of the algebraic Tate spectral sequence is similar to the structure of the Tate spectral sequence (see Remark \ref{r91}). That is to say, using the descent spectral sequence and the Nygaard filtration, we transform  the problem of determining the Tate differentials, which is topological in nature, into a purely algebraic problem which in turn can be solved by hand.

One particular merit of our method is that it allows us to handle all finite extensions  $K/\mathbb{Q}_2$, which was out of reach with earlier methods. In fact, the earlier methods cannot work with $\mathbb{F}_2$-coefficients multiplicatively as the Moore spectrum does not have a multiplication in this case due to the non-vanishing of the Toda bracket $\langle 2,\eta,2\rangle$; if one goes to  $\mathbb{Z}/4$-coefficients, then the various spectral sequences will depend on finer information (as encoded in the coefficients of $E_K(z)$) of the structure of $K$. On the other hand, in our method, all the nontrivial differentials appear in the algebraic Tate spectral sequence, which is algebraic in nature, and there is no difficulty in making modulo $2$ reductions. In other words, by introducing the descent spectral sequence, the obstruction for the multiplicativity of the Moore spectrum goes to a higher filtration, hence does not affect the determinations of the $E^2$-terms.

As indicated in the above diagram, our approach consists of two steps. The first step is to determine the algebraic Tate differentials, which is purely algebraic. The second step is to determine the descent spectral sequence for $\tp$.  As mentioned at the beginning,  we will apply this approach to investigate topological cyclic homology of locally complete intersection schemes over $\mathbb{Z}_p$. It turns out that for those schemes, the descent spectral sequence for $\tp$ is no longer degenerate; the abutment filtrations (which is conjectured to be the motivic filtration of Bhatt-Morrow-Scholze \cite{bms2}) are bounded by the number of generators of the sheaf of regular functions.

%In this paper, all tensor products without a specified base will be over $\mathbb{S}$.

%By abuse of notation, we will denote the $p$-completed topological Hochschild homology 
%$$THH(-;\Zp)$$ simply by $THH(-)$. Similarly for $\tc^-$, $TP$ and $\tc$.
\subsection*{Relation with other works}
The present work started with an attempt to determine the structure of $\tc_*(\OK)$ using the motivic-like spectral sequence introduced by Bhatt-Morrow-Scholze relating the prisms and topological cyclic homology \cite[Theorem 1.12]{bms2}. In fact,  one may resolve $\OK$ by perfectoids in the quasi-syntomic site, and obtain a cosimplicial object similar to $\thh(\OK/\Sz^{\otimes\oldbullet})$, but having $p$-fractional powers of $z_i$'s in the base.
Moreover, the $E^2$-term of the resulting spectral sequence has the descent spectral sequence as a subcomplex consisting of terms with integer exponents. We conjecture that the the second pages of these two spectral sequences are quasi-isomorphic, i.e. the subcomplex consisting of terms with non-integer exponents is acyclic. \\

\noindent In \cite{KN}, Krause-Nikolaus also introduce a descent style spectral sequence to determine the topological Hochschild homology of quotients of DVRs. Their work also recover the main result of Lindenstrauss-Madsen \cite{LM} as ours (Corollary \ref{C:thh-sz}).\\

\noindent Recently, Bhatt-Lurie \cite{BL} and Drinfeld \cite{Dri20} introduced the stacky reformulation of the prismatic cohomology, called the ``\emph{prismatization}".  We believe that the stacks of Bhatt-Lurie and Drinfeld should be closely related to the stacks defined by our Hopf algebroids.  For example, we conjecture that one may identify the prismatization of $\mathrm{Spf}(\OK)$, i.e. the Cartier–Witt stack $\mathrm{WCart}_{\OK}$ constructed in  \cite{BL}, with the stack $\mathcal{X}$ given by (\ref{X}).

\subsection*{Notation and conventions}
We fix a prime $p$. Let $K$ be a finite extension of $\mathbb{Q}_p$ with residue field $\k$. Denote by $K_0=W(\k)[1/p]$ the maximal unramified subextension of $K$ over $\mathbb{Q}_p$. Let $\eK$  and $f_K$ be the ramification index and inertia degree of $K$ over $\mathbb{Q}_p$ respectively. Let $\pK$ be a uniformizer of $\OK$, and let $E_K(z)$ be the minimal polynomial of $\varpi_K$ over $K_0$, normalized such that $E(0)=p$. 
%(We will show in Section 9 that this amouts to requiring $E_K(0)=p$.) 
Let $\mu$ denote the leading coefficient of $\EK(z)$, and put $\tilde{\mu}=-\frac{\mu^p}{\delta(E_K(z_0))}$.

We equip $W(\k)[z_0, \dots, z_n]$ with the Frobenius lift $\varphi$ which is the Frobenius on $W(\k)$ and sends $z_i$ to $z_i^p$. This makes $W(\k)[z_0,\dots, z_n]$ into a $\delta$-ring.

In this paper, all spectral sequences with two indices (descent spectral sequence, Tate spectral sequence, and homotopy fixed point spectral sequence) are homology type spectral sequences with differentials 
\[
d^r: E^r_{i,j} \to E^{r}_{i-r, j+r-1}.
\]
\textbf{Warning:} 
Throughout this paper, all Nygaard filtrations involved only jump at even degrees. For our purpose, we rescale the index of Nygaard filtrations by $2$ after Convention \ref{Conv:nygaard}.

\subsection*{Acknowledgements}
Both authors are very grateful to Lars Hesselholt for suggesting this project, and for his tremendous help during preparation of this paper. Especially, he proposed to us using the descent along the base change map (\ref{descent}) in topological Hochschild homology. 
Thanks also to Benjamin Antieau, Dustin Clausen, Bj{\o}rn Ian Dundas, Jingbang Guo, Marc Levine, Thomas Nikolaus, Paul Arne {\O}stv{\ae}r, John Rognes and Longke Tang for valuable discussions and suggestions. Finally, we would like to thank anonymous referees for many helpful comments that helped us correct and improve earlier versions of this paper.

\section{Cyclotomic structures on THH}
Let $E$ be an $E_\infty$-algebra in spectra, and let $A$ be an $E_\infty$-algebra over $E$. 
Recall that the topological Hochschild homology of $A$ over $E$ is defined by the cyclic bar construction over the base.
\begin{defn}
The topological Hochschild homology of $A$ over $E$ is defined as
\[
\thh(A/E) = A^{\otimes_E\mathbb{T}}
\]
in the $\infty$-category of $E_\infty$-algebras in spectra.  It is universal among the objects of $\mathbb{T}$-equivariant $E_\infty$-algebras over $E$ equipped with a (non-equivariant) map from $A$. Denote $\thh(A/\mathbb{S})$ by $\thh(A)$.
\end{defn}

The universal property of $\thh$ implies the following multiplicative property
\begin{equation}\label{E:rthh-mult}
\thh(A_1/E_1)\otimes_{\thh(A_2/E_2)}\thh(A_3/E_3)\simeq \thh(A_1\otimes_{A_2}A_3/E_1\otimes_{E_2}E_3).
\end{equation}
In general, $\thh(A/E)$ may not have cyclotomic structures. For example, the Hochschild homology $\mathrm{HH}(-) = \thh(-/\mathbb{Z})$ is not cyclotomic (\cite[III.1.10]{NS}). However, we may put more conditions on $E$ to obtain a natural cyclotomic structure on $\thh(A/E)$.
\begin{lem}\label{l3}
The following statements are true.

\begin{enumerate}
\item[(1)]
Let $E$ be an $E_\infty$-algebra in cyclotomic spectra such that the underlying $\mathbb{T}$-action is trivial. 
Then a lift of the augmentation map
\[
\thh(E)\rightarrow E
\]
to a map of $E_\infty$-algebras in cyclotomic spectra naturally determines a lift of the functor $\thh(-/E)$ to a functor from $E_\infty$-algebras over $E$ to $E_\infty$-algebras in cyclotomic spectra. 

\item[(2)]
 Moreover, suppose we have a commutative diagram of $E_\infty$-algebras in cyclotomic spectra
$$
\xymatrix{
\thh(E_1)\ar[d] \ar[r] & E_1 \ar[d]\\
\thh(E_2) \ar[r] & E_2
}
$$
with trivial $\mathbb{T}$-actions on $E_1$ and $E_2$ such that it extends to a commutative diagram of $E_\infty$-algebras in spectra
$$
\xymatrix{
\thh(E_1)\ar[d] \ar[r] & E_1 \ar[r]\ar[d] & A_1 \ar[d]\\
\thh(E_2) \ar[r] &  E_2 \ar[r] & A_2.
}
$$
If we equip $\thh(A_1/E_1)$ and $\thh(A_2/E_2)$ with the structure of  $E_\infty$-algebras in cyclotomic spectra given by (1), then the natural map $\thh(A_1/E_1) \rightarrow \thh(A_2/E_2)$ is a map of $E_\infty$-algebras in cyclotomic spectra.

\end{enumerate}
\end{lem}
\begin{proof}
Part (1) is essentially \cite[Construction 11.5]{bms2}. In fact, by (\ref{E:rthh-mult}), we get
\[
\thh(X/E)\simeq \thh(X)\otimes_{\thh(E)}E
\]
 in the $\infty$-category of $E_\infty$-algebras in spectra. 
Since the forgetful functor from $E_\infty$-algebras in cyclotomic spectra to $E_\infty$-algebras in spectra is symmetric monoidal and preserves small colimits, we may lift $\thh(X/E)$ as the pushout of $$\thh(X)\leftarrow \thh(E)\rightarrow E$$ in the $\infty$-category of $E_\infty$-algebras in cyclotomic spectra. Part (2) follows immediately.
\end{proof}

\begin{defn}
When the condition of Lemma \ref{l3}(1) holds, we set negative cyclic homology
\[
\tc^{-}(-/E)=\thh(-/E)^{h\mathbb{T}}
\] 
and the periodic cyclic homology
\[
\tp(-/E)=(\thh(-/E)^{t\mathbb{T}}.
\]
As in the case $E=\mathbb{S}$, for any prime $l$, the cyclotomic structure on $\thh(-/E)$ induces the Frobenius 
\begin{equation}\label{E:frob}
\varphi_l:\thh(-/E)\simeq \Phi^l\thh(-/E)\to\thh(-/E)^{tC_l}.
\end{equation}
Moreover, there is the canonical map
\begin{equation}\label{E:can}
\can: \tc^{-}(-/E)\simeq (\thh(-/E)^{hC_l})^{h(\mathbb{T}/C_l)}=(\thh(-/E)^{hC_l})^{h\mathbb{T}}\to(\thh(-/E)^{tC_l})^{h\mathbb{T}}.
\end{equation}

The topological cyclic homology is defined by the fiber sequence
\begin{equation*}\label{E:tc}
\tc(-/E)\to\tc^-(-/E)\xrightarrow{\prod_{l}(\varphi_l^{h\mathbb{T}}-\can)} \prod_{l}(\thh(-/E)^{tC_l})^{h\mathbb{T}}.
\end{equation*}
\end{defn}
~\\
Using the argument of \cite[Lemma II 4.2]{NS}, we have
\begin{equation*}\label{E:ns}
\tp(-/E;\mathbb{Z}_p)\simeq (\thh(-/E)^{tC_p})^{h\mathbb{T}}.
\end{equation*}
Taking $p$-completion of $\varphi_l^{h\mathbb{T}}$ and $\mathrm{can}$ for $l=p$, we get
\[
\varphi: \tc^{-}(-/E; \Zp)\to \tp(-/E; \Zp), \quad \can: \tc^{-}(-/E; \Zp)\to \tp(-/E; \Zp)
\]
and the fiber sequence
\[
\tc(-/E;\Zp)\to \tc^{-}(-/E; \Zp)\xrightarrow{\varphi-\can} \tp(-/E; \Zp).
\]
As in the case $E=\mathbb{S}$, there are the homotopy fixed point spectral sequence
\begin{equation}\label{E:homo-fix-spec}
E^2_{i,j}=\thh_*(-/E)[v]\Rightarrow \tc^{-}_{i+j}(-/E)
\end{equation}
and the Tate spectral sequence
\begin{equation}\label{E:tate-spec}
E^2_{i,j}=\thh_*(-/E)[\sigma^{\pm1}]\Rightarrow \tp_{i+j}(-/E),
\end{equation}
where $\thh_j(-/E)$ has degree $(0,j)$, $|v|=(-2, 0), |\sigma|=(2, 0)$, and $\can(v)=\sigma^{-1}$. The \emph{Nygaard filtration} $\mathcal{N}^{\geq j}$ is defined to be the filtration 
$E^{\infty}_{*, j}$ on the abutment; it is multiplicative as the Tate spectral sequence is multiplicative. When the Tate spectral sequence collapses at the $E^2$-term, we denote by $p_j$ the natural projection 
\begin{equation}\label{E:proj}
\mathcal{N}^{\geq j}(\tp_0(-/E))\to \thh_{j}(-/E). 
\end{equation}

Recall that by \cite[Proposition 11.3]{bms2}, $\mathbb{S}[z]$ admits an $E_\infty$-cyclotomic structure over $\thh(\mathbb{S}[z])$ in which the $\mathbb{T}$-action is trivial and the Frobenius sends $z$ to $z^p$.
\begin{prop}\label{p4}The following statments are true. 
\begin{enumerate}
\item[(1)]
There is a functorial $E_\infty$-cyclotomic structure on $\thh(-/\Sk)$.
\item[(2)]
There is a functorial $E_\infty$-cyclotomic structure on $\thh(-/\Sz)$.

\end{enumerate}
\end{prop}
\begin{proof}
We set the Frobenius on $\Sk$ to be the unique $E_\infty$-automorphism inducing the Frobenius on $\pi_0$. It follows that the resulting cyclotomic structure on $\Sk$ agrees with the $p$-completion of the cyclotomic structure on $\thh(\Sk)$ via the augmentation map \cite[IV.1.2]{NS}.
 This yields (1) by Lemma \ref{l3}.

For (2), note that
\[
\Sz \simeq \Sk\otimes_\mathbb{S}\mathbb{S}[z]
\] 
in the $\infty$-category of $E_\infty$-algebras in spectra. We then define the cyclotomic structure on $\Sz$ using the cyclotomic structures on $\Sk$ and $\mathbb{S}[z]$, and the monoidal structure on the $\infty$-category of $E_\infty$-algebras in cyclotomic spectra. We conclude (2) by (1) and Lemma \ref{l3}.
\end{proof}

\begin{conv}
Henceforth we equip $\thh(-/\Sk)$ and $\thh(-/\Sz)$ with the cyclotomic structures given by Proposition \ref{p4}.
\end{conv}

\begin{rem}\label{R:Sk-complete}
Since $\Sk$ is equivalent to the $p$-completion of $\thh(\Sk)$, it follows that 
\[
\thh(\OK/\Sk)\simeq \thh(\OK)\otimes_{\thh(\Sk)}\Sk
\]
is isomorphic to the $p$-completion of $\thh(\OK)$. Similarly, $\thh(\OK/\Sk[z])$ is isomorphic to the $p$-completion of $\thh(\OK/\mathbb{S}[z])$. 
\end{rem}

\begin{rem}\label{R:tc-tp-complete}
By the previous remark, 
we see that $\tc(\OK/\Sk)$, $\tc(\OK/\Sz)$, $\tc^-(\OK/\Sk)$, $\tc^-(\OK/\Sz)$, $\tp(\OK/\Sk)$ and $\tp(\OK/\Sz)$ are isomorphic to $p$-completions of $\tc(\OK)$, $\tc(\OK/\mathbb{S}[z])$, $\tc^-(\OK)$, $\tc^-(\OK/\mathbb{S}[z])$, $\tp(\OK)$ and 
$\tp(\OK/\mathbb{S}[z])$ respectively.
\end{rem}

Note that the composite
\[
\mathbb{S}[z] \rightarrow  \thh(-/\mathbb{S}[z]), 
\]
which is a map of $E_\infty$-algebras in cyclotomic spectra, induces 
\[
i_C: \mathbb{S}[z]^{h\mathbb{T}}\to \tc^-_0(-/\mathbb{S}[z]; \Zp),\quad i_P: (\mathbb{S}[z]^{tC_p})^{h\mathbb{T}}\to \tp_0(-/\mathbb{S}[z];\Zp).
\] 
Recall that $\mathbb{S}[z]$ is equipped with the trivial $\mathbb{T}$-action.
In the following, when the context is clear, we abusively use  $z$ to denote the its images under $i_C$ and $\can\circ i_C$.
\begin{prop}\label{P:frob-can}
We have
$\varphi(z) = z^p.$
\end{prop}
\begin{proof}
Recall that the Frobenius $\varphi_p$ on $\mathbb{S}[z]$ is the composite
\[
\mathbb{S}[z]\xrightarrow{z\mapsto z^p}\mathbb{S}[z]\rightarrow \mathbb{S}[z]^{hC_p}\xrightarrow{\can}\mathbb{S}[z]^{tC_p}.
\]
It follows that the composite 
\[
\varphi_p^{h\mathbb{T}}: \mathbb{S}[z]^{h\mathbb{T}}\xrightarrow{z\mapsto z^p}\mathbb{S}[z]^{h\mathbb{T}}\rightarrow \mathbb{S}[z]^{h\mathbb{T}}\xrightarrow{\can}(\mathbb{S}[z]^{tC_p})^{h\mathbb{T}}
\]
satisfies $\varphi_p^{h\mathbb{T}}(z)=\can(z)^p$. On the other hand, it is straightforward to see 
\[
\varphi\circ i_C=i_P\circ \varphi^{h\mathbb{T}}_p\quad\text{and}\quad \can\circ i_C=i_P\circ \can.
\]
The desired result follows.

\end{proof}

%\section{Relative \tc of $\OK$}

%\begin{nthm}
%There is a functorial $E_\infty$ cyclotomic structure on $THH(-/\Sk)$.
%\end{nthm}

%\begin{nthm}
%There is a functorial $E_\infty$ cyclotomic structure on $THH(-/\Sz)$, such that the Frobenius action 
%$$\varphi: \tc^-_0(-/\Sz) \rightarrow TP_0(-/\Sz)$$
%satisfies
%$$\varphi(z) = can(z)^p$$
%\end{nthm}

Now we specialize to the case of $\OK$. Firstly recall the B\"{o}kstedt periodicity 
$$\thh_*(\k) = \k[u],$$
where $u\in \thh_2(\k)$ is the B\"{o}kstedt element. Recall the following result from \cite{NS}. 

\begin{thm}\label{bokperiodicity}
\begin{enumerate}

\item[(1)]
Both the Tate spectral sequence for $\tp_*(\mathbb{F}_p)$ and the homotopy fixed point spectral sequence for $\tc^-_*(\mathbb{F}_p)$ collapse at the $E^2$-term.
Consequently, 
\[
\tp_*(\mathbb{F}_p) = \tp_0(\mathbb{F}_p)[\sigma_{\mathbb{F}_p}^{\pm1}], \qquad |\sigma_{\mathbb{F}_p}|=2,
\]
and the canonical map  $\mathrm{can}: \tc^-_j(\mathbb{F}_p) \rightarrow \tp_j(\mathbb{F}_p)$
is an isomorphism for $j\leq0$.

\item[(2)] 
We have $\tp_0(\mathbb{F}_p) = \mathbb{Z}_p$
and $p_0: \tp_0(\mathbb{F}_p)\to \thh_0(\mathbb{F}_p)$ is the natural projection
$\mathbb{Z}_p\to\mathbb{F}_p$. Moreover, the cyclotomic Frobenius on $\tc^-_0(\mathbb{F}_p)$ is the identity map under the isomorphism $\mathrm{can}: \tc^-_0(\mathbb{F}_p) \cong \tp_0(\mathbb{F}_p)$. 

\item[(3)]For any $u_{\mathbb{F}_p}\in \tc^-_2(\mathbb{F}_p)$ lifting the B\"{o}kstedt element under $p_0$, there exists a unique $v_{\mathbb{F}_p}\in \tc^-_{-2}(\mathbb{F}_p)$ such that
\[
\tc^-_*(\mathbb{F}_p) = \tc^{-}_0(\mathbb{F}_p)[u_{\mathbb{F}_p},v_{\mathbb{F}_p}]/(u_{\mathbb{F}_p}v_{\mathbb{F}_p}-p).
\]

\item[(4)]For any $(u_{\mathbb{F}_p}, v_{\mathbb{F}_p})$ in (3),  
\[
\varphi(u_{\mathbb{F}_p})=\can(v_{\mathbb{F}_p})^{-1}.
\] 

\end{enumerate}
\end{thm}
\begin{proof}
See Corollary IV.4.16 \cite{NS} and the discussion after that.
\end{proof}

\begin{defn}\label{EK}
Let $\EK(z)$ be the minimal polynomial of $\pK$ over $K_0$, normalized such that  $E_K(0)=p$.
\end{defn}

\begin{thm}\label{T:rtp}
\begin{enumerate}
\item[(1)]We have 
\[
\thh_*(\OK/\Sz) = \OK[u],
\]
where $u\in \thh_2(\OK/\Sz)$ is any lift of the B\"{o}kstedt element in $\thh_2(\k)$.
\item[(2)] The Tate spectral sequence for $\tp_*(\OK/\Sz)$ collapses at the $E^2$-term. Consequently, 
\[
\tp_*(\OK/\Sz) = \tp_0(\OK/\Sz)[\sigma^{\pm1}]
\]
with $|\sigma|=2$.

\item[(3)] 
We have 
\[
\tp_0(\OK/\Sz) = W(\k)[[z]],
\]
and $p_0: \tp_0(\OK/\Sz)\to \thh_0(\OK/\Sz)$ is the $W(k)$-algebra morphism
\[
W(\k)[[z]]\xrightarrow{z\mapsto \pK} \OK.
\]

\item[(4)] The homotopy fixed point spectral sequence for $\tc^-_*(\OK/\Sz)$ collapses at the $E^2$-term. Consequently, the canonical map  
\[
\mathrm{can}: \tc^-_j(\OK/\Sz) \rightarrow \tp_j(\OK/\Sz)
\]
induces an isomorphism $\tc^-_j(\OK/\Sz) \cong \mathcal{N}^{\geq j}\tp_j(\OK/\Sz)$
for all $j\in \mathbb{Z}$. In particular, 
\[
\mathrm{can}: \tc^-_j(\OK/\Sz) \rightarrow \tp_j(\OK/\Sz)
\]
is an isomorphism for $j\leq0$. 
\item[(5)] %{\color{red}(wish to see a detailed argument)}
Under (3) and (4), the cyclotomic Frobenius 
\[
\varphi:\tc^-_0(\OK/\Sz)\rightarrow \tp_0(\OK/\Sz)
\]
is the map
$W(\k)[[z]]\rightarrow W(\k)[[z]]$
which is the Frobenius on $W(\k)$ and sends $z$ to $z^p$.

\item[(6)] For any $u_{\mathbb{F}_p}$ given in Theorem \ref{bokperiodicity}, there exist unique $u\in \tc^-_2(\OK/\Sz), v\in \tc^-_{-2}(\OK/\Sz)$ and $\sigma\in \tp_2(\OK/\Sz)$
such that $u$ lifts $u_{\mathbb{F}_p}$, $\varphi(u) = \sigma$, $\can(v) = \sigma^{-1}$,  
$\tp_*(\OK/\Sz) = \tp_0(\OK/\Sz)[\sigma^{\pm1}]$ and
\[
\tc^-_*(\OK/\Sz) = \tc^{-}_0(\OK/\Sz)[u,v]/(uv-\EK(z)).
\]
As a consequence, under (3), the Nygaard filtration on $\tp_0(\OK/\Sz)$ is given by 
\begin{equation}\label{E:nyg-fil}
\mathcal{N}^{\geq 2j}\tp_0(\OK/\Sz)=\mathcal{N}^{\geq 2j-1}\tp_0(\OK/\Sz)=(\EK(z))^j,\quad  j\geq0. 
\end{equation}
%Moreover, we can make the constant term of $E_K(z)$ the same for all $K$. 
\end{enumerate}
\end{thm}
\begin{proof}
By Remark \ref{R:tc-tp-complete}, we see that all the statements except (6) follow immediately from \cite[Proposition 11.10]{bms2}. In fact,  the argument given in \emph{loc. cit.} is enough to show the following statement:
\begin{enumerate}
\item[(6)'] For any $\tilde{u}\in \tc^-_{2}(\OK/\Sz)$ lifting the $u$ given in (1), there exist $\tilde{v}\in \tc^-_{-2}(\OK/\Sz)$ and $\tilde{\sigma}\in \tp_2(\OK/\Sz)$
such that 
\[
\can(\tilde{v})=\tilde{\sigma}^{-1}, \quad \varphi(\tilde{u}) = \tilde{\sigma},  \quad \tp_*(\OK/\Sz) = \tp_0(\OK/\Sz)[\tilde{\sigma}^{\pm1}]
\]
and 
\[
\tc^{-}_*(\OK/\Sz) = \tp_0(\OK/\Sz)[\tilde{u}, \tilde{v}]/(\tilde{u}\tilde{v}- \lambda(z)\EK(z))
\]
for some $\lambda(z)\in W(k)[[z]]^\times$.
%where $\EK(z)$ is a minimal polynomial of $\varpi_K$ over $K_0$.
\end{enumerate}
In the following, we give a proof of (6) based on (6)'. 
%Firstly, by \cite{NS}, there exists some nonzero $u_\Fp\in \tc^-_2(\Fp)$ such that
%\begin{equation*}
%\can(u_\Fp)=p\tau\varphi(u_\Fp)
%\end{equation*}
%for some $\tau\in\Zp^\times$. Let $u_\k$ be the image of $u_\Fp$ along $\Fp\rightarrow \k$.
By Lemma \ref{l3}, the commutative diagram
$$
\xymatrix{
\Sz \ar[r]^{z\mapsto 0} \ar[d]_{z\mapsto \pK} & \Sk\ar[d] \\
\OK \ar[r]& \k
}$$
induces a map of $E_\infty$-algebras in cyclotomic spectra $\thh(\OK/\Sz)\rightarrow \thh(\k)$. By (3) and (4),  the induced map 
\[
\tc^-_0(\OK/\Sz)\rightarrow \tc^-_0(\k)
\]
is the $W(k)$-algebra morphism $W(\k)[[z]]\xrightarrow{z\mapsto 0} W(\k)$, which is surjective. Moreover, by (6)', $\tc^-_2(\OK/\Sz)$ is free of rank 1 over $\tc^-_0(\OK/\Sz)$. Hence 
\[
\tc^-_2(\OK/\Sz)\rightarrow \tc^-_2(\k)
\]
is surjective as well. 

Now take a lift $\tilde{u}\in \tc^-_2(\OK/\Sz)$ of the image of $u_{\mathbb{F}_p}$, which is given by Theorem
\ref{bokperiodicity}, in $\tc^-_2(\k)$.
Using (6)', we have $\tilde{v}, \tilde{\sigma}$ such that
\[
\can(\tilde{v})=\tilde{\sigma}^{-1}, \quad \varphi(\tilde{u}) = \tilde{\sigma},  \quad \tp_*(\OK/\Sz) =\tp_0(\OK/\Sz)[\tilde{\sigma}^{\pm1}]
\]
and 
\[
\tc^{-}_*(\OK/\Sz) = \tp_0(\OK/\Sz)[\tilde{u}, \tilde{v}]/(\tilde{u}\tilde{v}- \lambda(z)\EK(z))
\]
for some $\lambda(z)\in W(k)[[z]]^\times$. Now by the constructions of $E_K(z)$ and $\tilde{u}$, and applying Theorem \ref{bokperiodicity}, we deduce that % the constant term of $\EK(z)$ is equal to $p\tau$, yielding the desired result.
$\lambda(z)$ has constant term $1$. It follows that there exists $b(z)\in W(k)[[z]]$ with constant term $1$ such that $\lambda(z) = {\varphi(b(z))\over b(z)}$. Then we set $u = b(z)u'$, $\sigma=\varphi(b(z))\sigma'$, $v=\varphi(b(z))^{-1}v'$. 

For the uniqueness, suppose $(u', v', \sigma')$ is another choice, it follows that there exists $\lambda(z)\in W(k)[[z]]^\times$ such that $u'=\lambda(z)u, v'=\lambda(z)^{-1}v, \sigma'=\lambda(z)\sigma$. Thus the conditions that $u'$ lifts $u_{\mathbb{F}_p}$ and $\varphi(u')=\sigma'$ imply that $\lambda(0)=1$ and $\varphi(\lambda)=\lambda$, yielding $\lambda=1$. The rest of (6) follows immediately.

\end{proof}

\begin{conv}\label{R:choice}
Henceforth we fix $u, v, \sigma$ as in Theorem \ref{T:rtp}(6). 
%Using similar argument in \emph{loc. cit.}, it is straightforward to see that if $u', v', \sigma'$ is another choice, then there exists $\lambda\in\mathbb{Z}^\times_p$ such that $u'=\lambda u, v'=\lambda^{-1}v, \sigma'=\lambda \sigma$. 

\end{conv}

%\begin{cor}
%$$\varphi(v) = \sigma^{-1}\varphi(\EK(x))$$
%\end{cor}

\section{Structure of $\tp_0(\OK/\Szz)$} \label{s3}
This section is devoted to determining the structure of $\tp_0(\OK/\Szz)$. Here we regard $\OK$ as an $\Szz$-algebra via the map $\Szz\xrightarrow{z_0, z_1\mapsto \pK} \OK$.

\begin{prop}
The topological Hochschild homology $\thh(\OK/\Szz)$ has a natural structure of $E_\infty$-algebra in cyclotomic spectra.
\end{prop}
\begin{proof}
Recall that in \cite{NS} the $\infty$-category of cyclotomic spectra is promoted to a symmetric monoidal $\infty$-category.
By the multiplicative property of $\thh$ (\ref{E:rthh-mult}), we have
\[
\thh(\OK/\Szz) \simeq \thh(\OK/\mathbb{S}_{W(\k)}[z_0])\otimes_{\thh(\OK/\Sk)}\thh(\OK/\mathbb{S}_{W(\k)}[z_1]).
\]
The cyclotomic structures on $\thh(\OK/\mathbb{S}_{W(\k)}[z_i]), i=0, 1$, and the symmetric monoidal structure on the $\infty$-category of $E_\infty$-algebras in cyclotomic spectra give rise to the $E_\infty$-cyclotomic structure on $\thh(\OK/\Szz)$. 
\end{proof}

For $\heartsuit\in\{\thh, \tc, \tc^-, \tp\}$, the left unit $\eta_L$ and right unit $\eta_R$ are the maps 
\[
\heartsuit(\OK/\Sz)\to \heartsuit(\OK/\Szz)
\] 
induced by $z\mapsto z_0$ and $z\mapsto z_1$ respectively. 
For $?\in\{z, u, v, \sigma\}$, we denote by $?_0$ and $?_1$ the images of $?$ under the left and right units respectively. In the following, we regard $\thh_*(\OK/\Szz)$ as an $\OK[u_0]$-module via $\eta_L$. 

In the following, for a commutative ring $R$ and an ideal $I\subset R$, denote by $D_R(I)$ the divided power envelope of $I$ in $R$.  We equip it with the \emph{Nygaard filtration} $\mathcal{N}^{\geq \bullet}$ where $\mathcal{N}^{\geq 2j}D_R(I)=\mathcal{N}^{\geq 2j-1}D_R(I)$ is the $R$-submodule generated by $I^{[l]}$ for all $l\geq j$. For an $R$-module $M$, denote by $\Gamma_R(M)$ the divided power envelope of $\mathrm{Sym}_R(M)$ with respect to the ideal generated by $M\subset \mathrm{Sym}_R(M)$. Recall that if $M$ is a free $R$-module with a basis  $x_1,\dots, x_n$,  then $\Gamma_R(M)$ is just the divided power algebra $R\langle x_1,...,x_n\rangle$. 

To proceed, we need a variant of the classical Hochschild-Kostant-Rosenberg theorem, whose proof is given in the appendix. 
\begin{thm}\label{T:HKR}(Theorem \ref{T:AHKR})
Let $R$ be a commutative ring over $\mathbb{Z}_p$, and let $I$ be a locally complete intersection ideal of $R$. Let $A=R/I$. Suppose $R$ is $I$-separated and $A$ is $p$-torsion free. Then  as filtered rings, the periodic cyclic homology $\mathrm{HP}_0(A/R)$ is canonically isomorphic to the completion of $D_R(I)$ with respect to the Nygaard filtration. Moreover, the Tate spectral sequence for $\mathrm{HP}_0(A/R)$ collapses at the $E^2$-term. Consequently,  
there is a canonical isomorphism of graded rings 
\[
\mathrm{HH}_*(A/R) \cong \Gamma_A( I/I^2).
\]
%under the canonical isomorphism $I/I^2\cong \mathrm{HH}_2(A/R)$. 
\end{thm}

\begin{rem}
One may obtain a similar result by combing the motivic filtration for $\mathrm{HP}$ developed in \cite[\S5]{bms2} (see also \cite{An}) and the computation about the derived de Rham cohomology for lci maps (\cite[Theorem 3.27]{Bha}).
\end{rem}

Let $I$ be the kernel of the $W(k)$-algebra map $W(\k)[z_0,z_1]\xrightarrow{z_0, z_1\mapsto \varpi}\OK$, and let $t_{z_0-z_1}$ denote the image of $z_0-z_1$ in $\thh_2(\OK/\Szz)$ under the isomorphism
\[
I/I^2 \cong \mathrm{HH}_2(\OK/W(\k)[z_0,z_1])=\thh_2(\OK/\Szz).
\]
given by Theorem \ref{T:HKR}. %In the following, for any ring $R$, denote by $R\langle t\rangle$ the divided algebra over $R$.
\begin{lem}\label{L:grad-thhzz}
The graded algebra associated to the filtration on $\thh_*(\OK/\Szz)$ defined by powers of $u_0$ is isomorphic to 
$\OK[u_0]\otimes_{\OK}\OK\langle t_{z_0-z_1}\rangle$.
\end{lem}
\begin{proof}
By Theorem \ref{T:rtp}(1), we have
\begin{eqnarray}\label{E:mod-u}
\begin{split}
\thh(\OK/\Szz)/(u_0)&\simeq \thh(\OK/\Szz)\otimes_{\thh(\OK/\mathbb{S}_{W(k)}[z_0])}\thh(\OK/\OK)\\
&\simeq \thh(\OK/\OK[z_1]). 
\end{split}
\end{eqnarray}
Since $\thh_*(\OK/\OK[z]) =\mathrm{HH}_*(\OK/\OK[z]) =  \OK\langle t \rangle$ for any generator $t\in \mathrm{HH}_2(\OK/\OK[z])$, 
we deduce that the $u_0$-Bockstein spectral sequence collapses since everything is concentrated in even degrees. Hence the associated graded algebra is isomorphic to 
$\OK[u_0]\otimes_{\OK}\OK\langle t\rangle$. Note that under the isomorphism (\ref{E:mod-u}), $t_{z_0-z_1}$ maps to a generator of $\thh_2(\OK/\OK[z])$. This yields the desired result. 
\end{proof}

The following result follows immediately. 

\begin{cor}\label{C:grad-thhzz}
The graded ring $\thh_*(\OK/\Szz)$ is a $p$-torsion free integral domain. 
\end{cor}

\begin{cor}\label{C:degen-zz}
Both the Tate spectral sequence for $\tp_*(\OK/\Szz)$ and the homotopy fixed point spectral sequence for $\tc^{-}_*(\OK/\Szz)$ collapse at the $E^2$-term. Consequently, 
both $\tc^{-}_*(\OK/\Szz)$ and $\tp_*(\OK/\Szz)$ are concentrated in even degrees, and the canonical map 
\[
\can:  \tc^{-}_{j}(\OK/\Szz)\rightarrow\tp_j(\OK/\Szz)
\] 
induces an isomorphism $\tc^{-}_{j}(\OK/\Szz)\cong \mathcal{N}^{\geq j}\tp_j(\OK/\Szz)$ for all $j\in\mathbb{Z}$. In particular, 
\[
\can: \tc^{-}_{j}(\OK/\Szz)\to \tp_{j}(\OK/\Szz)
\] 
is an isomorphism for $j\leq0$.
\end{cor}
\begin{proof}
By Lemma \ref{L:grad-thhzz}, $\thh_*(\OK/\Szz)$ is concentrated in even degrees. It follows that both the Tate spectral sequence and the homotopy fixed point spectral sequence degenerate at the $E^2$-term. The rest of the corollary follows immediately. 
\end{proof}

\begin{rem}\label{R:tate-n-variable}
In general, for $n\geq0$, we may regard $\OK$ as an $\mathbb{S}_{W(\k)}[z_0, \dots, z_n]$-module by sending all $z_i$ to $\pK$. Using the argument of Lemma \ref{L:grad-thhzz} inductively, one easily shows that  $\thh(\OK/\mathbb{S}_{W(\k)}[z_0, \dots, z_n])$ is concentrated in even degrees.  Consequently, Corollary \ref{C:degen-zz} generalizes to this case. 
\end{rem}

\begin{lem}\label{L:grad-tpzz}
The graded algebra associated to the Nygaard filtration of $\tp_0(\OK/\Szz)$ is isomorphic to $\thh_*(\OK/\Szz)$. 
\end{lem}
\begin{proof}
This follows from Corollary \ref{C:degen-zz}.
\end{proof}
\begin{rem}
The isomorphism given by Lemma \ref{L:grad-tpzz} is not canonical: it depends on the choice of the generator $\sigma$. By Theorem \ref{T:rtp} and Theorem \ref{bokperiodicity}, our choice of $\sigma$ is fixed up to a unit in $\mathbb{Z}_p$ (Convention \ref{R:choice}).
\end{rem}

The following two results follow immediately from the fact that the Nygaard filtration on $\tp_0(\OK/\Szz)$ is separated  .

\begin{cor}\label{C:tpzz-nyg}
For $a\in \tp_0(\OK/\Szz)$, it has Nygaard filtration $j$ if and only if $pa$ has Nygaard filtration $j$.
\end{cor}

\begin{cor}\label{C:tpzz-integral}
The ring $\tp_0(\OK/\Szz)$ is a $p$-torsion free integral domain.
\end{cor}

Henceforth we identify $\tc_0^{-}(\OK/\Szz)$ with $\tp_0(\OK/\Szz)$ via the canonical map, and regard the cyclotomic Frobenius $\varphi$ as a map on $\tp_0(\OK/\Szz)$.  By Proposition \ref{P:frob-can}, we first have
\begin{equation}\label{E:frob-can-lr}
\varphi(z_0)=z_0^p, \qquad
\varphi(z_1)=z_1^p.
\end{equation}

\begin{lem}\label{L:nyg-frob}
If $a\in \mathcal{N}^{\geq 2j}\tp_0(\OK/\Szz)$, then $\varphi(a)$ is divisible by $\varphi(\EK(z_0))^j$.
\end{lem}
\begin{proof}
By Corollary \ref{C:degen-zz}, the Tate spectral sequence for $\tp_*(\OK/\Szz)$ collapses at the $E^2$-term. We may write $a=\sigma_0^{-j}a_0$ for some 
\[
a_0\in \mathcal{N}^{\geq 2j}\tp_{2j}(\OK/\Szz)=\tc^-_{2j}(\OK/\Szz).
\]
 Hence by Theorem \ref{T:rtp},
\[
\varphi(a) = \varphi(\sigma_0^{-1})^{j}\varphi(a_0) =\varphi(v_0)^j\varphi(u_0)^j\sigma_0^{-j}\varphi(a_0)=
\varphi(\EK(z_0))^j\sigma_0^{-j}\varphi(a_0),
\]
yielding the desired result.
\end{proof}

\begin{rem}\label{R:phi-z1z2}
By Theorem \ref{T:rtp}, $\EK(z)$ has Nygaard filtration $2$ in $\tp_0(\OK/\Sz)$. Hence $\EK(z_i)$ has Nygaard filtration $2$ in $\tp_0(\OK/\Szz)$. By Lemma \ref{L:nyg-frob}, $\varphi(\EK(z_i))$ is divisible by $\varphi(\EK(z_{1-i}))$ for $i=0,1$. Thus $\varphi(\EK(z_0))\varphi(\EK(z_1))^{-1}$ is a unit in  $\tp_0(\OK/\Szz)$.
\end{rem}

\begin{defn}\label{D:n-pn}
For a ring $R$ equipped  with a multiplicative decreasing filtration $\mathcal{N}^{\geq\bullet}$, we call the topology on $R$ defined by the filtration $\mathcal{N}^{\geq\bullet}$ the \emph{$\mathcal{N}$}-topology. We define the \emph{$(p,\mathcal{N})$-topology} on $R$ to be the topology in which $\{(p^j)+\mathcal{N}^{\geq j}\}_{j\geq0}$ forms a basis of open neighborhoods of $0$. 
\end{defn}

Note that in general neither the $\mathcal{N}$ nor the $(p,\mathcal{N})$-topology is adic topology.
Clearly $R$ becomes a topological ring under either the $\mathcal{N}$ or the $(p,\mathcal{N})$-topology. 

\begin{rem}
By Theorem \ref{T:rtp}, it is straightforward to see that $\tp_0(\OK/\Sz)$ is complete and separated under either the $\mathcal{N}$ or the $(p, \mathcal{N})$-topology. Moreover, the cyclotomic Frobenius is continuous with respect to the $(p,\mathcal{N})$-topology, but not the $\mathcal{N}$-topology.
\end{rem}

\begin{lem}\label{L:top-complete}
Both the $\mathcal{N}$ and $(p,\mathcal{N})$-topology on $\tp_0(\OK/\Szz)$ are complete and separated.
\end{lem}
\begin{proof}
The assertion for the $\mathcal{N}$-topology follows from the isomorphism 
\[
\tp_0(\OK/\Szz)\cong \tc^{-}_0(\OK/\Szz)
\]
given by Corollary \ref{C:degen-zz} and the fact that $\tc^{-}_*(\OK/\Szz)$ are all complete with respect to the $\mathcal{N}$-topology. For the $(p, \mathcal{N})$-topology, we first note that by Lemma \ref{L:grad-thhzz},  $\thh_*(\OK/\Szz)$ are all $p$-complete. By degeneration of the Tate spectral sequence, this implies that for each $j\geq0$, 
\[
\tp_0(\OK/\Szz)/\mathcal{N}^{\geq j}\tp_0(\OK/\Szz)
\]  
is $p$-complete and separated. Hence
 the $(p,\mathcal{N})$-completeness (resp. separateness) follows from the $\mathcal{N}$-completeness (resp, separateness). 
\end{proof}

\begin{lem}\label{L:frob-con}
The cyclotomic Frobenius on $\tp_0(\OK/\Szz)$ is continuous with respect to the $(p, \mathcal{N})$-topology.
\end{lem}
\begin{proof}

By Lemma \ref{L:nyg-frob}, we have $\varphi((p^{2j})+\mathcal{N}^{\geq 2j})\subset (p^{2j})+\mathcal{N}^{\geq 2j}$. The desired result follows. 
\end{proof}

In the rest of this section, we give an explicit description of $\tp_0(\OK/\Szz)$.
To this end, we make use of the theory of $\delta$-rings. 
\begin{defn}
A $\delta$-ring is a pair $(R,\delta)$ where $R$ is a commutative ring and $\delta: R\to R$ is a map of sets with $\delta(0) = \delta(1) = 0$, satisfying the following two identities
\begin{equation*}\label{E:delta}
\delta(xy)=x \delta(y)+y \delta(x)+p\delta(x)\delta(y), \quad \delta(x+y)=\delta(x)+\delta(y)+ \frac{x^p+y^p-(x+y)^p}{p}.
\end{equation*}
\end{defn}
When $R$ is $p$-torsion free, a $\delta$-ring structure on $R$ is equivalent to the datum of a ring map $\varphi: R\to R$ lifting the Frobenius on $R/p$;
the corresponding $\delta$-structure is given by 
\[
\delta(x) = \frac{\varphi(x) - x^p}{p}.
\]
By Theorem \ref{T:rtp}, it is clear that the cyclotomic Frobenius on $\tp_0(\OK/\Sz)$ gives rise to a 
$\delta$-ring structure. Moreover, the congruence 
\[
\varphi(E_K(z)^{p^j})\equiv E_K(z)^{p^{j+1}} \mod p^{j+1}
\] 
implies that $\delta$ is continuous with respect to the $(p,\mathcal{N})$-topology. In the following, we will show that the same properties hold for $\tp_0(\OK/\Szz)$ as well. 

Using Theorem \ref{T:rtp}, we deduce that 
\[
p_0: \tp_0(\OK/\Szz)\rightarrow \thh_0(\OK/\Szz)=\OK
\]
sends $z_i$ to $\pK$. It follows that  $z_0-z_1$ lies in
\[
  \ker(p_0)=\mathcal{N}^{\geq2}\tp_0(\OK/\Szz).
\]
By Lemma \ref{L:nyg-frob}, there exists $\ef\in \tp_0(\OK/\Szz)$ such that
\[
\ef\varphi(\EK(z_0))=\varphi(z_0-z_1)=z_0^p-z_1^p.
\]
For $k\geq0$, we inductively define $\delta^k(h), f^{(k)}\in \tp_0(\OK/\Szz)[1/p]$ by setting 
\begin{equation}
\delta^0(h)=h, \quad \delta^{k+1}(h)=\frac{\varphi(\delta^k(h))-\delta^k(h)^p}{p}  
\end{equation}
and
\begin{equation}\label{E:induc-odd}
f^{(0)} =z_0-z_1, \quad f^{(k+1)} = {-(f^{(k)})^p + \delta^{k}(\ef)\EK(z_0)^{p^{k+1}}\over p}.
\end{equation}

\begin{prop}\label{P:f(k)}
For each $k\geq0$, we have $\delta^k(h)\in \tp_0(\OK/\Szz)$, $f^{(k)}\in  \mathcal{N}^{\geq 2p^{k}}\tp_0(\OK/\Szz)$ and
\begin{equation}\label{E:phi-delta}
\delta^{k}(\ef)\varphi(\EK(z_0))^{p^k}=\varphi(f^{(k)})
\end{equation}
\end{prop}
\begin{proof}
We will proceed by induction on $k$ to show that $\delta^k(h)\in \tp_0(\OK/\Szz)$, 
\begin{equation}\label{E:h-f}
f^{(k)}\in  W(\k)[z_0,z_1][h,\dots,\delta^{k-1}(h)]\cap \mathcal{N}^{\geq 2p^{k}}\tp_0(\OK/\Szz)
\end{equation}
and
\[
\delta^{k}(\ef)\varphi(\EK(z_0))^{p^k}=\varphi(f^{(k)}).
\]
The initial case is obvious. Now suppose for some $l\geq0$, the claim holds for  $k=l$.  Using (\ref{E:phi-delta}) for $k=l$, we get 
\begin{equation}\label{E:l-l+1}
\begin{split}
f^{(l+1)} &= {-(f^{(l)})^p + \delta^{l}(\ef)\EK(z_0)^{p^{l+1}}\over p}\\
&= {(\varphi(f^{(l)})- (f^{(l)})^p)+ (\delta^{l}(\ef)\EK(z_0)^{p^{l+1}}- \delta^{l}(\ef)\varphi(\EK(z_0))^{p^{l}})\over p}\\
&=\delta(f^{(l)})-\delta^l(h)\delta(\EK(z_0)^{p^l}).
\end{split}
\end{equation}
Since $\tp_0(\OK/\Sz)$ is a $\delta$-ring,  we have $\delta(\EK(z_0)^{p^l})\in \tp_0(\OK/\Szz)$. By the inductive hypothesis, we conclude $f^{(l+1)}\in  W(\k)[z_0,z_1][h,\dots,\delta^{l}(h)]$. Moreover, by the inductive hypothesis and Remark \ref{R:phi-z1z2}, $pf^{(l+1)}$ has Nygaard filtration $\geq 2p^{l+1}$. Hence $f^{(l+1)}$ has Nygaard filtration $\geq 2p^{l+1}$ by Corollary \ref{C:tpzz-nyg}.

To show (\ref{E:phi-delta}) for $k=l+1$, applying $\varphi$ to (\ref{E:induc-odd}) for $k=l$ and using the inductive hypothesis, we get
\begin{eqnarray*}
\varphi(f^{(l+1)}) &=& {-\varphi(f^{(l)})^p + \varphi(\delta^{l}(h))\varphi(\EK(z_0))^{p^{l+1}}\over p}\\
&=& {-\delta^{l}(h)^p\varphi(\EK(z_0))^{p^{l+1}} + \varphi(\delta^{l}(h))\varphi(\EK(z_0))^{p^{l+1}}\over p}\\
&=& \delta^{l+1}(h))\varphi(\EK(z_0))^{p^{l+1}}.
\end{eqnarray*}
Finally, using (\ref{E:phi-delta}) for $k=l+1$ and Lemma \ref{L:nyg-frob}, we deduce that $\delta^{l+1}(h)\in \tp_0(\OK/\Szz)$. This completes the proof. 
\end{proof}

\begin{lem}\label{L:dense}
The sub-$W(\k)[z_0,z_1]$-algebra $R\subset \tp_0(\OK/\Szz)$ generated by the family of elements $\{f^{(k)}| k\geq0 \}$ is dense with respect to the $\mathcal{N}$-topology.
\end{lem}
\begin{proof}
It suffices to show that for every $j\geq0$, the projection 
\begin{equation*}\label{E:surj}
p_{2j}: R\cap\mathcal{N}^{\geq 2j}\tp_0(\OK/\Szz)\rightarrow \thh_{2j}(\OK/\Szz),
\end{equation*}
which is induced by (\ref{E:proj}), is surjective. 

Firstly, by Theorem \ref{T:rtp}, we see that $p_2(E_K(z))=u$, yielding
\[
p_2(E_K(z_0)^j)=u^j_0
\] 
by functoriality of the Tate spectral sequence. To conclude, by Lemma \ref{L:grad-thhzz}, it suffices to show that $p_{2j}(R)$ contains $t^{[j]}_{z_0-z_1}$ for all $j\geq0$. 

By the commutative diagram
$$\xymatrix{
\mathcal{N}^{\geq 2}\tp_0(\OK/\Szz) \ar[r]\ar[d] & \thh_{2}(\OK/\Szz) \ar[d]^\cong &  \\
\mathcal{N}^{\geq 2}\mathrm{HP}_0(\OK/W(\k)[z_0,z_1])\ar[r] & \mathrm{HH}_{2}(\OK/W(\k)[z_0,z_1]) & 
}$$
one immediately checks that $f^{(0)}$ and $t_{z_0-z_1}$ have the same image in $\mathrm{HH}_{2}(\OK/W(\k)[z_0,z_1])$. Hence $p_2(f^{(0)})=(t_{z_0-z_1})$. For $k\geq0$, we have 
\[
-(f^{(k)})^p \equiv pf^{(k+1)} \mod \EK(z_0)^{p^{k+1}}
\]
by (\ref{E:induc-odd}). By induction, we deduce that for all $k\geq0$, $t^{[p^k]}_{z_0-z_1}$ lies in the image of $R\cap\mathcal{N}^{\geq 2p^k}\tp_0(\OK/\Szz)$. Note that in the divided power algebra $\OK\langle t\rangle$, for $j=j_0+pj_1+\cdots+j_kp^{k}$ with $0\leq j_i\leq p-1$, 
$t^{[j]}$ is equal to $t^{j_0}(t^{[p]})^{j_1}\cdots (t^{[p^k]})^{j_k}$ up to a unit of $\mathbb{Z}_p$.  It follows that  the image of $R\cap\mathcal{N}^{\geq 2j}\tp_0(\OK/\Szz)$ contains $t^{[j]}_{z_0-z_1}$ for all $j\geq0$. 
\end{proof}

\begin{rem}
In Corollary \ref{C:u-ext}, we will prove that $p_{2p^k}(f^{(k)})$ is equal to $t^{[p^k]}_{z_0-z_1}$ up to a unit of $\mathbb{Z}_p$. 
\end{rem}

\begin{prop}\label{P:tp-delta}
%We have that $\tp_0(\OK/\Szz)$ is a $\delta$-ring. 
The cyclotomic Frobenius on $\tp_0(\OK/\Szz)$ is a Frobenius lift, making $\tp_0(\OK/\Szz)$ a $\delta$-ring.
Moreover, $\delta$ is continuous with respect to the $(p, \mathcal{N})$-topology on $\tp_0(\OK/\Szz)$.
\end{prop}
\begin{proof}
We equip $\tp_0(\OK/\Szz)$ with the $(p, \mathcal{N})$-topology. Put $\phi(a)=\varphi(a)-a^p$ for $a\in \tp_0(\OK/\Szz)$. By Lemma \ref{L:frob-con},  $\phi$ is continuous. On the other hand, by Corollary \ref{C:tpzz-nyg}, it is straightforward to see that the map
\[
\tp_0(\OK/\Szz)\xrightarrow{a\mapsto pa} \tp_0(\OK/\Szz)
\]  
induces an embedding. Using Lemma \ref{L:top-complete}, we deduce that $(p)\subset \tp_0(\OK/\Szz)$ is a closed ideal. It remains to show that $\mathrm{Im}(\phi)\subset (p)$. Put $R'=W(k)[\{\delta^k(h)\}_{k\geq0}]$. Then $R'$ is stable under $\delta$. That is, $\phi(R')\subset (p)$.  Note that $R'$ is dense by (\ref{E:h-f}) and Lemma \ref{L:dense}. We therefore conclude that $\phi^{-1}((p))$, which is a closed subset of $\tp_0(\OK/\Szz)$, is forced to be $\tp_0(\OK/\Szz)$. 
\end{proof}

To describe the structure of $\tp_0(\OK/\Szz)$, we compare it with the periodic cyclic homology $\mathrm{HP}_0(\OK/W(\k)[z_0,z_1])$.  Let $I$ be the kernel of $W(\k)[z]\xrightarrow{z\mapsto \pK}\OK$ (resp. $W(\k)[z_0,z_1]\xrightarrow{z_0, z_1\mapsto \pK}\OK$), and let $t_{\EK(z)}$ (resp. $t_{\EK(z_i)}$) denote the image of $\EK(z)$ (resp. $\EK(z_i)$) in $\mathrm{HH}_2(\OK/\Sz)$ under the isomorphism 
\[
I/I^2 \cong \mathrm{HH}_2(\OK/W(\k)[z])\quad (\text{resp. $I/I^2 \cong \mathrm{HH}_2(\OK/W(\k)[z_0,z_1])$}) 
\]
given by Theorem \ref{T:HKR}.

The following result follows from Theorem \ref{T:HKR} immediately.
\begin{cor}\label{C:hh-hp}
The following statements are true.
\begin{enumerate}
\item[(1)]
As filtered rings, $\mathrm{HP}_0(\OK/W(\k)[z])$ and $\mathrm{HP}_0(\OK/W(\k)[z_0,z_1])$ are canonically isomorphic to $D_{W(\k)[z]}((\EK(z)))^\wedge_\mathcal{N}$ and $D_{W(\k)[z_0,z_1]}((\EK(z_0), z_0-z_1))^\wedge_\mathcal{N}$ respectively.

\item[(2)]As graded rings,  $\mathrm{HH}_*(\OK/W(\k)[z])=\OK\langle t_{\EK(z)}\rangle$ and 
\[
\mathrm{HH}_*(\OK/W(\k)[z_0,z_1]) =\OK\langle t_{\EK(z_0)}, t_{z_0-z_1} \rangle.
\]

\end{enumerate}
\end{cor}

Using Definition \ref{D:n-pn}, we consider the $\mathcal{N}$ and $(p,\mathcal{N})$-topologies for $\mathrm{HP}_0(\OK/W(\k)[z])$ and $\mathrm{HP}_0(\OK/W(\k)[z_0,z_1])$. 

 \begin{lem}\label{L:hp-complete}
Both $\mathrm{HP}_0(\OK/W(\k)[z])$ and $\mathrm{HP}_0(\OK/W(\k)[z_0, z_1])$ are complete and separated with respect to the $\mathcal{N}$ and $(p,\mathcal{N})$-topologies.
\end{lem}
\begin{proof}
For the $\mathcal{N}$-topology, it is an immediate consequence of Corollary \ref{C:hh-hp}(1). For the $(p,\mathcal{N})$-topology, as in the proof of Lemma \ref{L:top-complete}, it suffices to show that the associated graded algebras, which are isomorphic to $\mathrm{HH}_*(\OK/W(\k)[z])$ and $\mathrm{HH}_*(\OK/W(\k)[z_0, z_1])$ respectively,  are all $p$-complete and separated. This in turn follows immediately from Corollary \ref{C:hh-hp}(2). 
\end{proof} 

\begin{lem}\label{L:inj}
Both the natural maps
\[
\tp_0(\OK/\Sz) \rightarrow \mathrm{HP}_0(\OK/W(\k)[z])
\]
and
\[
\tp_0(\OK/\Szz) \rightarrow \mathrm{HP}_0(\OK/W(\k)[z_0,z_1])
\]
are injective and strict with respect to the Nygaard filtrations. Moreover, both maps are embeddings with respect to the $(p,\mathcal{N})$-topology.

\end{lem}
\begin{proof}
Since both maps are compatible with the Nygaard filtration, for the first assertion, we are reduced to show that the induced maps on associated graded algebras
\begin{equation}\label{E:thhz-inj}
\thh_*(\OK/\Sz)\rightarrow \mathrm{HH}_*(\OK/W(\k)[z])
\end{equation}
and
\begin{equation}\label{E:thhzz-inj}
\thh_*(\OK/\Szz) \rightarrow \mathrm{HH}_*(\OK/W(\k)[z_0,z_1])
\end{equation}
are injective. For the second assertion, it is sufficient to show that both (\ref{E:thhz-inj}) and (\ref{E:thhzz-inj}) are embeddings under the $p$-adic topology.

Firstly, note that under the identification
$\mathrm{HH}_2(\OK/W(\k)[z])=\thh_2(\OK/\Sz)$,  
$t_{\EK(z)}$ maps to $u$ up to a unit of $\OK$. 
By Theorem \ref{T:rtp}(1) and Corollary \ref{C:hh-hp}(2), we deduce that (\ref{E:thhz-inj}) induces an embedding under the $p$-adic topology. 

By Lemma \ref{L:grad-thhzz}, we deduce that $\thh_{2j}(\OK/\Szz)$ is a successive extension of $\OK u_0^{2l}t_{z_0-z_1}^{[2j-2l]}$ for $l=0, 1, \dots, j$. On the other hand, by Corollary \ref{C:hh-hp}(2), we see that $\mathrm{HH}_{2j}(W(\k)/\Szz)$ is a successive extension of $\OK t_{\EK(z_0)}^{[2l]}t_{z_0-z_1}^{[2j-2l]}$ for $j=0, 1, \dots, k$.  Since for each $0\leq l\leq j$, 
\[
\OK u_{0}^{2l}t_{z_0-z_1}^{[2j-2l]}\to \OK t_{\EK(z_0)}^{[2l]}t_{z_0-z_1}^{[2j-2l]}
\] 
is an embedding under the $p$-adic topology, we conclude that (\ref{E:thhzz-inj}) is an embedding under the $p$-adic topology as well.  
\end{proof}

\begin{cor}
The periodic topological cyclic homology group $\tp_0(\OK/\Szz)$ is isomorphic to the closure of the subring of $D_{W(\k)[z_0,z_1]}((\EK(z_0),z_0-z_1))^\wedge_\mathcal{N}$ generated by $W(\k)[z_0,z_1]$ and $\{\iota(\delta^k(h))\}_{k\geq0}$ under either the $\mathcal{N}$-topology or the $(p, \mathcal{N})$-topology.
\end{cor}
\begin{proof}
This follows from the combination of Lemma \ref{L:top-complete}, Lemma \ref{L:dense}, Lemma \ref{L:hp-complete}, Lemma \ref{L:inj}.
%and Lemma \ref{L:iota}.
\end{proof}

%\begin{nrem}
%In the theorem, we first construct an incomplete $\delta$-ring, and then take the completion. Because the Frobenius is not continuous with respect to the topology defined by the Nygaard filtration, the completion is not a priori a $\delta$-ring. 

%However, we can define the $J$-topology by setting $p^n+\mathcal{N}^{\geq n}$ as a base. Then the Frobenius is continuous in $J$-topology by Lemma \ref{l10}. Since the graded pieces of the Nygaard filtration is $p$-complete, the $I$-completion agrees with $J$-completion, we conclude that the resulting completed ring has a $\delta$-structure and its Frobenius agrees with $TP_0(\OK/\Szz)$.
%\end{nrem}

%Let $I=(p,\mathcal{N}^{\geq1}TP_0(\OK/\Szz))$.
%\begin{lem}
%$\varphi^k(f^{(0)})\in I^{k+1}$.
%\end{lem}
%\begin{proof}
%For $k=0$, $f^{(0)}=x_1-x_2\in I$.

%Now we prove by induction. Suppose we have
%$\varphi^{k-1}(f^{(0)})\in I^{k}$.
%Then
%$$\varphi^k(f^{(0)}) = x_1^{p^k}-x_2^{p^k} = (x_2^{p^k-1}+\varphi^{k-1}(f^{(0)}))^p-x_2^{p^k}\in (p\varphi^{k-1}(f^{(0)}),\varphi^{k-1}(f^{(0)})^p)$$
%This shows $\varphi^k(f^{(0)})\in I^{k+1}$.
%\end{proof}

%\begin{lem}
%$\varphi^k(h)\in I^k$.
%\end{lem}

\section{Hopf algebroid}\label{s4}
We recall that a Hopf algebroid object in a symmetric monoidal category (with colimits) is a cogroupoid object in commutative algebras, i.e. (following \cite[Definition 6.1.2.7]{HTT}) a cosimplicial object $A:\Delta\rightarrow \mathrm{CAlg}$ such that
\begin{equation}\label{E:algebroid}
\begin{split}
\text{for any partition } [n] = S\cup S'~\text{with}~S\cap S' =\{s\}, \\
\text{the canonical map}~A^S \otimes_{A^{\{s\}}}A^{S'}\to A^{[n]}~\text{is an equivalence}.
\end{split}
 \end{equation}
In the following, by abuse of notation, we will refer to a Hopf algebroid $A$ by the pair $(A^{[0]}, A^{[1]})$. The goal of this section is to show that the pairs $(\tp_0(\OK/\Sz), \tp_0(\OK/\Szz))$ and $(\thh_*(\OK/\Sz), \thh_*(\OK/\Szz))$ form Hopf algebroids in appropriate categories.

We first recall some basics on (complete) filtered modules and graded modules. By a filtered ring we mean a commutative ring $R$ equipped with a decreasing filtration $\mathcal{F}^{\geq\bullet}$ on additive subgroups indexed by $\mathbb{Z}$ such that $\mathcal{F}^{\geq i}R\cdot \mathcal{F}^{\geq j}R\subset \mathcal{F}^{\geq i+j}R$. For our purpose, we assume that $R$ is a complete filtered ring, i.e. $R$ is complete and separated with respect to the topology defined by the filtration. Moreover, we assume that $R$ is a non-negative filtered ring. That is, $\mathcal{F}^{\geq i}R=R$ for all $i\leq 0$. Consequently, all $\mathcal{F}^{\geq i}R$ are ideals of $R$. 

By a filtered $R$-module we mean an $R$-module $M$ equipped with a decreasing filtration $\mathcal{F}^{\geq\bullet}$ on additive subgroups indexed by $\mathbb{Z}$ such that $\mathcal{F}^{\geq i}R\cdot \mathcal{F}^{\geq j}M\subset \mathcal{F}^{\geq i+j}M$. By our assumption on $R$, we see that all $\mathcal{F}^{\geq i}M$ are $R$-submodules of $M$. We say $M$ is non-negative if $\mathcal{F}^{\geq i}M=M$ for all $i\leq 0$.  For non-negative filtered $R$-modules $M_1, M_2$, their tensor product in the category of filtered $R$-modules is defined to be $M_1\otimes_{R}M_2$ equipped with the tensor product filtration. We say a filtered $R$-module $M$ complete if it is complete and separated with respect to the topology defined by the filtration. By taking completion with respect to the filtration, we obtain a functor $M\mapsto \hat{M}$ from the category of filtered $R$-modules to the category of complete filtered $R$-modules. For non-negative complete filtered $R$-modules $M_1, M_2$, their completed tensor product $M_1\hat{\otimes}_{R}M_2$ is defined to be the completion of the filtered $R$-module $M_1{\otimes}_{R}M_2$.

By a graded ring we mean a commutative ring $S$ together with a decomposition
$S=\oplus_{i\in\mathbb{Z}}S^{i}$
of additive subgroups such that $S^i\cdot S^j\subset S^{i+j}$. For our purpose, in the following we assume that $S$ is non-negative. That is, $S^i=0$ for all $i<0$. 

By a graded $S$-module we mean an $S$-module $N$ together with a decomposition $N=\oplus_{i\in\mathbb{Z}}N^{i}$ of additive subgroups such that $S^{i}\cdot N^j\subset N^{ i+j}$. We say $N$ is non-negative if $N^i=0$ for all $i<0$.  For non-negative graded $S$-modules $N_1, N_2$, we may define their tensor product $N_1\otimes_{S}N_2$ in the category of graded $S$-modules by setting
\[
(N_1\otimes_{S}N_2)^{i}=\oplus_{j+k=i}N_1^j\otimes_{S_0}N_2^k.
\]

By taking associated graded, we obtain a functor from the category of filtered $R$-modules to the category of graded $S$-modules, where
\[
S=\mathrm{Gr}(R):=\oplus_{i\in\mathbb{Z}}\mathcal{F}^{\geq i}R/\mathcal{F}^{\geq i+1}R
\]
is a non-negative graded ring, and 
\[
N=\mathrm{Gr}(M):=\oplus_{i\in\mathbb{Z}}\mathcal{F}^{\geq i}M/\mathcal{F}^{\geq i+1}M
\]
is a graded $S$-module. If $M$ is non-negative, then so is $\mathrm{Gr}(M)$. Moreover, it is clear that $\mathrm{Gr}(M)=\mathrm{Gr}(\hat{M})$.

\begin{defn}
Let $M$ be a complete filtered $R$-module with a filtration $\mathcal{N}^{\geq \bullet}$. We say $M$ is free and locally finite over $R$ if there exists a family of elements $\{m_i\}_{i\in I}\subset M$ such that the following two conditions hold.
\begin{enumerate}
\item
 For any $j\in\mathbb{Z}$, there are only finitely many $i\in I$ such that $m_i\not\in \mathcal{N}^{\geq j}M$.
 \item The induced morphism $\oplus_{i\in I}Rx_i\xrightarrow{x_i\mapsto m_i}M$ of filtered $R$-modules becomes an isomorphism after taking completion.
\end{enumerate}
\end{defn}

\begin{defn}
We say a graded $S$-module $N$ is free and locally finite if there exists a family of homogeneous elements $\{n_i\}_{i\in I}\subset N$ such that the following conditions hold.
\begin{enumerate}
\item
 For any $j\in\mathbb{Z}$, there are only finitely many $i\in I$ such that for some $k\leq j$, the image of $m_i$ in $N^{k}$ is non-zero.
 \item The induced morphism $\oplus_{i\in I}Sx_i\xrightarrow{x_i\mapsto n_i}N$ of graded $S$-modules is an isomorphism.
\end{enumerate}
\end{defn}

For a complete filtered $R$-module $M$, one easily checks that $M$ is free and locally finite over $R$ if and only if the associated graded module $\mathrm{Gr}(M)$ is free and locally finite over $S=\mathrm{Gr}(R)$.
%We call a $\mathcal{N}$-completed $TP_0(\OK/\Sz)$-modules $M$ is $\mathcal{N}$-completedly flat if the graded pieces $Gr_*M$ is flat as a graded module over $Gr_*TP_0(\OK/\Sz)$.
Moreover, if $M_1$ and $M_2$ are non-negative, free and locally finite over $R$. Then $M_1\hat{\otimes}_{R}M_2$ is free and locally finite over $R$ as well. Moreover, $\mathrm{Gr}(M_1\hat{\otimes}_{R}M_2)$ is canonically isomorphic to $
\mathrm{Gr}(M_1)\otimes_{S}\mathrm{Gr}(M_2)$ as graded $S$-modules.
%\end{lem}

%Define $\prism_{\OK/\Sz} = W(k)[x]$, which has the Nygaard filtration defined by powers of $\EK(x)$.

%Define $\prism_{\OK/\Szz}$ to be the $\delta$-ring obtained from $W(k)[x_1,x_2]$ by adding a generator $h$ with the equation$$\ef\EK(x_1^p)=x_1^p-x_2^p$$and an invertible element $\epsilon$ with equation$\varphi(\epsilon)\EKf(x_1^p) = \epsilon\EKf(x_2^p)$$We define the Nygaard filtration on $\prism_{\OK/\Szz}$ to be the coarsest multiplicative filtration such that $f^{(k)}$ has filtration $p^k$.

\begin{prop} 
The following statements are true.
\begin{enumerate}
\item[(1)]Both
$d^0$ and $d^1$ (which are just $\eta_R$ and $\eta_L$ defined in \S3 respectively) exhibit $\tp_0(\OK/\Szz)$ as
a free and locally finite filtered $\tp_0(\OK/\Sz)$-module.
\item[(2)]Both
$d^0$ and $d^1$ (which are just $\eta_R$ and $\eta_L$ defined in \S3 respectively) exhibit $\thh_*(\OK/\Szz)$ as
a free and locally finite graded $\thh_*(\OK/\Sz)$-module.
\end{enumerate}
\end{prop}
\begin{proof}
Since the associated graded algebra of $(\tp_0(\OK/\Sz)), \tp_0(\OK/\Szz))$ is canonically isomorphic to $(\thh_*(\OK/\Sz), \thh_*(\OK/\Szz))$, we have (2) implies (1). For (2), 
We only need to treat the case of $\eta_L$. By Lemma \ref{L:grad-thhzz}, we see that $\thh_*(\OK/\Szz)/(u_0)$ is a free $\OK$-module with a basis of degrees $0, 2, 4,\dots$ respectively. Using Corollary \ref{C:grad-thhzz}, we may further deduce that such a basis lifts to a basis of $\thh_*(\OK/\Szz)$ over $\thh_*(\OK/\Sz)$ with the same degrees. Hence $\thh_*(\OK/\Szz)$ is free and locally finite over $\thh_*(\OK/\Sz)$ via $\eta_L$. 
\end{proof}

\begin{cor}\label{C:flat}
We consider $\tp_0(\OK/\Szz)$ (resp. $\thh_*(\OK/\Szz)$) as a filtered $\tp_0(\OK/\Sz)$-module (resp. graded $\thh_*(\OK/\Sz)$-module) via either $\eta_L$ or $\eta_R$. Then it is completely flat.
\end{cor}

For $0\leq i\leq n$, consider the natural maps
\begin{equation}\label{E:tensor-thh-n}
\thh_*(\OK/\mathbb{S}_{W(\k)}[z_0,\dots, z_i])\otimes_{\thh_*(\OK/\mathbb{S}[z_i])}\thh_*(\OK/\mathbb{S}_{W(\k)}[z_i, \dots, z_n])\to \thh_*(\OK/\mathbb{S}_{W(\k)}[z_0,\dots, z_n])
\end{equation}
and
\begin{equation*}
\tp_0(\OK/\mathbb{S}[z_0,\dots, z_i])\otimes_{\tp_0(\OK/\mathbb{S}[z_i])}\tp_j(\OK/\mathbb{S}_{W(\k)}[z_i, \dots, z_n]) \to \tp_j(\OK/\mathbb{S}_{W(\k)}[z_0, \dots, z_n]).
\end{equation*}
By Remark \ref{R:tate-n-variable}, the Tate spectral sequence for $\tp_*(\OK/\mathbb{S}[z_0,\dots,z_n])$ degenerates at the $E^2$-term. It follows that the Nygaard filtration on $\tp_j(\OK/\mathbb{S}[z_0,\dots,z_n])$ is complete by the same argument as in the proof of Lemma \ref{L:top-complete}. Hence the second map induces
\begin{equation}\label{E:tensor-tp-n}
\tp_0(\OK/\mathbb{S}[z_0,\dots, z_i])\hat{\otimes}_{\tp_0(\OK/\mathbb{S}[z_i])}\tp_j(\OK/\mathbb{S}_{W(\k)}[z_i, \dots, z_n]) \to \tp_j(\OK/\mathbb{S}_{W(\k)}[z_0, \dots, z_n]).
\end{equation}

\begin{lem}\label{L:tensor-thh-tp}
Both (\ref{E:tensor-thh-n}) and (\ref{E:tensor-tp-n}) are isomorphisms. 
\end{lem}
\begin{proof}
The first assertion follows from the multiplicative property of THH. This in turn implies that (\ref{E:tensor-tp-n}) becomes an isomorphism after taking associated graded algebras on both sides. Thus (\ref{E:tensor-tp-n}) itself is an isomorphism. 
\end{proof}

\begin{cor}
The cosimplicial objects $\thh_*(\OK/\mathbb{S}_{W(\k)}[z]^{\otimes\oldbullet})$ and $\tp_0(\OK/\mathbb{S}_{W(\k)}[z]^{\otimes\oldbullet})$ are cogroupoid objects (i.e. Hopf algebroids) in the categories of graded rings and complete filtered rings respectively.
\end{cor}
\begin{proof}
We need to check the condition (\ref{E:algebroid}). By Lemma \ref{L:tensor-thh-tp}, we deduce the case $S=[0,\dots,i]$ and $S'=[i,\dots,n]$. We conclude the general case using the action of symmetric groups on $\thh_*(\OK/\mathbb{S}_{W(\k)}[z]^{\otimes\oldbullet})$ and $\tp_0(\OK/\mathbb{S}_{W(\k)}[z]^{\otimes\oldbullet})$.
\end{proof}

%It follows that. These are called Hopf algebroids in the topology literature. We list the structure maps in the following.

%We list some maps induced from the structure maps of the cogroupoids, which generate all the structure maps.

\begin{conv}
For an Hopf algebroid, we denote its coproduct, counit and conjugation by $\Delta$, $\varepsilon$ and $c$ respectively. 
\end{conv}

%Let $\epsilon$ be defined by the equation
%$$\epsilon^{-1}\varphi(\epsilon) ={\EK(x_2^p)\over \EK(x_1^p)}$$
%\begin{thm}
%$$\psi(\sigma) = \epsilon^{-1}\sigma$$
%\end{thm}

In the following, we give an explicit description of the Hopf algebroid 
\[
(\thh_*(\OK/\Sz),\thh_*(\OK/\Szz)).
\]
Recall that the cyclotomic Frobenius on $\tp_0(\OK/\Szz)$ induces a $\delta$-ring structure (Proposition \ref{P:tp-delta}) .
\begin{lem}\label{L:delta-h}
For any $i\geq0$, $\delta^i(h)\in \mathcal{N}^{\geq2} \tp_0(\OK/\Szz)$.
\end{lem}
\begin{proof}
Firstly, it is clear that $\varepsilon(\varphi(z_0-z_1))=0$ and $\varepsilon(\varphi(\EK(z_0)))=\varphi(\EK(z))$. It follows that $\varepsilon(h)=0$. Hence for all $i\geq0$, 
\[
\varepsilon(\delta^i(h))=\delta^i(\varepsilon(h))=0.
\] 
On the other hand, note that $\varepsilon$ induces an isomorphism 
\[
\mathrm{Gr}^0(\tp_0(\OK/\Szz))\cong \mathrm{Gr}^0(\tp_0(\OK/\Sz)).
\]
This implies that 
$\ker(\varepsilon)\subset \mathcal{N}^{\geq2} \tp_0(\OK/\Szz)$.
The lemma follows. 
\end{proof}
\begin{cor}\label{C:u-ext}
As graded rings, we have
\begin{equation}\label{E:isom-thhzz}
\thh_*(\OK/\Szz) = \OK[u_0]\otimes_{\OK}\OK\langle t_{z_0-z_1}\rangle. 
\end{equation}
\end{cor}
\begin{proof}
By Lemma \ref{L:delta-h} and (\ref{E:induc-odd}), we get that $-(f^{(k)})^p$ and $pf^{(k+1)}$ have the same image in $\thh_{2p^{k+1}}(\OK/\Szz)$.  Using the argument of Lemma \ref{L:dense}, we conclude that the images of $\{f^{(k)}\}_{k\geq0}$ in $\thh_*(\OK/\Szz)$ generates $t^{[j]}_{z_0-z_1}$ for all $j\geq0$ over $\mathbb{Z}_p$. 
This allows us to define the $\OK[u_0]$-linear map 
\[
\OK[u_0]\otimes_{\OK}\OK\langle t\rangle\to \thh_*(\OK/\Szz), \quad t^{[j]}\mapsto t^{[j]}_{z_0-z_1}. 
\]
By Lemma \ref{L:grad-thhzz}, this map induces isomorphisms on associated graded modules under the $u_0$-filtrations. Hence it is an isomorphism. 
\end{proof}

\begin{rem}
Since we already know that $\thh_*(\OK/\Szz)$ is $p$-torsion free (Corollary \ref{C:tpzz-integral}), the existence of the divided powers in \eqref{E:isom-thhzz} is a merely a ring-theoretical property and not additional structure that needs to be defined.
\end{rem}

\begin{prop}\label{P:hopf-thh}
Under the isomorphism (\ref{E:isom-thhzz}), we have
\begin{equation}\label{E:right-unit-u}
u_1= u_0-E'_K(\pK)t_{z_0-z_1},
\end{equation}
and 
\begin{equation}\label{E:delta-t}
\Delta(t^{[i]}_{z_0-z_1}) = \sum_{0\leq j\leq i} t^{[j]}_{z_0-z_1}\otimes t^{[i-j]}_{z_0-z_1}, \quad \varepsilon(t_{z_0-z_1})=0.
\end{equation}
\end{prop}
\begin{proof}
Write 
\[
\EK(z_1)=\EK(z_0)-E'_K(z_0)(z_0-z_1)+(z_0-z_1)^2F(z_0, z_1)
\]
for some $F\in K_0[x, y]$. Since $z_0-z_1\in\mathcal{N}^{\geq2}\tp_0(\OK/\Szz)$, by Theorem \ref{T:rtp},  we deduce that 
\[
u_1=p_2(E_K(z_1))=p_2(\EK(z_0)-E'_K(z_0)(z_0-z_1))=u_0-E'_K(\pK)t_{z_0-z_1}.
\]
This yields (\ref{E:right-unit-u}).  We conclude (\ref{E:delta-t}) by the binomial expansion
\[
(z_0-z_2)^i=\sum_{0\leq j\leq i}\frac{i!}{j!(i-j)!}(z_0-z_1)^j(z_1-z_2)^{i-j}. 
\]
\end{proof}

\section{The descent spectral sequence}\label{s5}
We consider the $\Sz$-Adams resolution for $\Sk$: 
\begin{equation}\label{E:adams-witt}
\Sk\rightarrow \Sz^{\otimes \oldbullet},
\end{equation}
where the tensor product $\Sz^{\otimes [n]}$ is taken relative to $\Sk$.
It induces the augmented cosimplicial $E_\infty$-algebra in cyclotomic spectra
\begin{equation}\label{E:adams-thh}
\thh(\OK/\Sk) \rightarrow \thh(\OK/\Sz^{\otimes \oldbullet}),
\end{equation}
which in turn induces augmented cosimplicial $E_\infty$-algebras in spectra
\begin{equation}\label{E:adams-tc-}
\tc^{-}(\OK/\Sk) \rightarrow \tc^{-}(\OK/\Sz^{\otimes \oldbullet})
\end{equation}
and
\begin{equation}\label{E:adams-tp}
\tp(\OK/\Sk) \rightarrow \tp(\OK/\Sz^{\otimes \oldbullet}).
\end{equation}

By the multiplicative property of THH,
$\thh(\OK/\Sz^{\otimes [n]})$ is equivalent to $\thh(\OK/\Sz)^{\otimes[n]}$, where the tensor product is taken relative to $\thh(\OK/\Sk)$.
Hence $(\ref{E:adams-thh})$ is an Adams resolution for $\thh(\OK/\Sk)$ in the category of $\thh(\OK/\Sk[z])$-modules.

\begin{prop}\label{P:tot-thh}
The augmented cosimplicial $E_\infty$-algebra in cyclotomic spectra (\ref{E:adams-thh}) induces 
\begin{equation}\label{E:tot-thh}
\thh(\OK/\Sk) \simeq \biglimit_\Delta\thh(\OK/\Sz^{\otimes \oldbullet}).
\end{equation}
\end{prop}
\begin{proof}
By \cite[Proposition 2.14]{MNN}, the fiber of 
\begin{equation}\label{e43}
\thh(\OK/\Sk) \rightarrow \biglimit_{\Delta^{\leq n}}\thh(\OK/\Sz^{\otimes \oldbullet})
\end{equation}
is homotopy equivalent to $(n+1)$-fold self-smash product of the fiber of 
\begin{equation}\label{E:fiber-s}
\thh(\OK/\Sk) \rightarrow \thh(\OK/\Sz).
\end{equation} 
It follows that the fiber of (\ref{e43}) is $n$-connected as the fiber of (\ref{E:fiber-s}) is $0$-connected. The proposition follows.
\end{proof}

\begin{cor}\label{C:tot-tp}
The augmented cosimplicial spectra (\ref{E:adams-tc-}), (\ref{E:adams-tp}) induce
\begin{equation}\label{E:tot-tc-}
\tc^-(\OK/\Sk)\simeq \biglimit_\Delta\tc^-(\OK/\Sz^{\otimes \oldbullet})
\end{equation}
and 
\begin{equation}\label{E:tot-tp}
\tp(\OK/\Sk)\simeq \biglimit_\Delta\tp(\OK/\Sz^{\otimes \oldbullet}).
\end{equation}
\end{cor}
\begin{proof}
The claim for $\tc^-$ follows from the natural equivalence
\[
(\biglimit_\Delta\thh(\OK/\Sz^{\otimes \oldbullet}))^{h\mathbb{T}}\simeq \biglimit_\Delta\thh(\OK/\Sz^{\otimes \oldbullet})^{h\mathbb{T}}.
\] 
For the case of $\tp$, first note the natural equivalence
\[
(\biglimit_{\Delta^{\leq n}}\thh(\OK/\Sz^{\otimes \oldbullet}))_{h\mathbb{T}}\simeq \biglimit_{\Delta^{\leq n}}\thh(\OK/\Sz^{\otimes \oldbullet})_{h\mathbb{T}}.
\] 
Since the fiber of (\ref{e43}) is $n$-connected by the proof of Proposition \ref{P:tot-thh}, the fiber of 
\[
\thh(\OK/\Sk)_{h\mathbb{T}} \rightarrow (\biglimit_{\Delta^{\leq n}}\thh(\OK/\Sz^{\otimes \oldbullet}))_{h\mathbb{T}}
\]
is $n$-connected as well. Hence the fiber of 
\[
\thh(\OK/\Sk)_{h\mathbb{T}} \rightarrow \biglimit_{\Delta^{\leq n}}\thh(\OK/\Sz^{\otimes \oldbullet})_{h\mathbb{T}}
\]
is $n$-connected, concluding the natural equivalence
\[
\thh(\OK/\Sk)_{h\mathbb{T}}\simeq \biglimit_\Delta\thh(\OK/\Sz^{\otimes \oldbullet})_{h\mathbb{T}}.
\]
This yields the claim for $\tp$.
\end{proof}

Using Proposition \ref{P:tot-thh} and Corollary \ref{C:tot-tp},  the coskeleton filtrations of $\thh(\OK/\Sz^{\otimes \oldbullet})$, $\tp(\OK/\Sz^{\otimes \oldbullet})$ and $\tc^{-}(\OK/\Sz^{\otimes \oldbullet})$ give rise to multiplicative second quadrant homology type spectral sequences converging to $\thh_*(\OK/\mathbb{S}_{W(\k)})$, $\tp_*(\OK/\mathbb{S}_{W(\k)})$ and $\tc^{-}_*(\OK/\mathbb{S}_{W(\k)})$ respectively.
\begin{itemize}
\item The \emph{descent spectral sequence} for $\thh(\OK/\Sk)$:
\[
E^1_{i,j}(\thh(\OK))=\thh_j(\OK/\Sz^{\otimes {[-i])}})\Rightarrow \thh_{i+j}(\OK/\Sk).
\]
By Lemma \ref{L:tensor-thh-tp} and Corollary \ref{C:flat}, the $E^1$-term may be identified with the cobar complex for $\thh_*(\OK/\Sz)$ with respect to the Hopf algebroid 
\[
(\thh_*(\OK/\Szz), \thh_*(\OK/\Sz)).
\] 
It follows that   
\[
E^2_{i,j}(\thh(\OK))\cong \mathrm{Ext}^{-i, j}_{\thh_*(\OK/\Szz)}(\thh_*(\OK/\Sz)).
\]
\item The \emph{descent spectral sequence} for $\tp(\OK/\Sk)$:
\[
E^1_{i,j}(\tp(\OK))=\tp_j(\OK/\Sz^{\otimes[-i]})\Rightarrow \tp_{i+j}(\OK/\Sk).
\]
By Lemma \ref{L:tensor-thh-tp} and Corollary \ref{C:flat}, the $j$-th row of the $E^1$-term may be identified with the cobar complex for $\tp_j(\OK/\Sz)$ with respect to the Hopf algebroid $(\tp_0(\OK/\Szz), \tp_0(\OK/\Sz))$. It follows that   
\[
E^2_{i,j}(\tp(\OK))\cong \mathrm{Ext}^{-i}_{\tp_0(\OK/\Szz)}(\tp_j(\OK/\Sz)).
\]
\item The \emph{descent spectral sequence} for $\tc^-(\OK/\Sk)$:
\[
E^1_{i,j}(\tc^-(\OK))=\tc^-_j(\OK/\Sz^{\otimes[-i]})\Rightarrow \tc^-_{i+j}(\OK/\Sk).
\]
\end{itemize}
The Ext-groups appearing in the $E^2$-terms are considered in the abelian categories of comodules over the corresponding Hopf algebroids. Here we adopt the convention that for graded objects, $\mathrm{Ext}^{i,j}$ means the degree $j$ part of the graded abelian group $\mathrm{Ext}^i$.
In terms of geometric language, these abelian categories are the abelian categories of quasicoherent sheaves on the respective stacks defined by the Hopf algebroids, and the multiplicative structure on the $E^2$-terms are induced from the symmetric monoidal structure in the category of quasicoherent sheaves.

\begin{rem}
Indeed, the $E^2$-term of the descent spectral sequence for $\tc^-(\OK/\Sk)$ may also be identified with certain $\mathrm{Ext}$-groups in the category of complete filtered comodules over filtered Hopf algebroids. The details will be given in \cite{LW}.
\end{rem}

Using (\ref{E:tot-tc-}) and (\ref{E:tot-tp}), we may also construct a spectral sequence converging to $\tc_*(\OK/\Sk)$.
Firstly, the maps $\can, \varphi$ induce the maps of cosimplical $E_\infty$-algebras in spectra
\[
\can,\varphi: \tc^-(\OK/\Sz^{\otimes \oldbullet})\rightarrow \tp(\OK/\Sz^{\otimes \oldbullet}).
\]
Define $\tc(\OK/\Sk)_{(n)}$ to be the fiber of
\[
\can-\varphi: \biglimit_{\Delta^{\leq n}}\tc^-(\OK/\Sz^{\otimes \oldbullet})\rightarrow \biglimit_{\Delta^{\leq n-1}}\tp(\OK/\Sz^{\otimes \oldbullet}).
\]
By construction, we get
\begin{eqnarray*}
{\tc(\OK/\Sk)_{(n)}  \over \tc(\OK/\Sk)_{(n+1)}}\simeq {\biglimit_{\Delta^{\leq n}}\tc^-(\OK/\Sz^{\otimes \oldbullet})\over \biglimit_{\Delta^{\leq n+1}}\tc^-(\OK/\Sz^{\otimes \oldbullet})} \oplus \Sigma^{-1}  {\biglimit_{\Delta^{\leq n-1}}\tp(\OK/\Sz^{\otimes \oldbullet})\over \biglimit_{\Delta^{\leq n}}\tp(\OK/\Sz^{\otimes \oldbullet})}.
\end{eqnarray*}
The tower $\{\tc(\OK)_{(n)}\}_{n\geq0}$ gives rise to the \emph{descent spectral sequence} for $\tc(\OK/\Sk)$:
\[
E_{i,j}^1(\tc(\OK))\Rightarrow \tc_{i+j}(\OK/\Sk).
\] 
Note that 
 $E^1(\tc(\OK))$ may be identified with the total complex of the double complex
\[
E^1(\tc^-(\OK))\xrightarrow{\can-\varphi}E^1(\tp(\OK)).
\]
Consequently,  
there is a  multiplicative spectral sequence 
\[
\tilde{E}^2_{i,k, j}(\tc(\OK)) \Rightarrow E^2_{i-k,j}(\tc(\OK)), \quad k\in\{0,1\}, 
\]
associated with this double complex,
so 
\[
\tilde{E}^2_{i,0,j}(\tc(\OK))= \ker(\can-\varphi: E^2_{i,j}(\tc^{-}(\OK))\to E^2_{i, j}(\tp(\OK))),
\]
and 
\[
\tilde{E}^2_{i,1,j}(\tc(\OK))=\mathrm{coker}(\can-\varphi: E^2_{i,j}(\tc^{-}(\OK))\to E^2_{i, j}(\tp(\OK))).
\]

%\section{The descent spectral sequence for $THH(\OK/\Sk)$}

In the rest of this section, we will determine $E^2_{i,j}(\thh(\OK))$ explicitly. To this end, first note that it follows from Corollary \ref{C:u-ext} and (\ref{E:delta-t}) that the map of left-$\thh_*(\OK/\Sz)$-modules
\[
D: \thh_*(\OK/\Szz)\to \thh_*(\OK/\Szz), 
\]
which sends $t_{z_0-z_1}^{[i]}$ to $t_{z_0-z_1}^{[i-1]}$, is indeed a map of left $\thh(\OK/\Szz)$-modules. It follows that the complex 
\begin{equation}\label{E:rel-inj}
0\to \thh_*(\OK/\Sz)\xrightarrow{\eta_L}\thh_*(\OK/\Szz\xrightarrow{a\mapsto D(a)dz} \thh_*(\OK/\Szz) dz\to 0,
\end{equation}
where $dz$ has degree $2$, is a relative injective resolution for $\thh_*(\OK/\Sz)$ as left $\thh_*(\OK/\Szz)$-modules.

\begin{prop}\label{P:ext-thh}
The extension $\mathrm{Ext}_{\thh_*(\OK/\Szz)}(\thh_*(\OK/\Sz)$ is computed by the complex
\begin{equation}\label{E:complex-thh}
\thh_*(\OK/\Sz)\xrightarrow{(D_0\circ\eta_R)dz}\thh_*(\OK/\Sz)dz,
\end{equation}
where 
\[
D_0: \thh_*(\OK/\Szz)\rightarrow \thh_*(\OK/\Sz)
\] 
is the map of left $\thh_*(\OK/\Sz)$-modules given by 
\[
D_0(t_{z_0-z_1})=1, \quad D_0(t^{[i]}_{z_0-z_1})=0~\text{for $i\neq 1$}. 
\] 

\end{prop}
\begin{proof}
Using (\ref{E:rel-inj}),  we first get that $\mathrm{Ext}^{i,j}_{\thh_*(\OK/\Szz)}(\thh_*(\OK/\Sz)$ is computed by the complex
\begin{equation}\label{E:ext-1}
\begin{split}
&\mathrm{Hom}_{\thh_*(\OK/\Szz)}(\thh_*(\OK/\Sz), \thh_*(\OK/\Szz))\\
&\xrightarrow{f\mapsto (D\circ f)dz} \mathrm{Hom}_{\thh_*(\OK/\Szz)}(\thh_*(\OK/\Sz), \thh_*(\OK/\Szz))dz.
\end{split}
\end{equation}
Recall that for a (commutative) Hopf algebroid $(A, \Gamma)$, a left $\Gamma$-module $M$ and an $A$-module $N$, there is a canonical isomorphism 
\begin{equation}\label{E:hom-isom}
\mathrm{Hom}_{A}(M, N)\cong \mathrm{Hom}_{\Gamma}(M, \Gamma\otimes_AN), \quad f\mapsto \tilde{f}=(\mathrm{id}\otimes f)\circ \Delta.
\end{equation}
It is straightforward to check that $D_0$ corresponds to $D$ under this isomorphism. It follows that (\ref{E:ext-1}) may be identified with the complex
% $\thh_*(\OK/\Sz)$-linear complex
\begin{equation}\label{E:ext-2}
\begin{split}
\mathrm{Hom}&_{\thh_*(\OK/\Sz)}(\thh_*(\OK/\Sz), \thh_*(\OK/\Sz))\\
&\xrightarrow{f\mapsto (D\circ \tilde{f})dz} \mathrm{Hom}_{\thh_*(\OK/\Sz)}(\thh_*(\OK/\Sz), \thh_*(\OK/\Sz))dz.
\end{split}
\end{equation}
Note that under the isomorphism (\ref{E:hom-isom}), the identity map on $\thh_*(\OK/\Sz)$ corresponds to $\eta_R$ . We thus conclude the proposition by the isomorphism
 \[
 \mathrm{Hom}_{\thh_*(\OK/\Sz)}(\thh_*(\OK/\Sz), \thh_*(\OK/\Sz))\cong \thh_*(\OK/\Sz),
 \]
which sends $f$ to $f(1)$.
\end{proof}

The following results follow immediately. 
\begin{cor}\label{C:thh-sz}
There are canonical isomorphisms
\[
\mathrm{Ext}^{0,0}_{\thh_*(\OK/\Szz)}(\thh_*(\OK/\Sz))\cong \OK
\]
and
\[
\mathrm{Ext}^{1,2n}_{\thh_*(\OK/\Szz)}(\thh_*(\OK/\Sz))\cong \OK/(nE'_K(\pK)), \quad n\geq1.
\]
The other $\mathrm{Ext}$-groups vanish. As a consequence, the descent spectral sequence for $\thh(\OK/\Sk)$ collapses at the $E^2$-term.
\end{cor}

\begin{rem}
Corollary \ref{C:thh-sz} recovers the main result of \cite{LM}. A different proof of this statement was
given by Krause–Nikolaus \cite{KN}.
\end{rem}
%\begin{thm}
%$$THH(\OK)_{hS^1}\cong Tot(THH(\OK/\Sz^{\otimes i})_{hS^1})$$
%\end{thm}
%\begin{proof}
%For any connective $S^1$-spectrum $X$, the cell filtration on $BS^1$ induces a functorial filtration on $X_{hS^1}$ such that the $k$-th graded pieces are $\Sigma^{2k}X$: we have a tower
%$$X_{(0)}\rightarrow X_{(1)}\rightarrow X_{(2)}\rightarrow\dots \rightarrow X_{hS^1}$$
%such that $$X_{(k)}/X_{(k-1)}\cong \Sigma^k X$$
%and $X_{hS^1}\cong colim X_{(k)}$.

%Now we define
%$$Y_k = Tot(THH(\OK/\Sz^{\otimes i})_{(k)})$$

%\end{proof}
%\section{The $\tc$ spectral sequence}

%\section{The algebraic Tate and homotopy fixed points spectral sequence}

In the remainder of this section, we introduce the \emph{algebraic Tate spectral sequence} and the \emph{algebraic homotopy fixed points spectral sequence}. First note that the $d^1$-differentials of the $E^1$-terms of these descent spectral sequences
\[
d^1: E^1_{i,j}\to E^1_{i-1,j}
\] 
 leave $j$ unchanged, giving these $E^1$-terms a structure of chain complexes. Moreover, the Nygaard filtration defines filtrations on the $E^1$-terms of the descent spectral sequences for $\tc^-$ and $\tp$.    
The \emph{algebraic homotopy fixed points spectral sequence}
\begin{equation}\label{E:ahfs}
E_{i,j,k}^1(\tc^-(\OK))=H^i(\mathrm{Gr}^{2k}(\tc^-_j(\OK/\Sz^{\otimes \oldbullet})))\Rightarrow E_{-i,j}^2(\tc^{-}(\OK)), 
\end{equation}
and the \emph{algebraic Tate spectral sequence}
\begin{equation}\label{E:ats}
E_{i,j,k}^1(\tp(\OK))=H^i(\mathrm{Gr}^{2k}(\tp_j(\OK/\Sz^{\otimes \oldbullet})))\Rightarrow E_{-i,j}^2(\tp(\OK))
\end{equation}
are the spectral sequence associated with the resulting filtered chain complexes.
They are multiplicative spectral sequences with differentials
\[
d^r: E_{i,j,k}^{r}\to E_{i+1,j,k+r}^{r}.
\]
Since $d^r$ leaves $j$ unchanged, so these “algebraic” spectral sequences may be considered one $j$ at a time.

By Remark \ref{R:tate-n-variable}, we see that the associated graded of the Nygaard filtrations of $E^1(\tp(\OK))$ may be identified with part of $E^1(\thh(\OK))[\sigma^{\pm 1}]$. That is, 
\[
E_{i,j,k}^1(\tp(\OK))= 
\begin{cases}
0, \quad \text{$j$ odd} \\
E_{-i,2k}^1(\thh(\OK))\sigma^{j\over2},\quad\text{$j$ even}.
\end{cases}
\]
%Here we identify the graded pieces of $\tp_0(\OK/\Sk[z]^{\otimes[i]})$ with $\thh_*(\OK/\Sk[z]^{\otimes[i]})$
%and apply $\tp_{*}(\OK/\Sk[z]^{\otimes[i]})=\tp_0(\OK/\Sk[z]^{\otimes[i]})[\sigma^{\pm1}]$. 
Moreover, under this identification, the $d^r$-differential (for $j$ even) is given by 
\[
d^r: E_{-i,2k}^1(\thh(\OK))\sigma^{j\over2}\xrightarrow{c\sigma^{j\over2}\to d^r(c)\sigma^{j\over2}} E_{-i-1,2k}^1(\thh(\OK))\sigma^{j\over2}.
\]

%One can see that the above isomorphism respects the $d^1$-differentials of the algebraic Tate spectral sequence and the descent spectral sequence for $\thh(\OK)$.

Since the algebraic homotopy fixed points spectral sequence is a truncation of the algebraic Tate spectral sequence, using Corollary \ref{C:thh-sz}, the following result follows immediately. 
\begin{prop}\label{P:e2tc-tp}
Both $E^2(\tc^-(\OK))$ and $E^2(\tp(\OK))$
are concentrated in $E^2_{0,*}$ and $E^2_{-1,*}$. In particular, both the descent spectral sequences for $\tc^-(\OK/\Sk)$ and $\tp(\OK/\Sk)$ collapse at the $E^2$-term.
\end{prop}

\section{Refined algebraic Tate differentials} \label{s6}
In this section, we consider mod $p$ version of descent spectral sequences. To determine the $E^2$-terms of mod $p$ descent sequences for $\tp(\OK)$ and $\tc^{-}(\OK)$, we introduce refined version of algebraic Tate and algebraic homotopy fixed point spectral sequences, and completely determine the refined algebraic Tate differentials.
%in the case $p$ odd, $\eK>1$ or $p=2$, $\eK>3$. 

%Throughout this section, suppose $\eK>1$.

By Lemma \ref{L:tensor-thh-tp} and induction on $n$, we get that $\thh_*(\OK/\Sz^{\otimes[n]})$ is $p$-torsion free for all $n\geq0$. Hence $\tp_*(\OK/\Sz^{\otimes[n]})$ and $\tc^{-}_*(\OK/\Sz^{\otimes [n]})$ are all $p$-torsion free as well by the fact that both the Tate and the homotopy fixed point spectral sequence degenerate. It follows that for $n\geq1$, 
\begin{equation*}\label{E:thh-tp-fp}
\begin{split}
\thh_*(\OK/\Sz^{\otimes [n]};\Fp)&=\thh_*(\OK/\Sz^{\otimes [n]})\otimes_{\mathbb{Z}}\Fp,\\ 
\tp_*(\OK/\Sz^{\otimes [n]};\Fp)&=\tp_*(\OK/\Sz^{\otimes [n]})\otimes_{\mathbb{Z}}\Fp
\end{split}
\end{equation*}
and
\[
\tc^{-}_*(\OK/\Sz^{\otimes [n]};\Fp)=\tc^{-}_*(\OK/\Sz^{\otimes [n]})\otimes_{\mathbb{Z}}\Fp.
\]
This in turn implies that the Tate and homotopy fixed point spectral sequences for $\tp_*(\OK/\Sz^{\otimes [n]};\Fp)$ and $\tc^{-}_*(\OK/\Sz^{\otimes [n]};\Fp)$ collapse at the $E^2$-term respectively. Moreover, analogues of Proposition \ref{P:tot-thh} and Corollary \ref{C:tot-tp} hold as well. Thus  the coskeleton filtrations of the cosimplicial spectra 
\[
\thh(\OK/\Sz^{\otimes \oldbullet};\Fp), \quad \tp(\OK/\Sz^{\otimes \oldbullet};\Fp), \quad\tc^{-}(\OK/\Sz^{\otimes \oldbullet};\Fp)
\] 
give rise to \emph{mod p descent spectral sequences} converging to $\thh_*(\OK;\Fp)$, $\tp_*(\OK;\Fp)$ and $\tc^{-}_*(\OK;\Fp)$ as follows.
\begin{itemize}
\item The descent spectral sequence for $\thh(\OK; \Fp)$:
\[
E^1_{i,j}(\thh(\OK);\Fp)=\thh_j(\OK/\Sz^{\otimes [-i]};\Fp)\Rightarrow \thh_{i+j}(\OK;\Fp).
\]
The $E^1$-term may be identified with the cobar complex for $\thh_*(\OK/\Sz;\Fp)$ with respect to the Hopf algebroid 
\[
(\thh_*(\OK/\Szz;\Fp), \thh_*(\OK/\Sz;\Fp)).
\] 
Hence  
\[
E^2_{i,j}(\thh(\OK);\Fp)\cong \mathrm{Ext}^{-i, j}_{\thh_*(\OK/\Szz;\Fp)}(\thh_*(\OK/\Sz;\Fp)).
\]
\item The descent spectral sequence for $\tp(\OK;\Fp)$:
\[
E^1_{i,j}(\tp(\OK);\Fp)=\tp_j(\OK/\Sz^{\otimes (-i+1)};\Fp)\Rightarrow \tp_{i+j}(\OK;\Fp).
\]
The $j$-th row of the $E^1$-term may be identified with the cobar complex for $\tp_j(\OK/\Sz;\Fp)$ with respect to the Hopf algebroid 
\[
(\tp_0(\OK/\Szz;\Fp), \tp_0(\OK/\Sz;\Fp)).
\] 
It follows that   
\[
E^2_{i,j}(\tp(\OK);\Fp)\cong \mathrm{Ext}^{-i}_{\tp_0(\OK/\Szz;\Fp)}(\tp_j(\OK/\Sz;\Fp)).
\]
\item The descent spectral sequence for $\tc^-(\OK;\Fp)$:
\[
E^1_{i,j}(\tc^{-}(\OK);\Fp)=\tc^{-}_j(\OK/\Sz^{\otimes [-i]};\Fp)\Rightarrow \tc^{-}_{i+j}(\OK;\Fp).
\]
\item The descent spectral sequence for $\tc(\OK;\Fp)$:
\[
E^{i,j}_1(\tc(\OK); \Fp)\Rightarrow \tc_{i+j}(\OK;\Fp).
\] 
Similarly, there is a multiplicative spectral sequence
\[
\tilde{E}^2_{i,k, j}(\tc(\OK);\Fp) \Rightarrow E^2_{i-k,j}(\tc(\OK);\Fp), \qquad k\in\{0,1\},
\]
where
\[
\tilde{E}^2_{i,0,j}(\tc(\OK);\Fp)= \ker(\can-\varphi: E^2_{i,j}(\tc^{-}(\OK);\Fp)\to E^2_{i, j}(\tp(\OK);\Fp),
\]
and 
\[
\tilde{E}^2_{i,1,j}(\tc(\OK);\Fp)=\mathrm{coker}(\can-\varphi: E^2_{i,j}(\tc^{-}(\OK))\to E^2_{i, j}(\tp(\OK);\Fp).
\]
\end{itemize}

\begin{rem}\label{R:multiplicative}
Recall that the mod $p$ Moore spectrum is multiplicative if and only if $p\neq 2$. It follows that the spectra $\thh(\OK;\Fp)$, $\tc^-(\OK;\Fp)$, $\tp(\OK;\Fp)$ and $\tc(\OK;\Fp)$ are multiplicative if and only if $p\neq 2$. On the other hand, since the descent spectral sequences for $\OK$ are multiplicative, and the reduction maps of the $E^1$-terms are surjective, we deduce that the $E^1$ and $E^2$-terms of the mod $p$ descent spectral sequences are multiplicative for all $p$. 
\end{rem}

In the following, we will first determine the $E^2$-term of the descent spectral sequence for $\thh(\OK;\Fp)$. To simplify the notations, from now on for 
\[
?\in \{z, z_i, \sigma, \sigma_i, u, u_i, v, v_i, t_{z_0-z_1}\},
\]
we denote its image in the mod $p$ reduction by the same symbol.  Moreover, we abusively use $z, z_i$ to denote their images in $\thh_*(\OK/\Sz;\Fp)$ and $\thh_*(\OK/\Szz;\Fp)$ respectively under $p_0$ . Under these notations, we have 
\[
\tp_0(\OK/\Sz;\Fp)=  W(\k)[[z]]\otimes_{\mathbb{Z}}\Fp= \k[[z]]
\] 
and 
\[
\thh_*(\OK/\Sz;\Fp)= \OK[u]\otimes_{\mathbb{Z}}\Fp= (\OK/(p))[u]=\k[z]/(z^{\eK})[u],
\]
where $z$ corresponds to $\overline{\pK}$ under the last identification. Moreover, we have
\[
\thh_*(\OK/\Szz;\Fp)=(\OK\langle t_{z_0-z_1}\rangle \otimes_{\OK}\OK[u_0])\otimes_{\mathbb{Z}}\Fp=(\k[z]/(z_1^{\eK})[u_0]\langle t_{z_0-z_1}\rangle.
\]
Recall that we denote the leading coefficient of $E_K(z)$ by $\mu$ which is typically not $1$. Also recall that $\thh(\OK)$ is an $\OK$-module, so $\thh_*(\OK;\Fp)$ is naturally a $\k$-vector space.
\begin{prop}\label{P:thh-e2-fp}
The following statements are true.
\begin{enumerate}
\item[(1)]
The $k$-vector space $E^2_{0,*}(\thh(\OK);\Fp)$ has a basis given by
\[
 \begin{cases}
   z^lu^n,\quad 1\leq l\leq \eK-1~\text{or}~p\mid \eK n, &\text{if $e_K>1$}\\
	u^n, \quad p\mid n, &\text{if $e_K=1$.}
  \end{cases}
  \]
  \item[(2)]
The $\k$-vector space $E^2_{-1,*}(\thh(\OK);\Fp)$ has a basis given by the family of cycles
\[
 \begin{cases}
   z_1^l(u_0^{n-1}t_{z_0-z_1}-(n-1)E_K'(z_0)u_0^{n-2}t^{[2]}_{z_0-z_1}), \quad 0\leq l\leq \eK-2~\text{or}~p\mid \eK n, &\text{if $e_K>1$}\\
	\sum_{j=1}^l \frac{(n-1)!}{(n-j)!}(-\bar{\mu})^{j}u_0^{n-j}t_{z_0-z_1}^{[j]}, \quad p\mid n, &\text{if $e_K=1$.}
  \end{cases}
  \]
\item[(3)]
For $i\neq 0,-1$, 
$E^2_{i,*}(\thh(\OK);\Fp)=0$.
\end{enumerate}
\end{prop}
\begin{proof}
By  an argument  similar to the proof of Proposition \ref{P:ext-thh}, we get that $E^2(\thh(\OK);\Fp)$ is computed by the complex
\begin{equation}\label{E:thh-fp}
0\rightarrow \thh_*(\OK/\Sz;\Fp)\xrightarrow{(D_0\circ\eta_R)dz}\thh_*(\OK/\Sz;\Fp)dz\rightarrow0.
\end{equation}
This implies (3) immediately. Using (\ref{E:right-unit-u}) and $E_K'(z)\equiv \eK\mu z^{\eK-1} \mod p$,  (\ref{E:thh-fp}) may be identified with 
\begin{equation}\label{E:differential}
0\rightarrow (\k[z]/(z^{\eK}))[u] \xrightarrow{f(u)\mapsto -\eK\bar{\mu} z^{\eK-1}  f'(u)dz} (\k[z]/(z^{\eK}))[u] dz\rightarrow0.
\end{equation}
Then a short computation shows that $H^0$ is the $\k$-vector space with a basis given by 
\[
 \begin{cases}
   z^lu^n,\quad 1\leq l\leq \eK-1~\text{or}~p\mid \eK n, &\text{if $e_K>1$}\\
	u^n, \quad p\mid n, &\text{if $e_K=1$.}
  \end{cases}
  \] 
  and $H^1$ is the $\k$-vector space with a basis given by the family of cocycles 
\[
 \begin{cases}
   z^lu^{n-1}dz, \quad 0\leq l\leq \eK-2~\text{or}~p\mid \eK n, &\text{if $e_K>1$}\\
	u^{n-1}dz, \quad p\mid n,  &\text{if $e_K=1$.}
  \end{cases}
  \]

To compare (\ref{E:thh-fp}) with the cobar complex, for $n\geq1$, set
\begin{equation}\label{E:difference}
u^{(n)}=\sum_{j=1}^n \frac{(n-1)!}{(n-j)!}(-E_K'(z))^{j-1}u_0^{n-j}t_{z_0-z_1}^{[j]}=\frac{u_0^{n}-u_1^{n}}{nE_K'(z)}\in \thh_*(\OK/\Szz).
\end{equation}
It is straightforward to see that $u^{(n)}$ is a cocycle in the cobar complex for $\thh_*(\OK/\Sz)$. Now consider the diagram
\begin{equation}\label{E:thh-cobar}
\xymatrix{
\thh_*(\OK/\Sz;\Fp) \ar[rr]^{(D_0\circ\eta_R)dz} \ar[d]^{\mathrm{id}} && \thh_*(\OK/\Sz;\Fp)dz \ar[d]^{\beta} \\
\thh_*(\OK/\Sz;\Fp) \ar[rr]^{\eta_L-\eta_R}&& \thh_*(\OK/\Szz;\Fp),
}
\end{equation}
where $\beta$ is the $\k[z]$-linear map sending $u^ndz$ to $\overline{u^{(n+1)}}$.

By  (\ref{E:difference}),  it is straightforward to check that (\ref{E:thh-cobar}) is commutative. Thus it gives rise to a morphism from (\ref{E:thh-fp}) to the cobar complex
of $\thh_*(\OK/\Sz;\Fp)$. Note that the right vertical map of (\ref{E:thh-cobar}) is injective. Since both (\ref{E:thh-fp}) and the cobar complex compute the $E^2$-term of the descent spectral sequence, we deduce that (\ref{E:thh-cobar}) induces a quasi-isomorphism. Finally, note that if $e_K>1$, then $E_K'(z)^2=0$ in $\k[z]/(z^{\eK})$. Now the proposition follows. 
\end{proof}

\begin{rem}
The extra complication of $E^2(\thh(\OK;\Fp))$ originates from the ``accidental" filtration clash of the differentials 
\[
z^{m\eK} \mapsto m\eK \bar{\mu} z_1^{m\eK-1}dz
\]
in degree $2m$. To remedy this,  we will refine the Nygaard filtration in what follows. Note that we have reduced the determination of $\tp(\OK)$ to the purely algebraic problem of determining the cohomology of the stack $\mathcal{X}$ (defined in  (\ref{X})).  This is what makes it possible to refine the Nygaard filtration, a possibility that, as far as we know, is not available in homotopy theory.

\end{rem}

\begin{conv}\label{Conv:nygaard}
From now on, we rescale the index of Nygaard filtrations by $2$. That is, $\mathcal{N}^{\geq j}$ 
takes place of $\mathcal{N}^{\geq 2j}$. 
\end{conv}

%Because the graded pieces of the Nygaard filtration are free $\OK$-modules, 
\begin{defn}\label{D:refined}
Let $M$ be a filtered $\tp_0(\OK/\Sz;\mathbb{F}_p)$-module equipped with the filtration $\mathcal{N}^{\geq \bullet}$ indexed by $\mathbb{Z}$. Define a refinement of $\mathcal{N}^{\geq \bullet}$, which is indexed by ${1\over \eK}\mathbb{Z}$, on $M$ by setting 
\[
\mathcal{N}^{\geq j+{m\over\eK}}M = z^{m}\mathcal{N}^{\geq j}M+\mathcal{N}^{\geq j+1}M
\]
for $j\in\mathbb{Z}$ and $0\leq m <\eK$, and call it the \emph{refined filtration} of $\mathcal{N}^{\geq \bullet}$; it is refined in the sense that restriction
along $\mathbb{Z}\subset {1\over \eK}\mathbb{Z}$ gives back the original filtration. Note that under the refined filtrations, $M$ is still a filtered $\tp_0(\OK/\Sz;\mathbb{F}_p)$-module.
\end{defn}
%\quash{
\begin{lem}\label{611}
Let $M$ be a filtered $\tp_0(\OK/\Sz;\mathbb{F}_p)$-module equipped with the filtration $\mathcal{N}^{\geq \bullet}$ indexed by $\mathbb{Z}$. If the filtration $\mathcal{N}^{\geq \bullet}$ is multiplicative, then so is the refined filtration.
\end{lem}
\begin{proof}
%What we need to check is that $$\pK^\eK = 0 \mod p$$ in $\OK$.
This follows from the definition of refined filtrations and the fact that, if $0\leq m_1,m_2 < \eK$ and $m_1+m_2\geq \eK$, then 
$$z^{m_1}z^{m_2}\equiv z^{m_1+m_2 -\eK}\EK(z)\mod p.$$
\end{proof}
%}
In the following, regard $\tp_*(\OK/\Sz^{\otimes\oldbullet};\Fp)$ and $\tc^{-}_*(\OK/\Sz^{\otimes\oldbullet};\Fp)$ as $\tp_0(\OK/\Sz;\Fp)$-modules via the left unit. We call the refined filtration of the Nygaard filtration the \emph{refined Nygaard filtration}. Note that we may refine the Nygaard filtration of $\tp_0(\OK/\Szz;\Fp)$ via both $\eta_L$ and $\eta_R$. However, since 
\[
z_0^m-z_1^m\in \mathcal{N}^{\geq1}\tp_0(\OK/\Szz;\Fp)
\]
for $m\geq1$, we get that both ways end up with the same filtration. Combining Corollary \ref{C:u-ext} and Proposition \ref{P:hopf-thh}, we reach the following result.

\begin{lem} \label{L:ref-nyg}
The associated graded of the refined Nygaard filtration on the  Hopf algebroid 
\[
\left(\tp_0(\OK/\Sz;\Fp),\tp_0(\OK/\Szz;\Fp)\right)
\]
is
\[
(\k[z], \k[z_0]\otimes_{\k}\k\langle t_{z_0-z_1}\rangle),
\]
in which the following holds.
\begin{enumerate}
\item[(1)]If $e_K=1$, then $z_1=z_0-t_{z_0-z_1}$.  If $e_K>1$, then $z_1=z_0$; in this case the Hopf algebroid becomes the Hopf algebra
\[
(\k[z], \k[z]\otimes_{\k}\k\langle t_{z_0-z_1}\rangle).
\]
\item[(2)]
The coproduct $\Delta$ and counit $\varepsilon$ satisfy
\[
\Delta(t^{[i]}_{z_0-z_1}) = \sum_{0\leq j\leq i} t^{[j]}_{z_0-z_1}\otimes t^{[i-j]}_{z_0-z_1}, \quad \varepsilon(t^{[i]}_{z_0-z_1})=0
\]
for all $i\geq0$.
\end{enumerate}
\end{lem}

The refined Nygaard filtration on $E^1$-terms of the descent spectral sequences for $\tp(\OK;\Fp)$ and  $\tc^{-}(\OK;\Fp)$ give rise to the \emph{refined algebraic Tate spectral sequence} 
\[
\tilde{E}_{i,j,k}^{ \frac{1}{\eK}}(\tp(\OK);\Fp)=H^i(\mathrm{Gr}^{k}(\tp_j(\OK/\Sz^{\otimes \oldbullet})))\Rightarrow E_{-i,j}^2(\tp(\OK);\Fp)
\]
and the \emph{refined algebraic homotopy fixed points spectral sequence}
\[
\tilde{E}_{i,j,k}^{ \frac{1}{\eK}}(\tc^-(\OK);\Fp)=H^i(\mathrm{Gr}^{k}(\tc^-_j(\OK/\Sz^{\otimes \oldbullet})))\Rightarrow E^2_{-i,j}(\tc^{-}(\OK);\Fp).
\]
Note that $k$ takes values in ${1\over\eK}\mathbb{Z}$. They are multiplicative spectral sequences in view of Lemma \ref{611}, with $\tilde{E}^r$-terms for all $r\in \frac{1}{\eK}\mathbb{Z}_{\geq1}$, and with differentials
\[
d^r: \tilde{E}_{i,j,k}^{r}\to \tilde{E}_{i+1,j,k+r}^{r}.
\]
Moreover, by Remark \ref{R:tate-n-variable}, Lemma \ref{L:ref-nyg} and functoriality of the Tate spectral sequence, we see that $\tilde{E}^{\frac{1}{e_K}}(\tp(\OK);\Fp)$ may be identified with the cobar complex for  $\k[z][\sigma^{\pm1}]$ with respect to the Hopf algebroid $(\k[z], \k[z_0]\otimes_{\k}\k\langle t_{z_0-z_1}\rangle)$, and $\tilde{E}^{\frac{1}{e_K}}(\tc^-(\OK);\Fp)$ is a truncation of $\tilde{E}^{\frac{1}{e_K}}(\tp(\OK);\Fp)$.

\begin{lem}\label{L:ref-tate-tp0}
The following statements are true.
\begin{enumerate}
\item[(1)]
If $e_K>1$, then
\[
\tilde{E}_{0,j,*}^{1-{1\over\eK}}(\tp(\OK);\Fp)= \k[z]\sigma^j,\quad \tilde{E}_{1,j,*}^{1-{1\over\eK}}(\tp(\OK);\Fp)= \k[z_0]t_{z_0-z_1}\sigma^j.
\]
 Moreover, $d^{1-{1\over \eK}}(z\sigma^j) = t_{z_0-z_1}\sigma^j$.
\item[(2)]
If $e_K=1$, then
\begin{equation*}
\tilde{E}_{0,j,*}^{1}(\tp(\OK);\Fp)= \k[z^p]\sigma^j, \quad \tilde{E}_{1,j,*}^1(\tp(\OK);\Fp)= 
\mathop{\bigoplus}_{n\in p\mathbb{Z}_{>0}} \!\!\k r_n 
\end{equation*}
\text{where $r_n = \sum_{j=1}^n \frac{(n-1)!}{(n-j)!}(-1)^jz_0^{n-j}t_{z_0-z_1}^{[j]}\sigma^j$}.

\item[(3)]
For $i\neq 0,1$, $\tilde{E}_{i,*,*}^{\frac{1}{e_K}}(\tp(\OK);\Fp)=0$.
\end{enumerate}
\end{lem}
\begin{proof}
By functoriality of the Tate spectral sequence, we have 
\[
\sigma_0\sigma_1^{-1}-1\in\mathcal{N}^{\geq1}\tp_0(\OK/\Szz). 
\]
It follows that $\sigma_0=\sigma_1$ in the associated graded of the cobar complex for $\tp_*(\OK/\Sz)$. Therefore we reduce to the case $j=0$. 

By an argument similar to the proof of Proposition \ref{P:ext-thh}, we first see that $\mathrm{Ext}_{ \k[z_0]\otimes_{\k}\k\langle t\rangle}(\k[z], \k[z])$ is computed by the complex
\begin{equation}\label{E:ref-nyg-ext}
0\rightarrow \k[z]\xrightarrow{f(z)\mapsto -f'(z)dz} \k[z] dz\rightarrow0.
\end{equation}
Then we proceed as in the poof of Proposition \ref{P:thh-e2-fp}. We consider the commutative diagram
\begin{equation}\label{E:thh-cobar}
\xymatrix{
\k[z] \ar[rr]^{f(z)\mapsto -f'(z)dz} \ar[d]^{\mathrm{id}} && \k[z]dz\ar[d]^{\beta} \\
\k[z] \ar[rr]^{\eta_L-\eta_R}&&  \k[z_0]\otimes_{\k}\k\langle t\rangle,
}
\end{equation}
where $\beta$ is the $\k[z]$-linear (under $\eta_L$) map sending $z^ndz$ to $\sum_{j=0}^n \frac{n!}{(n-j)!}(-1)^jz_0^{n-j}t_{z_0-z_1}^{[j+1]}$.
By an argument similar to the proof of Proposition \ref{P:thh-e2-fp}, we deduce that it gives rise to an quasi-isomorphism between (\ref{E:ref-nyg-ext}) and the cobar complex. This yields the desired result on cohomology of the cobar complex. Finally, when $e_K>1$, the differential of the cobar complex sends 
\[
z^n\in \mathcal{N}^{\geq\frac{n}{e_K}} \setminus \mathcal{N}^{\geq\frac{n+1}{e_K}}
\] 
to
\[
z^n_1-z_0^n=\sum_{1\leq j\leq n}\binom{n}{j}(z_1-z_0)^jz_1^{n-j},
\] 
which belongs to $\mathcal{N}^{\geq \frac{n}{\eK}+1-\frac{1}{e_K}}\tp_0(\OK/\Szz;\Fp)$. 
It follows that 
\[
\tilde{E}_{*,0,*}^{{1\over\eK}}(\tp(\OK);\Fp)=\tilde{E}_{*,0,*}^{{2\over\eK}}(\tp(\OK);\Fp)=\cdots=\tilde{E}_{*,0,*}^{1-{1\over\eK}}(\tp(\OK);\Fp)
\]
and $d^{1-{1\over \eK}}(z) = t_{z_0-z_1}$.
\end{proof}

\begin{cor}\label{C:e2-tp-tc-}
Both $E^2(\tc^-(\OK);\Fp)$ and $E^2(\tp(\OK);\Fp)$
are concentrated in $E^2_{0,*}$ and $E^2_{-1,*}$. In particular, both the descent spectral sequences for $\tc^-(\OK; \Fp)$ and $\tp(\OK; \Fp)$ collapse at the $E^2$-term.
\end{cor}

\begin{conv}\label{C:u-z}
Motivated by the results of Lemma \ref{L:ref-tate-tp0}, in what follows, denote $t_{z_0-z_1}$ by $dz$. When $e_K=1$, denote
\[
\sum_{j=1}^n \frac{(n-1)!}{(n-j)!}(-1)^jz_0^{n-j}t_{z_0-z_1}^{[j]},
\]
which is formally equal to $\frac{z_0^{n}-z_1^{n}}{n}$ (say, in the fraction field of $\thh(\OK/\Szz)$), by $z_0^{n-1}dz$. 

 \end{conv}
Under Convention \ref{C:u-z}, we may reformulate Lemma \ref{L:ref-tate-tp0}(1), (2) as follows.

\begin{cor}\label{C:ref-tate-tp0-ref}
For $e_K>1$, we have
\[
\tilde{E}_{*,j,*}^{1-{1\over\eK}}(\tp(\OK);\Fp)= \k[z]\sigma^j\oplus \k[z_0] \sigma^jdz,
\]
and $d^{1-\frac{1}{e_K}}(z\sigma^j)=\sigma^jdz$. For $e_K=1$, we have
\[
\tilde{E}_{*,j,*}^{1}(\tp(\OK);\Fp)= \k[z^p]\sigma^j\oplus z_0^{p-1}\k[z_0^p]\sigma^jdz.
\]
\end{cor}

In the rest of this section, we will determine higher differentials of the refined algebraic Tate spectral sequence. 
We will prove in Proposition \ref{P:non-zero-stem}
that
for $n\geq0, j\in\mathbb{Z}, l=v_p(n-{p\eK j\over p-1})$, $\tilde{\mu} = -{\mu^p\over \delta(\EK(z_0))}$ 
 and $n'\equiv p^{-l}(n-{p\eK j\over p-1})\mod p$, we have 
\begin{equation}\label{E:ref-tate-dif0}
d^{{p^{l+1}-1\over p-1}-{1\over\eK}}(z^n\sigma^j) =n'\bar{\tilde{\mu}}^{{p^{l}-1\over p-1}} z_0^{p\eK{p^l-1\over p-1}+n-1}\sigma^jdz,
\end{equation}
which accounts for all the nontrivial refined algebraic Tate differentials.

%We first treat the case of $0$-stems. (
Recall that the algebraic Tate spectral sequences can be considered one $j$ at a time because the differentials respect $j$. We will begin with the case $j=0$ before treating the case of a general $j$. In our treatment all that is used is the definition of the differentials in the spectral seuqence associated with a filtered chain complex and some tricks in linear algebra.

In the following, when the context is clear, for $j\in\mathbb{Z}_{\geq0}$, we will simply denote $\mathcal{N}^{\geq j}\tp_0(\OK/\Szz)$ by $\mathcal{N}^{\geq j}$. For $r\in \frac{1}{\eK}\mathbb{Z}_{\geq0}$, we denote $\mathcal{N}^{\geq r}\tp_0(\OK/\Szz;\mathbb{F}_p)$ by $\mathcal{N}^{\geq r}$, and denote by $(p, \mathcal{N}^{\geq r})$ the preimage of $\mathcal{N}^{\geq r}\tp_0(\OK/\Szz;\mathbb{F}_p)$ under the natural projection 
\[
\tp_0(\OK/\Szz)\to \tp_0(\OK/\Szz;\mathbb{F}_p).
\]
For $a\in \tp_0(\OK/\Szz)$ (resp.  $a\in \tp_0(\OK/\Szz;\mathbb{F}_p)$), we denote by $\nu(a)$ the smallest $j\in \mathbb{Z}_{\geq0}$ (resp. $r\in\frac{1}{e_K}\mathbb{Z}_{\geq0}$) such that $a\in \mathcal{N}^{\geq j}$ (resp. $a\in \mathcal{N}^{\geq r}$). Since the associated graded algebras are integral in both cases, we have $\nu(ab)=\nu(a)+\nu(b)$.  

Recall $f^{(0)}=z_0-z_1$, put $\xi_0=-\delta(f^{(0)})/f^{(0)}$.

\begin{lem}\label{L:xi0}
We have 
$\xi_0 \in \tp_0(\OK/\Szz)$.
Moreover, 
\[
\xi_0 \equiv z_0^{p-1}\mod (p, \mathcal{N}^{\geq \frac{p-2}{\eK} + 1}).
\]
In particular, $\xi_0\in (p, \mathcal{N}^{\geq \frac{p-1}{\eK} })$.
\end{lem}
\begin{proof}
For the first claim, we have 
\begin{eqnarray*}
\delta(f^{(0)}) &=& {\varphi(f^{(0)}) -(f^{(0)})^p \over p}\\
&= &{z_0^p - (z_0-f^{(0)})^p - (f^{(0)})^p\over p}\\
&= &-f^{(0)}(z_0^{p-1} - {p-1\over 2}z_0^{p-2}f^{(0)}+\dots+((-1)^{p}+1){(f^{(0)})^{p-1}\over p}).
\end{eqnarray*}
Note that $\frac{(-1)^{p}+1}{p}\in\mathbb{Z}$. Hence 
\[
\xi_0=z_0^{p-1} - {p-1\over 2}z_0^{p-2}f^{(0)}+\dots+((-1)^{p}+1){(f^{(0)})^{p-1}\over p}
\]
belongs to $\tp_0(\OK/\Szz)$. For $1\leq i\leq p-1$, ${p-1-i\over \eK} + i \geq {p-2\over \eK} +1$. Thus for such $i$, $z_0^{p-1-i}(f^{(0)})^{i}\in (p,\mathcal{N}^{\geq \frac{p-2}{\eK} + 1})$. This implies that 
$\xi_0-z_0^p\in (p, \mathcal{N}^{\geq \frac{p-2}{\eK} + 1})$, yielding the second claim. 
\end{proof}
 
Recall that we put $\tilde{\mu} = -{\mu^p\over \delta(\EK(z_0))}$.

\begin{lem}\label{L:z1p-z2p}
We have 
\[
\varphi(f^{(0)}) \equiv \tilde{\mu} z_0^{p\eK+p-1}f^{(0)} \mod 
(p, \mathcal{N}^{\geq 2p}).
\]
\end{lem}
\begin{proof}
Recall that
$h\varphi(\EK(z_0)) = \varphi(f^{(0)})$. Note that 
\[
\varphi(\EK(z_0)) \equiv \mu^p z_0^{p\eK} \mod p.
\]
Thus 
\begin{equation}\label{E:cong-p}
\varphi(f^{(0)}) \equiv \mu^p z_0^{p\eK}h \mod p.
\end{equation}
On the other hand, using (\ref{E:l-l+1}) for $l=0$, we have
\begin{equation}\label{E:f1-f0}
f^{(1)} = \delta(f^{(0)}) - h \delta(\EK(z_0)).
\end{equation}
Since $f^{(1)}\in\mathcal{N}^{\geq p}$, we get
\[
h\equiv \delta(f^{(0)})/\delta(\EK(z_0)) \mod \mathcal{N}^{\geq p}.
\]
Combining this with Lemma \ref{L:xi0} and the fact that $p\eK+ \frac{p-2}{\eK} + 2\geq 2p$, we deduce that
\[
\tilde{\mu} z_0^{p\eK+p-1}f^{(0)}\equiv\tilde{\mu} \xi_0z_0^{p\eK}f^{(0)}=\mu^p z_0^{p\eK}\delta(f^{(0)})/\delta(\EK(z_0))\equiv \mu^p z_0^{p\eK}h\equiv \varphi(f^{(0)})\mod (p,\mathcal{N}^{\geq 2p}),
\]
concluding the lemma.
\end{proof}

\begin{lem}\label{L:cong-odd}
Suppose $p>2$ and $\eK>1$. Then for $l\geq1$,
\begin{equation}\label{E:odd-cong}
\varphi^{l}(f^{(0)})\equiv \tilde{\mu}^{p^l-1\over p-1} z_0^{(p\eK+p-1){p^l-1\over p-1}}f^{(0)} \mod (p, \mathcal{N}^{\geq p^l(1+{1\over p-1}+{1\over\eK})})
\end{equation}
\end{lem}
\begin{proof}
We will establish the lemma by induction on $l$. The case $l=1$ follows from Lemma \ref{L:z1p-z2p} and the inequality ${1\over\eK}+{p\over p-1}\leq {1\over2}+{3\over 2}= 2$.

Now suppose the claim holds for some $l\geq1$. Raising both sides of  (\ref{E:odd-cong}) to the $p$-th power,  we get
\begin{equation}\label{E:odd-ind-1}
\varphi^{l+1}(f^{(0)}) \equiv \tilde{\mu}^{p^{l+1}-p\over p-1} z_0^{(p\eK+p-1){p^{l+1}-p\over p-1}}\varphi(f^{(0)}) \mod (p, \mathcal{N}^{\geq p^{l+1}(1+{1\over p-1}+{1\over\eK})}).
\end{equation}
Using Lemma \ref{L:z1p-z2p} again, we have 
\begin{equation}\label{E:odd-ind-2}
\tilde{\mu}^{p^{l+1}-p\over p-1} z_0^{(p\eK+p-1){p^{l+1}-p\over p-1}}\varphi(f^{(0)})\equiv  \tilde{\mu}^{{p^{l+1}-1\over p-1}} z_0^{(p\eK+p-1){p^{l+1}-1\over p-1}}f^{(0)}  \mod (p,\mathcal{N}^{\geq {p^{l+2}-p^2\over p-1}+{p^{l+1}-p\over\eK}+2p}).
\end{equation}
On the other hand, it is straightforward to see that
\begin{equation}\label{E:odd-ind-3}
{p^{l+2}-p^2\over p-1}+{p^{l+1}-p\over\eK}+2p\geq p^{l+1}(1+{1\over p-1}+{1\over\eK}).
\end{equation}
Putting (\ref{E:odd-ind-1}), (\ref{E:odd-ind-2}) and (\ref{E:odd-ind-3}) together, we complete the inductive step.
\end{proof}

\begin{lem}\label{L:even} For $p=2$ and $l\geq1$, we have
\[
(f^{(1)})^{2^{l}} \in (2,\mathcal{N}^{\geq 2^{l+1}(1+{1\over4})}).
\]
\end{lem}
\begin{proof}
Recall that by construction, we have
\[
2f^{(2)} =-(f^{(1)})^2 + \delta^2(h)\EK(z_0)^4.
\]
By Lemma \ref{L:delta-h}, $\delta^2(h) \in \mathcal{N}^{\geq1}$.
It follows that  
\[
(f^{(1)})^2\in (2, \mathcal{N}^{\geq5}).
\]
We thus conclude by raising to the $2^{l-1}$-th power.
\end{proof}

\begin{lem}\label{L:cong-even}
Suppose $p=2$ and $\eK>3$. Then for $l\geq1$,
\[
\varphi^{l}(f^{(0)}) \equiv \tilde{\mu}^{2^l-1}  z_0^{({2^l-1})(2\eK+1)}f^{(0)} \mod (2, \mathcal{N}^{\geq 2^l(2+{1\over\eK})-\frac{2}{\eK}}).
\]
\end{lem}
\begin{proof}
We proceed by induction on $l$. The case $l=1$ follows from Lemma \ref{L:z1p-z2p}. Now suppose the claim holds for some $l\geq1$. 
Using (\ref{E:cong-p}), (\ref{E:f1-f0}), we first have 
\[
\varphi(f^{(0)})\equiv \mu^2z_0^{2\eK}h \equiv {\tilde{\mu}z_0^{2\eK}(\xi_0f^{(0)}+f^{(1)}) } \mod 2
\]
Raising to the power of $2^{l}$, we get
$$\varphi^{l+1}(f^{(0)})\equiv \tilde{\mu}^{2^{l}}z_0^{2^{l+1}\eK}(\xi_0^{2^{l}}\varphi^{l}(f^{(0)})+(f^{(1)})^{2^{l}}) \mod 2.$$
By the inductive hypothesis, we have
$$\varphi^{l}(f^{(0)}) \equiv  \tilde{\mu}^{2^{l}-1}  z_0^{({2^{l}-1})(2\eK+1)}f^{(0)} \mod (2, \mathcal{N}^{\geq 2^{l}(2+{1\over\eK})-\frac{2}{\eK}}).$$
It follows that
\begin{equation*}\label{E:even-phik}
\varphi^{l}(f^{(0)})\in (2, \mathcal{N}^{\geq (2^{l}-1)(2+{1\over\eK})+1}).
\end{equation*}
On the other hand, using Lemma \ref{L:xi0}, we get
\begin{equation*}\label{E:even-cong}
\xi_0^{2^{l}}\equiv z_0^{2^l} \mod (2, \mathcal{N}^{\geq 2^{l}}).
\end{equation*}
Putting these together, we deduce that
$$  \tilde{\mu}^{2^{l}}z_0^{2^{l+1}\eK}\xi_0^{2^{l}}\varphi^{l}(f^{(0)})\equiv \tilde{\mu}^{2^{l+1}-1} z_0^{2^{l+1}\eK+2^l}\varphi^l(f^{(0)})\mod (2, \mathcal{N}^{\geq (2^{l}-1)(2+{1\over\eK})+2^{l+1}+2^l+1}).$$
and
\[
\tilde{\mu}^{2^{l+1}-1} z_0^{2^{l+1}\eK+2^l}\varphi^l(f^{(0)})\equiv \tilde{\mu}^{2^{l+1}-1} z_0^{({2^{l+1}-1})(2\eK+1)}f^{(0)}\mod (2, \mathcal{N}^{\geq 2^{l+1}(2+{1\over\eK})-\frac{2}{\eK}}).
\]
Clearly $(2^{l}-1)(2+{1\over\eK})+2^{l+1}+2^l+1> 2^{l+1}(2+{1\over\eK})-\frac{2}{\eK}$. Hence we get
\begin{equation}\label{E:even-phik}
\tilde{\mu}^{2^{l}}z_0^{2^{l+1}\eK}\xi_0^{2^{l}}\varphi^{l}(f^{(0)})\equiv \tilde{\mu}^{2^{l+1}-1} z_0^{({2^{l+1}-1})(2\eK+1)}f^{(0)}\mod (2, \mathcal{N}^{\geq 2^{l+1}(2+{1\over\eK})-\frac{2}{\eK}}).
\end{equation}
 Finally, by previous lemma, we have
\begin{equation}\label{E:even-f1}
\tilde{\mu}^{2^l}z_0^{2^{l+1}\eK}(f^{(1)})^{2^{l}} \in (2,\mathcal{N}^{\geq 2^{l+1}+ 2^{l+1}(1+{1\over4})}) \subset (2,\mathcal{N}^{\geq  2^{l+1}(2+{1\over\eK})-\frac{2}{\eK}}).
\end{equation}
Combining (\ref{E:even-phik}) and (\ref{E:even-f1}), we conclude the inductive step. 
\end{proof}

\begin{prop}\label{P:stem0}
Suppose $p>2, \eK>1$ or $p=2, \eK>3$.
Then for $n\geq0, l=v_p(n), n'=\frac{n}{p^l}$, the refined algebraic Tate differential satisfies 
\begin{equation}\label{E:tate-diff-stem0}
d^{{p^{l+1}-1\over p-1}-{1\over\eK}}(z^{n})=n'\bar{\tilde{\mu}}^{{p^{l}-1\over p-1}}z_0^{p\eK{p^{l}-1\over p-1}+n-1}dz,
\end{equation} 
which is non-zero in  $\tilde{E}_{1,0, {p^{l+1}-1\over p-1}+\frac{n-1}{e_K}}^{{p^{l+1}-1\over p-1}}$. Moreover, the exponents of $z_0$ in the targets of (\ref{E:tate-diff-stem0})  are all different. Consequently, these are all the nontrivial refined algebraic Tate differentials.
\end{prop}
\begin{proof}
First note that 
\[
\nu(z_0^{p\eK{p^{l}-1\over p-1}+n-1}(z_0-z_1))={p^{l+1}-1\over p-1}+{n-1\over \eK}.
\] 
On the other hand, since $z_0^{p^l}-z_1^{p^l}\equiv \varphi^{l}(f^{(0)}) \mod p$ in $\tp_0(\OK/\Szz)$, by Lemma \ref{L:cong-odd} and Lemma \ref{L:cong-even}, we get 
\begin{equation*}\label{E:nu-phi-zk}
\nu(z_0^{p^l}-z_1^{p^l}-\bar{\tilde{\mu}}^{{p^{l}-1\over p-1}}z_0^{p\eK{p^{l}-1\over p-1}+p^l-1}(z_0-z_1))>{p^{l+1}-1\over p-1}+{p^l-1\over \eK}.
\end{equation*}

Hence
\begin{equation*}\label{E:phi-zk}
\nu(z_0^{p^l}-z_1^{p^l})=\nu(\bar{\tilde{\mu}}^{{p^{l}-1\over p-1}}z_0^{p\eK{p^{l}-1\over p-1}+p^l-1}(z_0-z_1))={p^{l+1}-1\over p-1}+{p^l-1\over \eK}. 
\end{equation*}
Write 
\[
z_0^n-z_1^n=z_0^{n'p^l}-z_1^{n'p^l}=-\sum_{0\leq i\leq n'-1}(-1)^{n'-i}\binom{n'}{i}z_0^{ip^l}(z_0^{p^l}-z_1^{p^l})^{n-i}.
\]
It is straightforward to see 
\[
{p^{l+1}-1\over p-1}+{n-1\over \eK}=\nu(z_0^{(n'-1)p^l}(z_0^{p^l}-z_1^{p^l}))<\nu(z_0^{ip^l}(z_0^{p^l}-z_1^{p^l})^{n-i})
\] 
for $i\leq n-2$. Note that ${p^{l+1}-1\over p-1}+{n-1\over \eK}=({p^{l+1}-1\over p-1}-{1\over \eK})+{n\over \eK}$. We thus deduce that
\[
d^{{p^{l+1}-1\over p-1}-{1\over\eK}}(z^{n}) = n'z_0^{(n'-1)p^l}(z_0^{p^l}-z_1^{p^l})= n'\bar{\tilde{\mu}}^{{p^{l}-1\over p-1}}z_0^{p\eK{p^{l}-1\over p-1}+n-1}dz.
\]

It remains to show that the exponents of $z_0$ in the targets of  (\ref{E:tate-diff-stem0}) are all different; note that this will automatically imply that  the right hand side of (\ref{E:tate-diff-stem0}) is non-zero. Put $\tilde{n}=p\eK{p^{l}-1\over p-1}+n$. Since $v_p(n)=l$, we get $l=v_p(\tilde{n}+\frac{pe_K}{p-1})$. Consequently, $n$ is uniquely determined by $\tilde{n}$. This yields the desired result. 
\end{proof}

Now we treat the remaining cases. The strategy is to compare them with the known cases. 

\begin{prop} \label{P:stem0-all}
The result of Proposition \ref{P:stem0} holds for all $p$ and $e_K$.
\end{prop}
\begin{proof}
Choose an integer $m>3$ coprime to $p$, and let $K' = K(\varpi^{1\over m}_K)$; the ramification index of $K'$ is $e_{K'}=me_{K}$, and the corresponding Eisenstein polynomial for 
$\varpi^{1\over m}_K$ is $E_{K'}(z) = \EK(z^m)$. Now the commutative diagram
$$\xymatrix{
\Sz \ar[rr]^{z\mapsto z^m} \ar[d]^{z\mapsto\pK} && \Sz \ar[d]^{z\mapsto \varpi^{1\over m}_K} \\
\OK \ar[rr] && \mathcal{O}_{K'}
}
$$
induces a map of cosimplicial cyclotomic spectra
\[
T_m: \tp(\OK/\Sz^{\otimes\oldbullet};\Fp) \rightarrow \tp(\mathcal{O}_{K'}/\Sz^{\otimes\oldbullet};\Fp).
\]
Define the "less refined" Nygaard filtration on $\tp_*(\mathcal{O}_{K'}/\Sz^{\otimes\oldbullet};\Fp)$ to be the filtration 
$\mathcal{N}^{\geq r} \tp_*(\mathcal{O}_{K'}/\Sz^{\otimes\oldbullet};\Fp)$ for $r\in\frac{1}{e_K}\mathbb{Z}_{\geq0}$, which in turn induces the "less refined" algebraic Tate spectral sequence $\tilde{E'}(\tp(\mathcal{O}_{K'});\Fp)$. Clearly $T_m$ is compatible with filtrations. Thus it induces a morphism of spectral sequences
\[
T_m: \tilde{E}(\tp(\mathcal{O}_{K});\Fp)\to \tilde{E'}(\tp(\mathcal{O}_{K'});\Fp).
\]

By an argument similar to the proof of Proposition \ref{P:thh-e2-fp} and Lemma \ref{L:ref-tate-tp0}, we first obtain that if $e_K>1$, then $\tilde{E'}^{\frac{1}{e_K}}_{*,0,*}(\tp(\mathcal{O}_{K'});\Fp)$ is isomorphic to $\k[z]\oplus \k[z_0] dz$, where $dz$ denotes $t_{z_0-z_1}$. If $e_K=1$, then $\tilde{E'}^{\frac{1}{e_K}}_{0,0,*}(\tp(\mathcal{O}_{K'});\Fp)$ is the $\k$-vector space with a basis 
$\{z^n | m\nmid n~\text{or}~p\mid  n\}$,
and $\tilde{E'}^{\frac{1}{e_K}}_{1,0,*}(\tp(\mathcal{O}_{K'});\Fp)$ is the $\k$-vector space with a basis  given by the family of cycles $\{z_0^{n}dz| m\nmid n+1~\text{or}~p\mid  n+1\}$, where 
$z_0^{n}dz$ denotes
\[
   z_0^s((z_0^m)^{k-1}t_{z_0-z_1}-(k-1)mz_0^{m-1}(z_0^m)^{k-2}t^{[2]}_{z_0-z_1}), \quad 0\leq s\leq m-1~\text{and}~s+(k-1)m=n, 
   \]
which is formally equal to $\frac{z_0^{n+m}-z_1^{n+m}}{kmz_0^{m-1}}$; for $j\neq0,1$, $\tilde{E'}^{\frac{1}{e_K}}_{j,0,*}(\tp(\mathcal{O}_{K'});\Fp)=0$. Under our convention of notations, it is straightforward to verify 
\begin{equation}\label{E:tm}
T_m(z^n)=z^{mn}, \qquad T_m(z_0^ndz)=mz_0^{mn+m-1}dz;
\end{equation}
note that right hand side of the second equality is just formally equal to $z_0^{mn}dz^m$. Combining with Lemma \ref{L:ref-tate-tp0} and Corollary \ref{C:ref-tate-tp0-ref}, we see that 
\[
T_m: \tilde{E}_{*,0,*}^{\frac{1}{e_K}}(\tp(\mathcal{O}_{K});\Fp)\to \tilde{E'}_{*,0,*}^{\frac{1}{e_K}}(\tp(\mathcal{O}_{K'});\Fp)
\]
is injective. 
To proceed, we need the following result. 
\begin{lem}\label{L:L-in-P}
 For $n\geq0, l=v_p(n)$, where $l\geq1$ if $e_K=1$, $n'=\frac{n}{p^l}$, the natural projection
\[
\phi: \tilde{E'}^{\frac{1}{e_K}}_{1,0, {p^{l+1}-1\over p-1}+\frac{n-1}{e_K}}(\tp(\mathcal{O}_{K'});\Fp)\to \frac{\tilde{E'}^{\frac{1}{e_K}}_{1,0, {p^{l+1}-1\over p-1}+\frac{n-1}{e_K}}(\tp(\mathcal{O}_{K'});\Fp)}{\oplus_{i\neq pm\eK{p^{l}-1\over p-1}+mn-1}\k z_0^i dz}\cong \k z_0^{pm\eK{p^{l}-1\over p-1}+mn-1}dz
\]
factors through $\tilde{E'}_{1,0, {p^{l+1}-1\over p-1}+\frac{n-1}{e_K}}^{{p^{l+1}-1\over p-1}}$. Moreover,
\[
d^{{p^{l+1}-1\over p-1}-{1\over\eK}}(z^{mn})\in \tilde{E'}_{1,0, {p^{l+1}-1\over p-1}+\frac{n-1}{e_K}}^{{p^{l+1}-1\over p-1}}
\] 
maps to $n'\bar{\tilde{\mu}}^{{p^{l}-1\over p-1}}z_0^{pm\eK{p^{l}-1\over p-1}+mn-1}dz$ via this projection.
In particular, $d^{{p^{l+1}-1\over p-1}-{1\over\eK}}(z^{mn})$ is non-zero.
\end{lem}
\begin{proof}
By first half of Proposition \ref{P:stem0}, if $z^t\in \tilde{E'}^{\frac{1}{e_K}}_{0,0, \frac{t}{e_K}}(\tp(\mathcal{O}_{K'});\Fp)$ has non-trivial contribution to $\tilde{E'}^{\frac{k-t}{e_K}}_{1,0, \frac{k-1}{e_K}}(\tp(\mathcal{O}_{K'});\Fp)$, then 
\begin{equation}\label{E:L-in-P}
{p^{l'+1}-1\over p-1}+\frac{t-1}{me_K}=\frac{k-1}{e_K}+\frac{s}{me_K}~ \text{for some}~ 0\leq s \leq m-1,
\end{equation}
where $l'=v_p(t)$. By the second half of Proposition \ref{P:stem0}, $t$ is uniquely determined by $(k, s)$. In particular, if 
\[
k=e_K{p^{l+1}-1\over p-1}+n, \quad s=m-1,
\] 
then $t$ has to be equal to $mn$. Moreover, when (\ref{E:L-in-P}) holds, we see from the argument of Proposition \ref{P:thh-e2-fp} and Lemma \ref{L:ref-tate-tp0} that the image of $z^t$ in  $\tilde{E'}^{\frac{1}{e_K}}_{1,0, \frac{k-1}{e_K}}(\tp(\mathcal{O}_{K'});\Fp)$ is contained in the subspace generated by the cycles 
$z_0^{m(k-1)+s'}dz, 0\leq s'\leq s$.  Putting these together, we deduce that 
\[
\ker(\tilde{E'}^{\frac{1}{e_K}}_{1,0, {p^{l+1}-1\over p-1}+\frac{n-1}{e_K}}(\tp(\mathcal{O}_{K'});\Fp)\to \tilde{E'}_{1,0, {p^{l+1}-1\over p-1}+\frac{n-1}{e_K}}^{{p^{l+1}-1\over p-1}}(\tp(\mathcal{O}_{K'});\Fp)
\]
is contained in the subspace generated by $z_0^{pm\eK{p^{l}-1\over p-1}+mn-s}dz, 2\leq s\leq m$, yelding the first half of the lemma. Using (\ref{E:tate-diff-stem0}), we conclude the second half of the lemma. 
\end{proof}
~\\
Now we prove the proposition. We first show that $z^{n}\in\tilde{E}^{\frac{1}{e_K}}_{0,0,\frac{n}{e_K}}(\tp(\mathcal{O}_{K});\Fp)$ survives to the $\tilde{E}^{{p^{l+1}-1\over p-1}}$-term.  We do this by induction. Suppose $z^{n}$ survives to some $\tilde{E}^{r}$-term with $\frac{1}{e_K}\leq r<{p^{l+1}-1\over p-1}$. That is, 
\[
d(z^n)\in\mathcal{N}^{\geq r+\frac{n-1}{e_K}}\tp_0(\mathcal{O}_{K}/\Szz;\Fp).
\]
Since $T_m(z^n)=z^{mn}$, which survives to the $\tilde{E'}^{{p^{l+1}-1\over p-1}}$-term by Lemma \ref{L:L-in-P}, we have
\[
T_m(d(z^n))=d(T_m(z^n))\in\mathcal{N}^{\geq r+\frac{n}{e_K}}\tp_0(\mathcal{O}_{K'}/\Szz;\Fp).
\]
Then the injectivity of $\tilde{E}_{1,0,r+\frac{n-1}{e_K}}^{\frac{1}{e_K}}(\tp(\mathcal{O}_{K});\Fp)\to \tilde{E'}_{1,0,r+\frac{n-1}{e_K}}^{\frac{1}{e_K}}(\tp(\mathcal{O}_{K'});\Fp)$ implies that $d(z^n)=d(\alpha)$ for some $\alpha\in \mathcal{N}^{r+\frac{n-1}{e_K}}\tp_0(\mathcal{O}_{K};\Fp)$, which is 
\[
\mathcal{N}^{\geq r+\frac{n-1}{e_K}}\tp_0(\mathcal{O}_{K};\Fp)/\mathcal{N}^{\geq r+\frac{n}{e_K}}\tp_0(\mathcal{O}_{K};\Fp).
\] 
Now 
\[
d(T_m(\alpha))=T_m(d(\alpha))=T_m(d(z^n))=0\in\mathcal{N}^{r+\frac{n-1}{e_K}}\tp_0(\mathcal{O}_{K'}/\Szz;\Fp),
\]
we get $T_m(\alpha)\in \tilde{E'}_{0,0,r+\frac{n-1}{e_K}}^{\frac{1}{e_K}}(\tp(\mathcal{O}_{K'});\Fp)$. By the explicit description of 
$ \tilde{E}_{0,0,*}^{\frac{1}{e_K}}(\tp(\mathcal{O}_{K});\Fp)$ and $\tilde{E'}_{0,0,*}^{\frac{1}{e_K}}(\tp(\mathcal{O}_{K'});\Fp)$, we conclude $\alpha\in \tilde{E'}_{0,0,r+\frac{n-1}{e_K}}^{\frac{1}{e_K}}(\tp(\mathcal{O}_{K});\Fp)$. Thus 
\[
d(z^n)=d(\alpha)=0\in \mathcal{N}^{r+\frac{n-1}{e_K}}\tp_0(\mathcal{O}_{K}/\Szz;\Fp),
\] 
yielding 
\[
d(z^n)\in \mathcal{N}^{\geq r+\frac{n}{e_K}}\tp_0(\mathcal{O}_{K}/\Szz;\Fp).
\]
Once we know $z^n$ survives to the $\tilde{E}^{\frac{p^{l+1}-1}{p-1}}$-term, since $\tilde{E}^{\frac{1}{e_K}}_{1,0, {p^{l+1}-1\over p-1}+\frac{n-1}{e_K}}(\tp(\OK;\Fp))$ is generated by $z_0^{p\eK{p^{l}-1\over p-1}+n-1}dz$, we may suppose 
\[
d^{{p^{l+1}-1\over p-1}-{1\over\eK}}(z^{n})=\lambda z_1^{p\eK{p^{l}-1\over p-1}+n-1}dz.
\] 
Applying the second half of Lemma \ref{L:L-in-P}, we get
  \[
\lambda=n'\bar{\tilde{\mu}}^{{p^{l}-1\over p-1}}.
 \]
The rest is the same as in the proof of Proposition \ref{P:stem0}.
\end{proof}

\begin{rem}\label{R:ref-less-ref}
In fact, employing the result of Proposition \ref{P:stem0-all} in the argument of Lemma \ref{L:L-in-P} will enable us to prove the following fact: for $r\in\frac{1}{e_K}\mathbb{Z}_{\geq1}\cup\{\infty\}$, if $\tilde{E'}_{1,0,\frac{k-1}{e_K}+1}^{r}(\tp(\mathcal{O}_{K});\Fp)$ is non-zero, that is $z_0^{k-1}dz$ is not in the image of $d^{r-\frac{1}{e_K}}$, then the natural projection
\[
\tilde{E'}^{\frac{1}{e_K}}_{1,0, \frac{k-1}{e_K}+1}(\tp(\mathcal{O}_{K'});\Fp)\to \frac{\tilde{E'}^{\frac{1}{e_K}}_{1,0, \frac{k-1}{e_K}+1}(\tp(\mathcal{O}_{K'});\Fp)}{\oplus_{i\neq mk-1}\k z_0^i dz}\cong \k z_0^{mk-1}dz
\]
factors through $\tilde{E'}_{1,0,\frac{k-1}{e_K}+1}^{r}(\tp(\mathcal{O}_{K'});\Fp)$. In particular, $T_m(z_1^{k-1}dz)$ is non-zero in $\tilde{E'}_{1,0,\frac{k-1}{e_K}+1}^{r}(\tp(\mathcal{O}_{K'});\Fp)$. 
\end{rem}

Next we investigate the differentials on non-zero stems. To this end, put 
\begin{equation*}\label{E:sigma1-sigma2}
\epsilon=\sigma_0\sigma_1^{-1},\quad \epsilon_0={\varphi(\EK(z_0))\over \varphi(\EK(z_1))};
\end{equation*}
by Remark \ref{R:phi-z1z2}, the latter is well-defined.  By the functoriality of Tate spectral sequence, we have
\[
\epsilon\in 1+\mathcal{N}^{\geq1}.
\] 
Using Theorem \ref{T:rtp}(6), we get
\begin{equation}\label{E:z1-z2}
{\epsilon\over\varphi(\epsilon)}={\varphi(\sigma^{-1}_0)\sigma_0\over\varphi(\sigma_1^{-1})\sigma_1}={\varphi(v_0)\varphi(u_0)\over\varphi(v_1)\varphi(u_1)}=\epsilon_0.
\end{equation}
Let $\bar{\epsilon}, \bar{\epsilon}_0$ be the images of $\epsilon, \epsilon_0$ in $\tp_0(\OK/\Szz;\mathbb{F}_p)$ respectively. 
It follows that 
\[
\bar{\epsilon}_0=\bar{\epsilon}^{1-p}\equiv 1 \mod \mathcal{N}^{\geq1}.
\]
Then it is straightforward to see that for $i\geq0$,
\begin{equation}\label{E:epsilon0}
\bar{\epsilon}^{p^{i}}_0\equiv 1 \mod \mathcal{N}^{\geq p^{i}},
\end{equation}
and 
\begin{equation}\label{E:epsilon}
\prod_{i=0}^\infty \bar{\epsilon}_0^{p^i}=\bar{\epsilon},
\end{equation}
where the LHS takes limit under the $\mathcal{N}$-topology. 

\begin{lem}\label{L:non-zero-stem}
For $r\in\frac{1}{\eK}\mathbb{N}$,   $j\in\mathbb{Z}$, $k,m\in\mathbb{N}$ such that 
\[
p^k>j, \quad \min\{p^m, p^k\}>r,
\] 
we have
$$z^{(p^k-j)\eK{p^{m+1}-p\over p-1}}\sigma^j\in \tilde{E}^{1-{1\over\eK}}(\tp(\OK);\mathbb{F}_p)$$
survives to the $\tilde{E}^{r}$-term.

%$$\sigma^k \over \prod_{i=1}^n \varphi^i(\EK(x))^{k}$$
%suvives to the $E_{l(n)}$ term in the mod $p$ algebraic Tate specral sequence, with $l(n)\rightarrow\infty$ as $n\rightarrow\infty$.

%When $k>0$, the precise meaning is this:

%For sufficiently large $j$, 
%t$$  (\prod_{i=1}^n \varphi^i(\EK(x)))^{p^j} \sigma^k \over \prod_{i=1}^n \varphi^i(\EK(x))^{k}$$ suvives to the $E_{l(n)}$ term.
\end{lem}
\begin{proof}
We consider 
\[
\alpha={ (\prod_{i=1}^m \varphi^i(\EK(z)/\mu))^{p^k-j}} \sigma^j \in \tp_{2j}(\OK/\Sz).
\]
Clearly $\alpha$ is a lift of $z^{(p^k-j)\eK{p^{m+1}-p\over p-1}}\sigma^j$. 
We have
\[
\eta_L(\alpha) = { (\prod_{i=1}^m \varphi^i(\EK(z_0)/\mu))^{p^k-j}} \sigma_0^j
\]
and 
\begin{eqnarray*}
\eta_R(\alpha)=  { (\prod_{i=1}^m \varphi^i(\EK(z_1)/\mu))^{p^k-j}} \sigma_1^j=\eta_L(\alpha)\epsilon^{-j}\prod_{i=0}^{m-1}\varphi^i(\epsilon_0)^{j-p^k}=\eta_L(\alpha)(\epsilon^{-1}\prod_{i=0}^{m-1}\varphi^i(\epsilon_0))^{j}\prod_{i=0}^{m-1}\varphi^i(\epsilon_0)^{-p^k}.
\end{eqnarray*}
By (\ref{E:epsilon0}) and (\ref{E:epsilon}), we deduce that
\[
\bar{\epsilon}^{-1}\prod_{i=0}^{m-1}\varphi^i(\bar{\epsilon}_0)\equiv1 \mod  \mathcal{N}^{\geq p^{m}}
\]
and
\[
\prod_{i=1}^m\varphi^i(\bar{\epsilon}_0)^{-p^k}\equiv1 \mod  \mathcal{N}^{\geq p^{k}}.
\]
It follows that $\eta_L(z^{(p^k-j)\eK{p^{m+1}-p\over p-1}}\sigma^j)-\eta_R(z^{(p^k-j)\eK{p^{m+1}-p\over p-1}}\sigma^j)\in\mathcal{N}^{\geq{(p^k-j){p^{m+1}-p\over p-1}}+\min\{p^m, p^k\}}$, concluding the lemma. 
\end{proof}

\begin{prop}\label{P:non-zero-stem}
%Suppose $p>2, \eK>1$ or $p=2, \eK>3$. 
For $n\geq0, j\in\mathbb{Z}, l=v_p(n-{p\eK j\over p-1})$, and $n'\equiv p^{-l}(n-{p\eK j\over p-1})\mod p$, we have 
\begin{equation}\label{E:ref-tate-dif}
d^{{p^{l+1}-1\over p-1}-{1\over\eK}}(z^n\sigma^j) =n'\bar{\tilde{\mu}}^{{p^{l}-1\over p-1}} z_0^{p\eK{p^l-1\over p-1}+n-1}\sigma^jdz,
\end{equation}
which is non-zero in  $\tilde{E}_{1,j, {p^{l+1}-1\over p-1}+\frac{n-1}{e_K}}^{{p^{l+1}-1\over p-1}}$. Moreover, the exponents of $z_0$ in the  targets of (\ref{E:ref-tate-dif}) are all different. Consequently, these are all the nontrivial refined algebraic Tate differentials.
\end{prop}
\begin{proof}
Choose $k, m\in\mathbb{N}$ such that 
\[
p^k>j,\quad \min\{m, k\}>l. 
\]
Thus
\[
(p^k-j)\eK{p^{m+1}-p\over p-1}-n\equiv {p\eK j\over p-1}-n\mod p^{l+1}.
\]
It follows that $(p^k-j)\eK{p^{m+1}-p\over p-1}=n+sp^l$ with $s\equiv -n' \mod p$. By Lemma \ref{L:non-zero-stem}, $z^{(p^k-j)\eK{p^{m+1}-p\over p-1}}\sigma^j$ survives to the $\tilde{E}^{{p^{l+1}-1\over p-1}}$-term. Hence
\[
d^{{p^{l+1}-1\over p-1}-{1\over\eK}}(z^{(p^k-j)\eK{p^{m+1}-p\over p-1}}\sigma^j)=d^{{p^{l+1}-1\over p-1}-{1\over\eK}}(z^{n+sp^l}\sigma^j)=0.
\]
By Leibniz rule and Proposition \ref{P:stem0}, we deduce that
\[
z_1^{sp^{l}}d^{{p^{l+1}-1\over p-1}-{1\over\eK}}(z^n\sigma^j)=-z_0^{n}\sigma^jd^{{p^{l+1}-1\over p-1}-{1\over\eK}}(z^{sp^l})= -s\bar{\tilde{\mu}}^{{p^{l}-1\over p-1}} z_0^{p\eK{p^l-1\over p-1}+sp^l-1+n}\sigma^jdz.
\]
Recall that both $\eta_L$ and $\eta_R$ define the refined Nygaard filtrations. It follows that
\[
d^{{p^{l+1}-1\over p-1}-{1\over\eK}}(z^n\sigma^j)= n'\bar{\tilde{\mu}}^{{p^{l}-1\over p-1}}z_0^{p\eK{p^l-1\over p-1}+n-1}\sigma^jdz.
\]

The rest is similar to the proof of Proposition \ref{P:stem0}: put $\tilde{n}=p\eK{p^l-1\over p-1}+n$, then 
\[
l=v_p(\tilde{n}-\frac{pe_K(j-1)}{p-1}).
\]
That is, $n$ is uniquely determined by $\tilde{n}$.
\end{proof}

\begin{rem}\label{r91}
There is a correspondence between refined algebraic Tate differentials and Tate differentials in prior works. More precisely, 
$z$, $z^\eK$,  $\sigma$ and $dz$ correspond to $\pK$, $\tau_K\alpha_K$, $\tau_K^{-1}$ and $\tau_K\pK d\log\pK$ in \cite[Theorem 5.5.1]{HM} respectively;
for $p=2$ and $\eK=1$, $\sigma$, $z^2\sigma$ and $z\sigma^2 dz$ correspond to $t^{-1}$, $te_4$ and $e_3$ in \cite[Theorem 8.14]{Rog} respectively;
for $p$ odd and $\eK=1$, $\sigma$, $z^p\sigma^{p-1}$ and $z^{p-1}\sigma^pdz$ correspond to $t^{-1}$, $tf$ and $e$ in \cite[Theorem 7.4]{Tsa} respectively. 

Thus Proposition \ref{P:non-zero-stem} not only recovers and extends (by including the case $p=2$ and $\eK>1$) the previous results in \cite{HM} \cite {Rog} and \cite{Tsa}, but its proof is purely algebraic and vastly simpler, and that it is this simplification that has made the extension possible.
\end{rem}

%\section{$TC$ spectral sequence mod $p$}
\section{The $E^2$-term of mod $p$ descent spectral sequence I}\label{S:e2tc-tp}
In this section, we determine the $E^2$-terms of the mod $p$ descent spectral sequences for $\tc^{-}(\OK)$ and $\tp(\OK)$.

\begin{prop}\label{P:e20}
For $j\in\mathbb{Z}$, $E^2_{0,2j}(\tp(\OK);\Fp)$ is non-zero if and only if  
$j\geq0$ and $p-1$ divides $\eK j $. If this condition holds,  then $E^2_{0,2j}(\tp(\OK);\Fp)$ is a $1$-dimensional $\k$-vector space generated by a cycle with leading term $z^{p\eK j\over p-1}\sigma^j$. Moreover, %the canonical map induces 
the map
\[
E^2_{0,*}(\tc^-(\OK);\Fp)\rightarrow E^2_{0,*}(\tp(\OK);\Fp).
\]
induced by the canonical map
is an isomorphism.
\end{prop}
\begin{proof}
By Proposition \ref{P:non-zero-stem}, we deduce that $d_{\frac{p^{l+1}-1}{p-1}-\frac{1}{e_K}}(z^n\sigma^j)=0$ is equivalent to 
\[
l<v_p(n-\frac{pe_Kj}{p-1}).
\] 
Thus $z^n\sigma^j$ has non-trivial contribution to $\tilde{E}^{\infty}(\tp(\OK);\Fp)$ if and only if  
\[
n=\frac{pe_Kj}{p-1}.
\] 
This concludes the first two assertions. For the last one, since 
\[
\can: \tc^-_*(\OK/\Sz;\Fp)\to \tp_*(\OK/\Sz;\Fp)
\]
is injective, 
we have
\[
\can: E^2_{0,*}(\tc^-(\OK);\Fp)\to E^2_{0,*}(\tp(\OK);\Fp)
\]
is injective as well. On the other hand, when $E^2_{0,2j}(\tp(\OK);\Fp)$ is non-zero, by Theorem \ref{T:rtp}, we have $\mathrm{can}(z^{p\eK j \over p-1}\sigma^j)=\bar{\mu}^{-j}z^{\eK j \over p-1}u^j\in E^2_{0,2j}(\tc^-(\OK);\Fp)$. Thus 
\[
\can: E^2_{0,*}(\tc^-(\OK);\Fp)\to E^2_{0,*}(\tp(\OK);\Fp)
\]
is also surjective.
\end{proof}

\begin{prop}\label{P:e21-tp}
The $\k$-vector space $E^2_{-1,2j}(\tp(\OK);\Fp)$
has a basis given by a family of cocycles with leading terms
%is freely generated by a set of cycles whose leading terms are
\begin{itemize}
\item
$z_0^{{p\eK(j-1) +bp^l\over p-1}-1}\sigma^jdz$ with $l\geq 1, b\in \mathbb{Z}$ satisfying
%$${\eK (a-1)+bp^{v-1}}>0$$
\[
-{\eK(j-1)\over p^{l-1}}<b<{p\eK }-{\eK j\over p^{l-1}}, \quad p\nmid b, \quad b\equiv -\eK(j-1) \mod p-1, 
\]
and %whenever %the exponents are natural numbers.  
\item
$z_0^{{pe_K(j-1)\over p-1}-1}\sigma^jdz$, 
if $ j>1$ and $p-1\mid \eK(j-1)$.
\end{itemize}
\end{prop}
\begin{proof}
We first treat the case of $e_K>1$. In this case, by Corollary \ref{C:ref-tate-tp0-ref}, we see that $E^2_{-1,2j}(\tp(\OK);\Fp)$ is generated over $\k$ by cycles which are detected by 
$\{z_0^{n-1}\sigma^jdz\}_{n\geq1}$.  
By Proposition \ref{P:non-zero-stem}, $z_0^{n-1}\sigma^jdz$ is hit by 
$z^m\sigma^j$ if and only if 
\begin{equation}\label{E:r}
pe_K\frac{p^{l}-1}{p-1}+m=n
\end{equation}
with $l=v_p(m-{pe_K j\over p-1})<\infty$. In this case, it follows that 
\[
n\equiv m- {pe_K \over p-1}\mod p^{l+1},
\] 
yielding $l=v_p(m-{pe_K j\over p-1})=v_p(n-{pe_K(j-1)\over p-1})$. Hence $m$ is uniquely determined by $n, j$.

Now put $l=v_p(n-{pe_K(j-1)\over p-1})$. If $l=\infty$, then by previous argument $z_0^{n-1}\sigma^jdz$ is not hit by any $z^m\sigma^j$; in this case it follows that $ j>1$, $p-1\mid \eK (j-1)$ and $n=\frac{p\eK (j-1)}{p-1}$. 

If $l<\infty$, then we may write
\[
n={p\eK (j-1)+bp^l\over p-1}
\]
for some $b\in\mathbb{Z}$ satisfying 
\[
p\nmid b,\quad b\equiv -(j-1)\eK \mod p-1, \quad {p(j-1)\eK +bp^l\over p-1}\geq 1;
\] 
the last one is equivalent to  
\begin{equation}\label{E:lower}
bp^l+pe_Kj \geq p-1+pe_K. 
\end{equation} 
On the other hand, by (\ref{E:r}), $z_0^{n-1}\sigma^jdz$ is not hit by any refined algebraic Tate differential if and only if 
\begin{equation}\label{E:upper}
n-{p\eK(p^l-1)\over p-1}<0.
\end{equation}
Note that (\ref{E:upper}) implies that $l\geq1$. Conversely, if $l\geq1$, then (\ref{E:lower}) plus (\ref{E:upper}) is equivalent to
\[
-{\eK(j-1)\over p^{l-1}}<b<{p\eK }-{\eK j\over p^{l-1}}, 
\]
concluding the desired result. Finally, note that all the resulting leading terms $z_0^{n-1}\sigma^j$ satisfy $p|n$. Thus by Corollary \ref{C:ref-tate-tp0-ref}, the above argument applies equally to the case of $e_K=1$.
\end{proof}

\begin{prop}\label{P:e21-tc-}
For $j\geq 1$, the $\k$-vector space
$E^2_{-1,2j}(\tc^-(\OK);\Fp)$
has a basis given by a family of cocycles with leading terms 
%is freely generated over $\k$ by a set of cycles whose leading terms are

\begin{itemize}
\item
$$z_0^{{p\eK (j-1)+bp^l\over p-1}-1}\sigma^jdz$$ with $l\geq0$, $b\in \mathbb{Z}$ satisfying
%$${\eK (a-1)+bp^v}>0$$
\[
-{\eK(j-1)\over p^l}<b<{p\eK }-{\eK j\over p^l},\quad p\nmid b, \quad b\equiv -\eK(j-1) \mod p-1,
\]
and
\item
$z_0^{{p\eK (j-1)\over p-1}-1}\sigma^jdz$
with $ j>1$ and $p-1\mid \eK(j-1)$.
%whenever %the exponents are natural numbers. 
\end{itemize}
\end{prop}
\begin{proof}
Recall that the refined algebraic homotopy fixed points spectral sequence is a truncation of the refined algebraic Tate spectral sequence. More precisely, 
for 
\[
z^n\sigma^j\in \tilde{E}^{\frac{1}{e_K}}(\tp(\OK);\Fp) \quad (\text{resp. $z_0^{n-1}\sigma^jdz\in \tilde{E}^{\frac{1}{e_K}}(\tp(\OK);\Fp)$}), 
\] 
it belongs to $\tilde{E}^{\frac{1}{e_K}}(\tc^{-}(\OK);\Fp)$ is equivalent to $je_K\leq n$ (resp. $(j-1)e_K\leq n-1$).

Therefore, using the argument of Proposition \ref{P:e21-tp}, we deduce that for 
\[
z^{n-1}\sigma^jdz\in \tilde{E}^{\frac{1}{e_K}}(\tc^{-}(\OK);\Fp),
\] 
it is not hit by any refined algebraic homotopy fixed points differential if and only if 
\[
n=\frac{pe_K(j-1)}{ p-1}
\]
or
\begin{equation}\label{E:tc-}
n-p\eK\frac{p^l-1}{p-1}<je_K
\end{equation}
for $l=v_p(n-{p\eK(j-1)\over p-1})$.

In the first case, we have $ j>1$ and $p-1\mid \eK(j-1)$. Conversely, under this condition, it is straightforward to verify that $z^{\frac{pe_Kj}{p-1}}\sigma^jdz$
 belongs to $\tilde{E}^{\frac{1}{e_K}}(\tc^{-}(\OK);\Fp)$. 
 
 In the second case, we may write $n={p\eK (j-1)+bp^l\over p-1}$ with 
 \[
 p\nmid b, \quad b\equiv -\eK(j-1) \mod p-1.
 \] 
 Moreover, the conditions $n\geq1$ plus (\ref{E:tc-}) is equivalent to 
 \[
 -{\eK(j-1)\over p^l}<b<{p\eK }-{\eK j\over p^l}.
 \]
Finally, if $b$ satisfies all these conditions, then it is straightforward to check that $z_0^{{p\eK (j-1)+bp^l\over p-1}-1}\sigma^jdz$ belongs to $ \tilde{E}^{\frac{1}{e_K}}(\tc^{-}(\OK);\Fp)$. 
 \end{proof}
%$$x_1^{{p\eK (a-1)\over p-1}-1}\sigma^adx,\;\;\text{with}\;\; a>1, \;\text{and},\; p-1\mid \eK(a-1)$$
%$$x_1^{{p\eK (a-1)+bp^v\over p-1}-1}\sigma^adx,\;\;\text{with}\;\;{\eK (a-1)+bp^v}\geq p-1$$
%with $b\in \Zp^\times$ and $$b<{p\eK }-{\eK a\over p^v}$$
%whenever %the exponents are natural numbers.

\begin{lem}\label{l99}
For $j\geq1$, the kernel of the canonical map
\[
\can:E^2_{-1,2j}(\tc^-(\OK);\Fp)\rightarrow E^2_{-1,2j}(\tp(\OK);\Fp)
\]
is an $\eK j$-dimensional $\k$-vector space which has a basis given by a family of cycles with leading terms 
\[
z_0^{{p\eK (j-1)+bp^l\over p-1}-1}\sigma^jdz
\]
with $l\geq0, b\in\mathbb{Z}$ satisfying
\begin{equation}\label{E:ker-cond}
 p\nmid b, \quad b\equiv -\eK(j-1) \mod p-1, \quad p^{l-1}\leq {\eK j\over p\eK-b}< p^l. 
\end{equation}
\end{lem}
\begin{proof}
By Propositions \ref{P:e21-tp}, \ref{P:e21-tc-}, we obtain that the kernel of can
is the $k$-vector space with a basis given by a family of cycles which are detected by
$z_0^{{p\eK (j-1)+bp^l\over p-1}-1}\sigma^jdz$
with $l\geq0, b\in\mathbb{Z}$ satisfying
\[
-{\eK(j-1)\over p^l}<b<{p\eK }-{\eK j\over p^l}, \hspace{1mm} p\nmid b, \hspace{1mm} b\equiv -\eK(j-1) \mod p-1, \hspace{1mm} b\notin(-{\eK(j-1)\over p^{l-1}}, {p\eK }-{\eK j\over p^{l-1}}). 
\]
It is straightforward to see that the first condition plus the last conditions is equivalent to 
\begin{equation}\label{E:ker-can}
p\eK-{\eK j\over p^{l-1}}\leq b < p\eK-{\eK j\over p^l},
\end{equation}
which in turn is equivalent to the last condition of (\ref{E:ker-cond}).

It remains to count the number of cycles. To this end, first note that (\ref{E:ker-can}) implies that 
\[
pe_K(1-j)\leq b<pe_K.
\] 
Conversely, for any $e_K(1-j)\leq m<e_K$, there is exactly one  $b\in [pm, pm+p-1]$ satisfying the first two conditions of  (\ref{E:ker-cond}). Moreover, for any $b\in [pe_K(1-j), pe_K)$, there is exactly one $l$ satisfying (\ref{E:ker-can}). We thus conclude that the number of such cycles is $e_K-e_K(1-j)=e_Kj$. 
\end{proof}

\section{The $E^2$-term of mod $p$ descent spectral sequence II} \label{S:e2tc}
In this section, we determine the $E^2$-term of the mod $p$ descent spectral sequence for $\tc(\OK)$. Firstly, we study the action of Frobenius on $E^2(\tc^-(\OK); \Fp)$.
\begin{lem} \label{L:frob-e20-tc}
For $n\geq \eK j$, we have
\[
\varphi(z^n\sigma^j) = \bar{\mu}^{-pj}z^{p(n-\eK j)}\sigma^j.
\]
\end{lem}
\begin{proof}
Using Theorem \ref{T:rtp}, we have
\[
\varphi(z^n\sigma^j) = \varphi(z^{n-e_Kj})\bar{\mu}^{-pj}\varphi(E_K(z) \sigma)^j = \bar{\mu}^{-pj} z^{p(n-e_Kj)}\varphi(u)^j=\bar{\mu}^{-pj} z^{p(n-e_Kj)}\sigma^j.
\]
\end{proof}

\begin{lem}\label{l100} If $\eK>1$, then
\[
\varphi(\sigma_0 (z_0-z_1)) \equiv -z_0^{p-1}\sigma_0 (z_0-z_1) \mod \mathcal{N}^{\geq {p\over \eK}+1 }.
\]
%where $$\gamma_0={min(p,\eK-1)\over \eK}$$
\end{lem}
\begin{proof}
%Recall that by our convention $$\sigma dz=can^{-1}((z_1-z_2)\sigma)$$
In $\tc_2^-(\OK/\Szz)$, we have
\[
\varphi(\sigma_0(z_0-z_1)) = \varphi(\sigma_0 E_K(z_0))\frac{\varphi (z_0-z_1)}{\varphi(E_K(z_0))}= h\varphi(u_0)=h\sigma_0.
\]
Using (\ref{E:l-l+1}) for $l=0$ and the fact that $f^{(1)}\in\mathcal{N}^{\geq p}$, we get
\[
h\equiv \delta(z_0-z_1)/\delta(E_K(z_0)) \mod \mathcal{N}^{\geq p}.
\]
By Lemma \ref{L:xi0}, we have
\[
\delta(z_0-z_1) \equiv -z_0^{p-1}(z_0-z_1) \mod (p,\mathcal{N}^{\geq\frac{p-2}{e_K}+2}).
\]
On the other hand, a short computation shows that
\[
\delta(E_K(z_0))\equiv 1\mod (p, \mathcal{N}^{\frac{p}{e_K}}).
\]
Putting these together, we conclude
\[
h\equiv -z_0^{p-1}(z_0-z_1) \mod (p, \mathcal{N}^{\geq r_0} ),
\]
where 
\[
 r_0 = \mathrm{min}(p,{2p-1\over\eK}+1,{p-2\over\eK}+2)\geq {p\over \eK}+1
 \]
as $e_K>1$. This yields the desired result by modulo $p$.
\end{proof}

\begin{lem}\label{L:frob-nyg}
If $\alpha \in \mathcal{N}^{\geq m} \tc^-_{2j}(\OK/\Szz;\mathbb{F}_p)$, then 
\[
\varphi(\alpha)\in\mathcal{N}^{\geq p(m-j)} \tp_{2j}(\OK/\Szz;\mathbb{F}_p).
\]  
\end{lem}
\begin{proof}
Write $m = m_0 + {m_1\over\eK}$ with $m_0\geq j$, $0\leq m_1<\eK$. Then there exist 
\[
x\in\mathcal{N}^{\geq m_0} \tc^-_{2j}(\OK/\Szz;\mathbb{F}_p), \quad y\in\mathcal{N}^{\geq m_0+1} \tc^-_{2j}(\OK/\Szz;\mathbb{F}_p)
\] 
such that $\alpha = z_0^{m_1}x + y$. By a variant of the proof of Lemma \ref{L:nyg-frob}, we get $\varphi(x)$ divisible by $\varphi(\EK(z_0))^{m_0-j}$, yielding 
\[
\varphi(x)\in\mathcal{N}^{\geq p(m_0-j)}\tp_{2j}(\OK/\Szz;\mathbb{F}_p).
\] 
Similarly,  we get $\varphi(y)\in \mathcal{N}^{\geq p(m_0+1-j)}\tp_{2j}(\OK/\Szz;\mathbb{F}_p)$. It follows that 
\[
\varphi(\alpha) = z_0^{pm_1}\varphi(x) + \varphi(y) \in \mathcal{N}^{\geq p(m-j)}\tp_{2j}(\OK/\Szz;\mathbb{F}_p).
\] 
\end{proof}

\begin{prop} \label{P:frob}  
For $j\geq1$, if $\alpha\in E^2_{-1,2j}(\tc^-(\OK);\mathbb{F}_p)$ is detected by $z_0^{n-1}\sigma^jdz$ in the refined algebraic homotopy fixed points spectral sequence,  then
$\varphi(\alpha)\in E^2_{-1,2j}(\tp(\OK);\mathbb{F}_p)$ %lies in refined Nygaard filtration (at least) $$p({l\over \eK}-a +1) + 1 -{1\over\eK}$$
% ${pl-1\over \eK } -pa+p+1$, 
is detected by
\[
-\bar{\mu}^{-p(j-1)} z_0^{p(n-\eK(j-1))-1}\sigma^jdz
\]
in the refined algebraic Tate spectral sequence.
\end{prop}
\noindent  Before proving Proposition \ref{P:frob}, note that the map
\[
z_0^{n-1}\sigma^jdz \mapsto z_0^{p(n-\eK(j-1))-1}\sigma^jdz
\]
gives rise to a bijection between leading terms of the cycles given in Propositions \ref{P:e21-tp} and Proposition \ref{P:e21-tc-} respectively. Therefore, granting Proposition \ref{P:frob}, we obtain the following results.

\begin{cor}\label{C:frob-bij}
For $j\geq1$, $\varphi:  E^2_{-1,2j}(\tc^-(\OK);\mathbb{F}_p)\to  E^2_{-1,2j}(\tp(\OK);\mathbb{F}_p)$ is an isomorphism.
\end{cor}

\begin{cor}\label{C:frob-fil}
Suppose $\alpha\in E^{2}_{-1,2j}(\tc^{-}(\OK);\Fp)$ has refined Nygaard filtration $m$.
\begin{enumerate}
\item[(1)]
For $j\geq1$, the filtration of $\varphi(\alpha)$ is $>m$ (resp. $<m$, $=m$)
%higher than (resp. lower than, equal to) the filtration of $\alpha$ 
if and only if 
\[
m>{pj-1\over p-1}-\frac{1}{e_K} \quad \text{(resp. $m<{pj-1\over p-1}-\frac{1}{e_K}$, $m={pj-1\over p-1}-\frac{1}{e_K}$).}
\]
\item[(2)]
For $j\leq0$, the filtration of $\varphi(\alpha)$ is $>m$. % higher than that of $\alpha$. 
\end{enumerate}
\end{cor}
\begin{proof}
For (1), by Proposition \ref{P:frob}, $\varphi(\alpha)$ has filtration 
\[
m'={p(e_K(m-1)+1-\eK(a-1))-1\over\eK}+1= p(m-j) + {p-1\over \eK}+1.
\]
A short computation shows the desired result. For $(2)$, since $dz$ has filtration 1, we may assume $m\geq1$. Then we may write $\alpha = \beta v^{-j}$ with $\beta \in  \mathcal{N}^{\geq m}E^{2}_{-1,0}(\tc^{-}(\OK);\Fp)$. It follows that $\varphi(\alpha)$ is divisible by $\varphi(\beta)$, which belongs to $\mathcal{N}^{\geq pm}E^{2}_{-1,0}(\tp(\OK);\Fp)$. Now the desired result follows as $pm>m$. 
\end{proof}

Now we prove Proposition \ref{P:frob}. 
\begin{proof}
Regard $z_0^{n-1}\sigma_0^jdz$ as an element of the cobar complex of $\tc^-_{2j}(\OK/\Sz;\Fp)$. Note that
$d(z_0), d(\sigma_0)\in\mathcal{N}^{\geq1}\tc^-_{2j}(\OK/\mathbb{S}_{W(\k)}[z_0,z_1, z_2];\Fp)$. By the Leibniz rule, we deduce that 
\[
d(z_0^{n-1}\sigma_0^j dz)\in\mathcal{N}^{\geq\frac{n-2}{e_K}+2}\tc^-_{2j}(\OK/\mathbb{S}_{W(\k)}[z_0,z_1, z_2];\Fp). 
\]
Using Lemma \ref{L:ref-tate-tp0}, we deduce that there exists $\beta\in \mathcal{N}^{\geq\frac{n-2}{e_K}+2}\tc^-_{2j}(\OK/\mathbb{S}_{W(\k)}[z_0,z_1];\Fp)$ such that 
$d(\beta)=d(z_0^{n-2}\sigma_0^j dz)$; hence $d(z_0^{n-1}\sigma_0^j dz-\beta)=0$. Therefore, by induction on $n$, we are reduced to treat the case 
\[
\alpha \equiv z_0^{n-1}\sigma_0^j (z_0-z_1) \mod \mathcal{N}^{\geq {n-2\over\eK}+2}.
\]
By Lemma \ref{L:frob-nyg},  we have
\[
\varphi(\alpha) \equiv \varphi(z_0^{n-1}\sigma_0^{j-1})\varphi(\sigma_0 (z_0-z_1)) \mod 
\mathcal{N}^{\geq p({n-2\over\eK}+2-j)}.
\]
By Lemma \ref{l100} and Lemma \ref{L:frob-e20-tc},  we have
\[
\varphi(z_0^{n-1}\sigma_0^{j-1})\varphi(\sigma_0(z_0-z_1)) \equiv -\bar{\mu}^{-p(j-1)}z_0^{p(n-\eK(j-1))-1}\sigma^j_0(z_0-z_1) \mod \mathcal{N}^{\geq {p(n-\eK(j-1))\over\eK}+1}.
\]
Note that if $\eK>3$, then
 \[
 p({n-2\over\eK}+2-j)\geq {p(n-\eK(j-1))\over\eK}+1,
 \]
yielding the desired result for $\eK>3$.
 
 %By a variant of the argument of Lemma \ref{l10}, $\varphi(\alpha)$ is divisible by $\varphi(\EK(z_1))^{[m]-a}$, which lies in $(p,\mathcal{N}^{\geq p([m]-a)})$. So we conclude that the Nygaard filtration of $\varphi(\alpha)$ tend to infinity when the Nygaard filtration of $\alpha$ tend to infinity. Then by an inverse induction, it suffices to show the weaker version, that for certain representatives $\alpha$ of $z^{l-1}\sigma^adz$, $\varphi(\alpha)$ is detected as claimed.

%The case for $\eK>1$ follows from Lemma \ref{l100} and Lemma \ref{l75}.

For the case of $\eK\leq3$, let $m, K', T_m$ and $\tilde{E'}(\tp(\mathcal{O}_{K'});\Fp)$ be as in the proof of Proposition \ref{P:stem0-all}. Let $\tilde{E'}(\tc^-(\mathcal{O}_{K'});\Fp)$ be the "less refined"  algebraic homotopy fixed point spectral sequence, which is a truncation of $\tilde{E'}(\tp(\mathcal{O}_{K'});\Fp)$. First note that Remark \ref{R:ref-less-ref} \footnote{Using Proposition \ref{P:non-zero-stem}, the argument of Lemma \ref{L:L-in-P} (hence Remark \ref{R:ref-less-ref}) adapts to $\tilde{E'}_{1,j,*}(\tp(\mathcal{O}_{K'});\Fp)$ for all $j\in\mathbb{Z}$.} implies that 
\[
T_m: \tilde{E}^\infty_{1,*,*}(\tp(\mathcal{O}_{K});\Fp)\to \tilde{E}^{'\infty}_{1,*,*}(\tp(\mathcal{O}_{K'});\Fp)
\] 
is injective. Thus it restricts to an injective map $\tilde{E}^{\infty}(\tc^-(\OK);\Fp)\to \tilde{E}^{'\infty}(\tc^-(\mathcal{O}_{K'});\Fp)$. Since $z_0^{n-1}dz$ is non-zero in $\tilde{E}^{\infty}(\tc^-(\mathcal{O}_{K});\Fp)$,  using Remark \ref{R:ref-less-ref}, we may deduce that $T_m(\alpha)$ is detected by $mz_0^{mn-1}\sigma^jdz$.  Then by the case of $e_K>3$, we get that  
\[
-m\bar{\mu}^{-p(j-1)}z_0^{p(mn-m\eK(j-1))-1}\sigma^j_0dz
\]
detects $T_m(\varphi(\alpha))=\varphi(T_m(\alpha))$ in $\tilde{E}^\infty(\tp(\mathcal{O}_{K'});\Fp)$. Now suppose $\varphi(\alpha)$ is detected by $\lambda z^{l-1}\sigma^jdz$. Then $T_m(\varphi(\alpha))$ is detected by $T_m(\lambda z^{l-1}\sigma^jdz)=\lambda mz^{mj-1}\sigma^jdz$. Comparing the two expressions, we get 
\[
l=p(n-m\eK(j-1)), \quad \lambda=-\bar{\mu}^{-p(j-1)}
\] 
by Remark \ref{R:ref-less-ref} again. This completes the proof.
\end{proof}

 \begin{lem}\label{L:can-fil}
The canonical map induces 
\begin{itemize}
\item
for  $j>0$, a surjection
\[
E^{2}_{-1,2j}(\tc^{-}(\OK);\Fp)\rightarrow \mathcal{N}^{\geq j}E^{2}_{-1,2j}(\tp(\OK);\Fp) ;
\]
\item  for $j\leq0$, an isomorphism
\[
E^{2}_{-1,2j}(\tc^{-}(\OK);\Fp)\rightarrow E^{2}_{-1,2j}(\tp(\OK);\Fp);
\]
%\end{itemize}
%\end{lem}

%\begin{lem} 
 %\begin{itemize}
\item for $m\geq j$,  a surjection
\[
\mathcal{N}^{\geq m}E^{2}_{-1,2j}(\tc^{-}(\OK);\Fp)\rightarrow \mathcal{N}^{\geq m}E^{2}_{-1,2j}(\tp(\OK);\Fp).
\]
%is surjective.
\end{itemize}
\end{lem}
\begin{proof}
These statements follow from the corresponding result on 
\[
\can: \tc^-_{2j}(\OK/\Sz^{\otimes\oldbullet})\to \tp_{2j}(\OK/\Sz^{\otimes\oldbullet}).
\]
\end{proof}

Combining Corollary \ref{C:frob-fil} and Lemma \ref{L:can-fil}, we deduce the following results immediately. 

\begin{cor}\label{c111}
For $j\geq1$ and $m\geq{pj-1\over p-1}$,
the map
\[
\can-\varphi:\mathcal{N}^{\geq m}E^{2}_{-1,2j}(\tc^{-}(\OK);\Fp)\rightarrow \mathcal{N}^{\geq m}E^{2}_{-1,2j}(\tp(\OK);\Fp)
\]
is surjective.
\end{cor}

\begin{cor} \label{c109}
For $j\leq0$,
the map
\[
\can-\varphi:E^{2}_{-1,2j}(\tc^{-}(\OK);\Fp)\rightarrow E^{2}_{-1,2j}(\tp(\OK);\Fp)
\]
is an isomorphism.
\end{cor}

Now we are ready to determine $E^{2}(\tc(\OK);\mathbb{F}_p)$. Let $d$ be the minimal number such that
\[
p-1\mid \eK d,\quad \mathrm{N}_{\k/\Fp}(\bar{\mu})^d =1, 
 \]
where $\mathrm{N}_{\k/\Fp} : \k\rightarrow \mathbb{F}_p$ is the norm map.

The following lemma is a reformulation of Hilbert 90 for $\k/\Fp$.
 \begin{lem}\label{L:h90}
For $b\in \k^\times$, the map
\[
b\varphi-\mathrm{id}:\k\xrightarrow{%\alpha\mapsto b\varphi(\alpha)-\alpha
}\k
\]
is bijective if $\mathrm{N}_{k/\Fp}(b)\neq1$, otherwise both the kernel and cokernel are isomorphic to $\Fp$. 
\end{lem}

Using Lemma \ref{L:h90}, we may choose a $(p-1)$-th root $\bar{\mu}^{pd\over p-1}$ of $\bar{\mu}^{pd}$ in $\k$. Denote by $\beta$ the element in $\tilde{E}^2_{0,0,2d}(\tc(\OK);\Fp)\subseteq {E}^2_{0,2d}(\tc^-(\OK);\Fp)$ detected by $\bar{\mu}^{pd\over p-1}z^{p\eK d\over p-1}\sigma^{d}$.
\begin{prop}\label{P:e200-tc}
We have
\[
\tilde{E}^2_{0,0,*}(\tc(\OK);\Fp)= \Fp[\beta].
\]
\end{prop}
\begin{proof}
By Proposition \ref{P:e20}, we first have
\[
E^{2}_{0,*}(\tc^{-}(\OK);\Fp)= \k[z^{p\eK j\over p-1}\sigma^{j}], 
\]
where $j$ is the smallest positive integer such that 
$p-1 \mid \eK j$. 

On the other hand, by Lemma \ref{L:frob-e20-tc},
\[
\varphi(z^{p\eK j\over p-1}\sigma^{j}) = \bar{\mu}^{-pj}z^{p\eK j\over p-1}\sigma^{j}.
\]
Thus $\lambda z^{p\eK j\over p-1}\sigma^{j}\in \tilde{E}^2_{0,0,*}(\tc(\OK);\Fp)$ if and only if $\bar{\mu}^{pj}=\lambda^{-1}\varphi(\lambda)=\lambda^{p-1}$ for some $\lambda\in \k$. In this case,  it follows that 
$\mathrm{N}_{\k/\Fp}(\bar{\mu})^j=1$. Hence $d\mid j$.  Conversely, if $d\mid j$, then such $\lambda$ is of the form $\lambda'\bar{\mu}^{pd\over p-1}$ with $\lambda'\in \Fp$. Now the proposition follows.
\end{proof}

It turns out that $\tilde{E}^2_{i,j,*}(\tc(\OK);\Fp)$ is a free $\Fp[\beta]$-module of finite rank for all $i,j$. In the following, we will find out their generators over $\Fp[\beta]$. Firstly,  combing the proof of Proposition \ref{P:e200-tc}, Lemma \ref{L:frob-e20-tc} and Lemma \ref{L:h90}, we obtain the following result. 

\begin{prop}\label{P:e201-tc}
 The $\Fp[\beta]$-module $\tilde{E}^2_{0,1,*}(\tc(\OK);\Fp)$ is free of rank $1$. 
 \end{prop}

\begin{lem}\label{L:gamma}
There exists $\gamma\in \ker(\can-\varphi)$ detected by $\bar{\mu}^{pd\over p-1}z^{{p\eK d\over p-1}-1}\sigma^{d+1}dz$. 
\end{lem}
\begin{proof}
Let $\gamma_0\in {E}^2_{-1,2(d+1)}(\tc^-(\OK);\Fp)$ be detected by $z^{{p\eK d\over p-1}-1}\sigma^{d+1}dz$.
By Proposition \ref{P:frob}, $\varphi(\gamma_0)$ is detected by $\bar{\mu}^{-pd} z^{{p\eK d\over p-1}-1}\sigma^{d+1}dz$.
It follows that 
\[
(\can-\varphi)(\bar{\mu}^{pd\over p-1}\gamma_0)\in \mathcal{N}^{\geq {{p d\over p-1}}+1}E^2_{-1,2(d+1)}(\tp(\OK);\mathbb{F}_p).
\]
By Corollary \ref{c111}, $\can-\varphi$ is surjective on $\mathcal{N}^{\geq {{p d\over p-1}}+1}E^2_{-1,2(d+1)}(\tc^-(\OK);\mathbb{F}_p)$. Hence we may modify $\gamma_0$ with higher terms to construct the desired element.
\end{proof}

In the following, let $\gamma$ be as in Lemma \ref{L:gamma}. 
\begin{prop}\label{P:e211-tc}
The $\Fp[\beta]$-module $\tilde{E}^2_{-1, 1,*}(\tc(\OK);\Fp)$ is free of rank 1 generated by
$\can(\gamma)\in {E}^2_{-1,2(d+1)}(\tp(\OK);\Fp)$.
\end{prop}
\begin{proof}
Let $\alpha \in E^2_{-1,2j}(\tc(\OK);\mathbb{F}_p)$ represents a non-trivial class in the cokernel of $\can-\varphi$ such that it has the highest leading term in that class. 
By Corollary \ref{c111} and Corollary \ref{c109}, we see that $j\geq1$ and the leading degree of $\alpha$ lies in $[1, {pj\over p-1}-{1\over p-1}-{1\over\eK}]$.

On the other hand, if the leading degree of $\alpha$ is less than ${pj\over p-1}-{1\over p-1}-{1\over\eK}$, by Corollary \ref{C:frob-bij} and Corollary \ref{C:frob-fil}, then we may find some $\alpha'$ with higher leading degree such that $\alpha = \varphi(\alpha')$. Note that $\can(\alpha')$ represents the same class as $\alpha$, yielding a contradiction.

Therefore $\alpha$ must have leading degree ${pj\over p-1}-{1\over p-1}-{1\over\eK}$. That is, $\alpha$ is detected by some $\lambda z_1^{{p\eK (j-1)\over p-1}-1}\sigma^jdz$. Using Lemma \ref{L:h90} and Lemma \ref{L:gamma}, we conclude that $d\mid j-1$ and $\alpha\in \Fp\beta^{\frac{j-1}{d}-1}\can(\gamma)$.

\end{proof}

\begin{prop}\label{P:e210-tc}
As an $\Fp[\beta]$-module, $\tilde{E}^2_{-1,0,*}(\tc(\OK);\Fp)$ is free with a basis given by $\gamma$
and a family of cycles detected respectively by
\[
c z^{{p\eK (j-1)+bp^l\over p-1}-1}\sigma^jdz\in {E}^2_{-1,2j}(\tc^-(\OK);\Fp)
\]
with $l\geq0$ and
\begin{equation}\label{E:e210-tc}
0<b<p\eK, \quad p\nmid b, \quad b\equiv -\eK(j-1) \mod p-1, \quad p^{l-1}\leq {\eK j\over p\eK-b}< p^l, \quad 1\leq j \leq d,
\end{equation}
and $c$ runs over a basis of $\k$ over $\Fp$.
\end{prop}
\begin{proof}
By Corollary \ref{c109}, $\ker(\can-\varphi)$ is trivial for $j\leq0$.  Now suppose $j\geq 1$, and let $0\neq \alpha\in \tilde{E}^2_{-1,0,*}(\tc(\OK);\Fp)$.
By Corollary \ref{C:frob-fil}, $\varphi$ lowers the filtration if the filtration is less than  ${pj\over p-1}-{1\over p-1}-{1\over\eK}$. Thus the leading degree of $\alpha$ is at least ${pj\over p-1}-{1\over p-1}-{1\over\eK}$.

If the leading degree of $\alpha$ is ${pj\over p-1}-{1\over p-1}-{1\over\eK}$, the by Lemma \ref{L:h90} and the argument of Proposition \ref{P:e211-tc}, there exists some $\beta'\in\Fp[\beta]\gamma$ such that $\alpha-\beta'$ has leading degree higher than ${pj\over p-1}-{1\over p-1}-{1\over\eK}$.

Now suppose $\alpha$ has leading degree higher than ${pj\over p-1}-{1\over p-1}-\frac{1}{e_K}$. First note that for a cycle given in Propositions \ref{P:e21-tp}, \ref{P:e21-tc-}, it lies in $\mathcal{N}^{>{pj\over p-1}-{1\over p-1}-\frac{1}{e_K}}$ if and only if $b>0$. Then it is straightforward to see that 
\[
\can: \mathcal{N}^{>{pj\over p-1}-{1\over p-1}-\frac{1}{e_K}}{E}^2_{-1,2j}(\tc^-(\OK);\Fp)\to \mathcal{N}^{>{pj\over p-1}-{1\over p-1}-\frac{1}{e_K}}{E}^2_{-1,2j}(\tp(\OK);\Fp).
\] 
 is surjective, and the cocyles given in the statement of the proposition form an $\Fp$-basis of $\ker(\can)$. Let $S$ be the $\k$-vector space generated by the remaining cycles in $\mathcal{N}^{>{pj\over p-1}-{1\over p-1}-\frac{1}{e_K}}{E}^2_{-1,2j}(\tc^-(\OK);\Fp)$. It follows that $\can$ induces a filtration preserving isomorphism between $S$ and $\mathcal{N}^{>{pj\over p-1}-{1\over p-1}-\frac{1}{e_K}}{E}^2_{-1,2j}(\tp(\OK);\Fp)$.  
 
Now we may write $\alpha=\alpha_1+\alpha_2$ with $\alpha_1\in\ker(\can), \alpha_2\in S$.  It follows that 
\[
(\can-\varphi)(\alpha_2)= \varphi(\alpha_1).
\]
Since $\varphi$ raises the filtration, it follows that 
\[
\alpha_2=(1-\can^{-1}\varphi)^{-1}(\can^{-1}\varphi(\alpha_1))=\sum_{i\geq1} (\can^{-1}\varphi)^i(\alpha_1).
\]
Hence $\alpha_2$ is uniquely determined by $\alpha_1$ and has higher filtration than $\alpha_1$. Thus the map $\alpha\mapsto \alpha_1$ induces an isomorphism between $\ker(\can)$ and $\ker(\can-\varphi)$ preserving the leading term. This completes the proof. 
\end{proof}

\begin{rem}
The above argument can be summarized by the following picture. Put $a=m-j$. The cycles of $E^2_{-1,2j}(\tc^-(\OK);\mathbb{F}_p)$ with leading degree $m$ is represented by the point $(a=m-j, j)$.  Then we may divide the area of cocyles into three regions, bounded by the lines $j+a=0$, $a=0$ and ${j\over p-1}-a-{1\over p-1}-{1\over\eK}=0$. The blue line is the ``critical line" for the Frobenius action. In region I, the canonical map is an isomorphism, and the Frobenius raises filtration; thus $\can-\varphi$ is an isomorphism (Corollary \ref{c109}). In region II, the Frobenius raises the filtration. One may produce an isomorphism between $\ker(\can)$ and $\ker(\can-\varphi)$ preserving the leading term. In region III, the Frobenius lowers the filtration; thus $\ker(\can-\varphi)=0$. Along the critical line, the Frobenius differs from the canonical map by a certain power of $\bar{\mu}$.
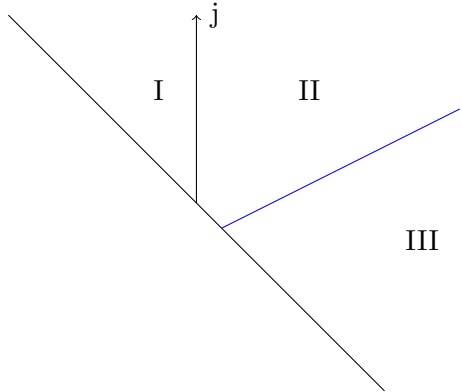
\begin{center} 
\begin{tikzpicture}[scale  = 0.5]
%\draw[help lines] (-2,0) grid (6,2);
\draw (-1,3) node {I};
\draw (3,3) node {II};
\draw (6,-1) node {III};
\draw (0.5,5) node{j};
\draw [->] (0,0) -- (0,5);
\draw  (5,-5) -- (-5,5);
\draw [color = blue] (0.67, -0.67) -- (7, 2.5);
\end{tikzpicture}
\end{center}
\end{rem}
~\\
Note that for $1\leq i\leq e_K$ and $1\leq j\leq d$, there is exactly one 
\[
b\in[(p-1)i+1, pi],
\] 
and hence one pair $(b, l)$, satisfying (\ref{E:e210-tc}). Denote by $\alpha_i^{(j)}$ the cycle detected by 
\[
z^{{p\eK (j-1)+bp^l\over p-1}-1}\sigma^jdz\in E^2_{-1,2j}(\tc^-(\OK);\Fp).
\] 
given in Proposition \ref{P:e210-tc}. Using Proposition \ref{P:e201-tc}, let $\lambda$  be an $\Fp$-basis of $\tilde{E}^2_{0,1,0}(\tc(\OK);\Fp)$. 
By Remark \ref{R:multiplicative}, $E^2(\tc(\OK);\Fp)$ is multiplicative. Combining  Propositions \ref{P:e200-tc}, \ref{P:e201-tc}, \ref{P:e210-tc}, \ref{P:e211-tc}, and Corollary \ref{C:e2-tp-tc-}, we conclude

\begin{thm}\label{T:e2tc}
As $\Fp[\beta]$-modules, we have
\[
E^2_{0,*}(\tc(\OK);\Fp)= \Fp[\beta],
\]
\[
E^2_{-1,*}(\tc(\OK);\Fp)= \Fp[\beta]\{\lambda, \gamma\} \oplus \mathbb{F}_p[\beta]\{\alpha^{(j)}_{i,l}|1\leq i\leq \eK, 1\leq j\leq d, 1\leq l\leq f_K\},
\]
and
\[
E^2_{-2,*}(\tc(\OK);\Fp)= \Fp[\beta]\{\lambda\gamma\},
\]
with $|\lambda| = (-1,0)$, $|\gamma| = (-1,2(d+1))$, $|\alpha^{(j)}_{i,l}| = (-1,2j)$. Moreover, for $i\neq0,-1,-2$, 
\[
E^2_{i,*}(\tc(\OK);\Fp)=0.
\]

\end{thm}

By Theorem \ref{T:e2tc},  both $E^2_{0,*}(\tc(\OK);\Fp)$ and $E^2_{-2,*}(\tc(\OK);\Fp)$ are concentrated in even degrees. This implies the following result. 
\begin{cor}\label{C:e2-tc}
The descent spectral sequence converging to $\tc_*(\OK;\mathbb{F}_p)$ collapses at the $E^2$-term.
\end{cor}

By Remark \ref{R:multiplicative}, $\tc_*(\OK;\Fp)$ is multiplicative for odd $p$. This implies that in this case the collapsing $E^2(\tc_*(\OK);\mathbb{F}_p)$ has no hidden extensions.
\begin{thm}\label{T:oddp}
For $p$ odd, as $\Fp[\beta]$-modules, 
\[
\tc_*(\OK;\Fp) \cong \Fp[\beta]\{1,\lambda,\gamma,\lambda\gamma\} \oplus \mathbb{F}_p[\beta]\{\alpha^{(j)}_{i,l}|1\leq i\leq e_K, 1\leq j\leq d, 1\leq l\leq f_K\}
\]
with $|\beta|=2d$, $|\lambda| = -1$, $|\gamma| = 2d+1$, $|\alpha^{(j)}_{i,l}| = 2j-1$. In particular, $\tc_*(\OK;\mathbb{F}_p)$ is a free $\Fp[\beta]$-module.
\end{thm}

The case $p=2$ is more subtle as $\tc(\OK;\mathbb{F}_2)$ is no longer multiplicative. First note that there is a Bott class in $\tc_2(\OK;\mathbb{F}_2)$ lifting $\beta$; by abuse of notation, we also denote it by $\beta$. Using the $v_1^4$ self map on $\mathbb{S}/2$, we obtain a $\mathbb{Z}_2[\beta^4]$-module structure on $\tc_*(\OK;\mathbb{F}_2)$.

\begin{thm}\label{T:evenp}
As a $\mathbb{Z}_2[\beta^4]$-module,  $\tc_*(\OK;\mathbb{F}_2)$ is isomorphic to 
\begin{equation*}
=\left\{\begin{array}{lll}
     \Fp[\beta^2]\{1\}\oplus \mathbb{Z}/4[\beta^2]\{\beta\} \oplus \Fp[\beta]\{\lambda,\gamma\}\\
   \oplus\Fp[\beta^2]\{\beta\lambda\gamma\} \oplus \mathbb{F}_2[\beta]\{\alpha_{i,l}|1\leq i\leq e_K, 1\leq l\leq f_K\} \vspace{-8mm}\\
   \hspace{90mm}\text{if $[K:\mathbb{Q}_2]$ is odd,}\\
   ~\\
   ~\\
         \Fp[\beta]\{1,\lambda,\gamma,\lambda\gamma\} \oplus \mathbb{F}_2[\beta]\{\alpha_{i,l}|1\leq i\leq e_K,1\leq l\leq f_K\} ~~~\text{if $[K:\mathbb{Q}_2]$ is even,}
      \end{array}
       \right.
  \end{equation*}
with $|\beta|=2$, $|\lambda| = -1$, $|\gamma| = 3$, $|\alpha_i| = 2i-1$. 
\end{thm}
\begin{proof}
It suffices to determine the $2$-extensions of $\beta^k$. First recall that for the mod $2$ reduction of an $E_\infty$-algebra in spectra, the $2$-extension of a class $x$ is equal to the mod 2 reduction of ${\partial(x)\over 2}\eta$, where $\partial(-)\over 2$ is the Bockstein homomorphism and $\eta$ is the Hopf invariant one class in $\pi_1(\mathbb{S})$. 

Back to our situation, by comparing with the algebraic $K$-theory of real numbers, we first deduce that the Hurewicz image of $\eta$ in $\mathrm{K}_1(K)\cong K^{\times}$ is $-1$, which is the unique order $2$ element in $\mathrm{K}_1(K)$. It follows that under the cyclotomic trace map, the $\eta$ corresponds to $\partial(\beta)\over 2$, as the latter has order 2 as well.

Now suppose $\tilde{\beta}\in E^1_{0,2}(\tc^-(\OK);\mathbb{Z}_2)$ lifts $\beta\in E^2_{0,2}(\tc(\OK);\mathbb{F}_2)$. Then the Bockstein image $\partial(\beta)\over 2$ is detected by the class $c = (a,b)\in E^2_{-1,2}(\tc(\OK);\mathbb{Z}_2)$, where
\[
a = {\eta_R(\tilde{\beta})-\eta_L(\tilde{\beta})\over 2},\quad b = {\can(\tilde{\beta})-\varphi(\tilde{\beta})\over 2}.
\] 
By Leibniz rule, the 2-extension of $\beta^k$ is therefore equal to
\[
 k\bar{c}^2\beta^{k-1}\in E^2_{-2,2k+2}(\tc(\OK);\mathbb{F}_2).
 \] 
It follows that if $k$ is even, then the 2-extension of $\beta^k$ is trivial. Otherwise the 2-extension of $\beta^k$ is trvial if and only if $\bar{c}^2=0$.   A short computation shows that $c^2$ is represented by the cycle 
\[
(\can(a)\eta_R(b)+\eta_L(b)\varphi(a), [a|a]).
\]
By Proposition \ref{P:e2tc-tp}, we may choose $a'\in \tc^-_4(\OK/\Szz;\mathbb{Z}_2)$ such that $d(a')=[a|a]$. Thus $c^2$ is homologous to $(b',0)$, where
\[
b' = (\can(a)-\varphi(a))b + (\can-\varphi)(a') =(\eta_R - \eta_L)(b)b + (\can-\varphi)(a').
\]

Note that we may take $\tilde{\beta} = \EK(z)^2 \sigma=E_K(z)u$. It follows that $a$ has Nygaard filtration $2$ and 
\[
b =\frac{\EK(z)^2 \sigma-\varphi(E_K(z)u)}{2}=\delta(\EK(z))\sigma.
\] 
We may choose $a'$ with Nygaard filtration  $4$ (by the proof of Proposition \ref{P:e2tc-tp}). Using the argument of the proof of Lemma \ref{L:frob-e20-tc}, we deduce that the mod 2 reduction $\overline{\varphi(a')}$ has Nygaard filtration at least 4.  Thus $\overline{(\can-\varphi)(a')}$ has Nygaard filtration at least $4$. Note that $\overline{\delta(\EK(z))}$ is a polynomial of $z^2$. By Proposition \ref{P:non-zero-stem}, we deduce that if $\eK$ is odd, then $\overline{(\eta_R - \eta_L)(b)}$ has leading term $\overline{\mu}^2z_1^{2\eK-1}\sigma dz$; otherwise the leading term of $\overline{(\eta_R - \eta_L)(b)}$ is of higher degree.

Write the image of $\bar{b'}$ in $E^2_{-1, 4}(\tp(\OK); \mathbb{F}_2)$ as a $k$-linear combination of cycles of distinct leading degrees given in Proposition \ref{P:e21-tp}. Moreover, by the proof thereof, we deduce that if $\eK$ is even, then the leading degrees of these cycles are all higher than $3\eK-1$; if $\eK$ is odd, then the lowest leading degree is $3\eK-1$ (contributed by $\overline{(\eta_R - \eta_L)(b)b}$). Therefore, by the proof of Proposition \ref{P:e211-tc}, we conclude that if $\eK$ is even, then $\bar{b'}$ is homologous to $0$ in $E^2_{-2, 4}(\tc(\OK); \mathbb{F}_2)$; if $\eK$ is odd, 
note that the leading term of the lowest degree is $\bar{\mu}^2z_1^{2\eK-1}\sigma^2dz$, combining with Proposition \ref{P:e201-tc}, we deduce that $\bar{b'}$ is non-trivial in $E^2_{-2, 4}(\tc(\OK); \mathbb{F}_2)$ if and only if $1\in k$ is not in the image of $\mathrm{id}-\varphi$. That is,  if and only if $x^2-x=1$ does not split over $k$, which is equivalent to $[k:\Fp]$ being odd. Moreover, in this case, $\bar{b'}$ is homologous to $\lambda\gamma$ in $E^2_{-2, 4}(\tc(\OK); \mathbb{F}_2)$.
\end{proof} 

\begin{rem}
The problem of 2-extensions may also be treated using the norm residue isomorphism. Recall that the Hurewicz image of $\eta$ in $\mathrm{K}_1(K)\cong K^{\times}$ is $-1$. It follows that the mod 2 reduction of  $\eta^2$ corresponds to the Hilbert symbol $\{-1,-1\}_K$, which in turn corresponds to the algebra of Hamilton's quaternions, under the norm residue isomorphism. 
Thus the 2-extension of $\beta$, which is equal to the mod 2 reduction of $\eta^2$ as noted in the proof of Theorem \ref{T:evenp}, is trivial if and only if the algebra of Hamilton's quaternions  splits over $K$; the latter is equivalent to $[K:\mathbb{Q}_2]$ being even.
\end{rem}

To complete the proof of Theorem \ref{T:main}, it remains to show $d=[K(\zeta_p):K]$. 
 \begin{prop} \label{P:d}
The integer $d$ is equal to $[K(\zeta_p):K]$.
\end{prop}
\begin{proof}
Put $d'=[K(\zeta_p):K]$. Note that $\mu=\frac{p}{\mathrm{N}_{K/K_0}(-\varpi_K)}$. We first have 
\[
p-1=[K_0(\zeta_p):K_0]\mid [K(\zeta_p):K_0]=d'e_K.
\]
Secondly,  we have
\begin{equation*}
 \begin{split}
 \mathrm{N}_{K_0/\mathbb{Q}_p}(\mu)^{d'}&=\mathrm{N}_{K_0/\mathbb{Q}_p}(\frac{p}{\mathrm{N}_{K/K_0}(-\varpi_K))})^{d'}=\frac{p^{f_Kd'}}{\mathrm{N}_{K/\mathbb{Q}_p}(-\varpi_K)^{d'}}\\
 &=\frac{\mathrm{N}_{K_0(\zeta_p)/\mathbb{Q}_p}((1-\zeta_p)^{d'})}{\mathrm{N}_{K(\zeta_p)/\mathbb{Q}_p}(-\varpi_K)}=\mathrm{N}_{K_0(\zeta_p)/\mathbb{Q}_p}(\frac{(1-\zeta_p)^{d'}}{\mathrm{N}_{K(\zeta_p)/K_0(\zeta_p)}(-\varpi_K)}).
\end{split}
\end{equation*}
This yields 
\begin{equation*}
 \begin{split}
\mathrm{N}_{\k/\Fp}(\bar{\mu})^{d'}&=\overline{ \mathrm{N}_{K_0/\mathbb{Q}_p}(\mu)^{d'}}=\overline{\mathrm{N}_{K_0(\zeta_p)/\mathbb{Q}_p}(\frac{(1-\zeta_p)^{d'}}{\mathrm{N}_{K(\zeta_p)/K_0(\zeta_p)}(-\varpi_K)}})\\
&=\mathrm{N}_{k/\Fp}(\overline{\frac{(1-\zeta_p)^{d'}}{\mathrm{N}_{K(\zeta_p)/K_0(\zeta_p)}(-\varpi_K)}})^{p-1}=1.
\end{split}
\end{equation*}
Hence $d|d'$. 

It remains to show $d'|d$. The strategy is to construct a suitable degree $d$ extension of $K$ containing $K(\zeta_p)$ as a subfield. First note that if we replace $K$ by an intermediate extension $K_0\subset K'\subset K$ and $\varpi_K$ by $-\mathrm{N}_{K/K'}(-\varpi_K)$, then $\mu$ is unchanged. Thus we may replace $K$ with its tamely ramified subextension over $K_0$. Now $K$ is of the form $K_0(\pi^{1/e_K})$ for some uniformizer $\pi$ of $K_0$. Let $d_1$ be the minimal positive integer such that $p-1\mid \eK d_1$, and write $d=d_1d_2$. We may further reduce to the case $p-1=e_Kd_1$ by replacing $K$ with $K_0(\pi^{\frac{d_1}{p-1}})$. Now replacing $K$ with its degree $d_2$ unramified extension, we reduce to the case $d=d_1$. 

Note that $\mathrm{N}_{k/\mathbb{F}_p}(\bar{\mu})^{d_1}=1$ implies that
$\bar{\mu}\in (k^\times)^{e_K}$.
By Hensel's lemma, we have
\[
\mu= \lambda^{e_K}
\]
for some $\lambda\in K_0$. Replacing $\varpi_K$ by $\lambda\varpi_K$, we may further suppose $\mu=1$; thus $E_K(z)$ becomes an Eisenstein polynomial
\[
z^{e_K}+\cdots+p.
\]
Now consider $K'=K(\sqrt[d_1]{\varpi_K})$. Over $K_0(\zeta_p)$, replacing $z$ by $(\zeta_p-1)z$, $E_K(z^{d_1})=0$ reduces to
\[
E(z)=z^{p-1}+\cdots+\frac{p}{(\zeta_p-1)^{p-1}}=0,
\]
where all intermediate coefficients have positive $p$-adic valuations. Using the minimal polynomial $\frac{(1+x)^p-1}{x}$ of $\zeta_p-1$, we get
\[
p/(\zeta_p-1)^{p-1}\equiv -1, \mod p.
\] By Hensel's lemma again, we deduce that $E(z)$, hence $E_K(z^{d_1})$, splits over $K_0(\zeta_p)$, yielding $K'\subset K_0(\zeta_p)$. Hence $K'=K_0(\zeta_p)$ as the ramification index of $K'$ is 
\[
e_Kd_1=p-1.
\]

\end{proof}

\begin{rem}\label{R:bott}
Recall that the Bott element in $\mathrm{K}_2(\OK;\mathbb{Z}/p )$ can be constructed as the Bockstein pre-image of the class in $\mathrm{K}_1(\OK)$ represented by $\zeta_{p}$. Using the cyclotomic trace map, we therefore have another way of showing that $d=1$ if and only if $\zeta_p\in K$.  Moreover, in this case, the Bott element
correponds to the class detected by $\beta$ in the descent spectral sequence.
\end{rem}
\begin{rem}
The proof of Proposition \ref{P:d} depends crucially on the fact that $E_K(0)\equiv p \mod p^2$. We note that there is a unique normalization of the minimal polynomial $E_K(z)$ of $\pK$ over $K_0$ that makes the identity $uv = E_K(z)$ hold in Theorem \ref{T:rtp}(6). That
this normalization agrees with the normalization $E_K(0)=p$ is a non-trivial fact, which is equivalent to the statement of Theorem~\ref{bokperiodicity}(4). In fact, by investigating certain lifts of the class $\beta$ in the algebraic spectral sequences modulo $p^n$ for $K=\mathbb{Q}(\zeta_{p^n})$ and all $n\geq1$, we may have another way of proving Theorem~\ref{bokperiodicity}(4) by the existence of Bott elements in $\mathrm{K}_*(\mathbb{Q}(\zeta_{p^n}))$ for all $n\geq1$.
\end{rem}

\section{Comparison with motivic cohomology}\label{S:motivic}
In this section we compare the descent spectral sequence converging to $\tc_*(\OK;\mathbb{F}_p)$ with the motivic spectral sequence converging to $\mathrm{K}_*(K;\mathbb{F}_p)$. We take $d=2$ for the illustration.

By Theorem \ref{T:e2tc}, the $E^2$-term of the spectral sequence converging to $\tc(\OK;\Fp)$ may be pictured as follows,  in which a circle or box with an arrow means a free $\Fp[\beta]$-module with a basis given by the elements below it. For our convenience, the horizontal axis denotes the total stem $i+j$ and the vertical axis denotes the index $i$. 

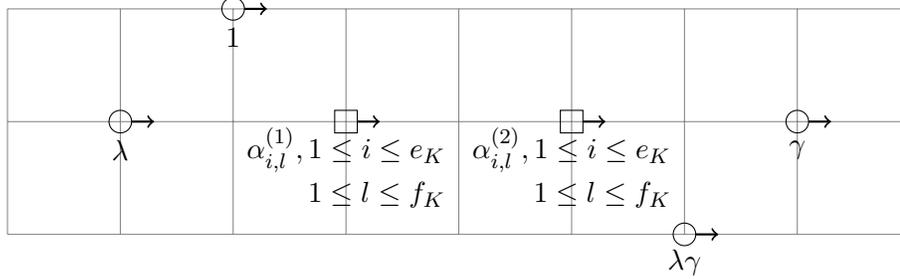
\begin{tikzpicture}[scale  = 1.5]
\draw[help lines] (-2,0) grid (6,2);
\draw  (0,2) circle (0.1);
\draw[->,thick] (0.1,2) -- (0.3,2);
\draw (0,1.75) node {1};

\draw (-1,1) circle (0.1);
\draw[->,thick] (-0.9,1) -- (-0.7,1);
\draw (-1,0.75) node {$\lambda$};

\draw (0.9,0.9) rectangle (1.1,1.1);
\draw[->,thick] (1.1,1) -- (1.3,1);
\draw (1,0.75) node {$\alpha^{(1)}_{i,l}, 1\leq i\leq e_K$};
\draw (1.27, 0.35) node {$1\leq l\leq f_K$};

\draw (2.9,0.9) rectangle (3.1,1.1);
\draw[->,thick] (3.1,1) -- (3.3,1);
\draw (3,0.75) node {$\alpha^{(2)}_{i,l}, 1\leq i\leq e_K$};
\draw (3.27, 0.35) node {$1\leq l\leq f_K$};

\draw (5,1) circle (0.1);
\draw[->,thick] (5.1,1) -- (5.3,1);
\draw (5,0.75) node {$\gamma$};

\draw (4,0) circle (0.1);
\draw[->,thick] (4.1,0) -- (4.3,0);
\draw (4,-0.25) node {$\lambda\gamma$};
\end{tikzpicture}

Let $\beta$ be a generator of $\mu_p^d$, which is isomorphic to $\mathbb{Z}/p$ as a $\mathrm{Gal}(\overline{K}/K)$-module.  Let $\alpha^{(1)}$ be a generator of the $\OK/p$-module 
\[
U_K/U_K^p\subset K^\times/(K^\times)^p \cong H^1_\text{\'{e}t}(K, \mu_p),
\]
where $U_K$ is the torsion free part of $\mathcal{O}_K^\times$. Let $\beta^{-1}\gamma\in H^1_\text{\'{e}t}(K, \mu_p)$ be the class represented by $\overline{\pK}\in K^\times/(K^\times)^p$. Let $\lambda$ be the element of 
\[
H^1_\text{\'{e}t}(K, \mathbb{Z}/p)= \mathrm{Hom}(\mathrm{Gal}(\overline{K}/K),\mathbb{Z}/p)\cong \mathrm{Hom}(K^\times/(K^\times)^p, \mathbb{Z}/p)
\]
corresponding to the unramified character sending Frobenius to 1. Let $\beta^{-1}\alpha^{(2)}\in H^1_\text{\'{e}t}(K, \mathbb{Z}/p)$  be a generator of the $\OK /p$-module 
$\mathrm{Hom}(U_K/U_K^p, \mathbb{Z}/p).$ It follows that $\beta^{-1}\lambda\gamma\in H^2_\text{\'{e}t}(K, \mu_p)$ corresponds to the division algebra of invariant $1\over p$ in the Brauer group. 

The \'{e}tale spectral sequence $E_{i,j}^2=H^{-i}_\text{\'{e}t}(K, \mu_p^{\otimes j}) \Rightarrow L_{K(1)}\mathrm{K}_{2j+i}(K, \Fp)$ may be pictured as follows, where a circle (resp. box) with two arrows means a free $\Fp[\beta, \beta^{-1}]$-module (resp. $(\OK/p)[\beta, \beta^{-1}]$-module) with a basis given by the elements below it. 

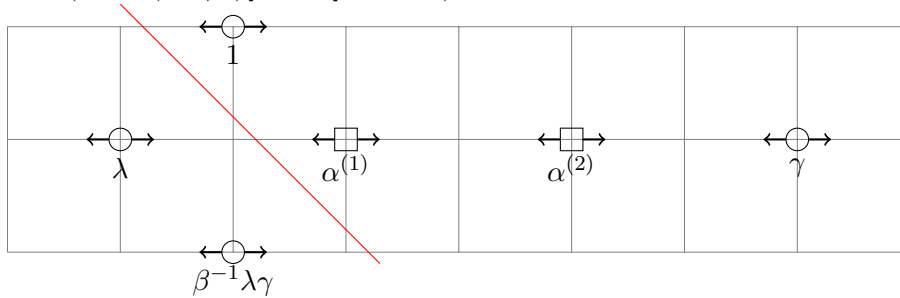
\begin{tikzpicture}[scale  = 1.5]
\draw[help lines] (-2,0) grid (6,2);
\draw  (0,0) circle (0.1);
\draw[->,thick] (0.1,0) -- (0.3,0);
\draw[->,thick] (-0.1,0) -- (-0.3,0);
\draw (0,-0.25) node {$\beta^{-1}\lambda\gamma$};

\draw (-1,1) circle (0.1);
\draw[->,thick] (-0.9,1) -- (-0.7,1);
\draw[->,thick] (-1.1,1) -- (-1.3,1);
\draw (-1,0.75) node {$\lambda$};

\draw (0.9,0.9) rectangle (1.1,1.1);
\draw[->,thick] (1.1,1) -- (1.3,1);
\draw[->,thick] (0.9,1) -- (0.7,1);
\draw (1,0.75) node {$\alpha^{(1)}$};

\draw (2.9,0.9) rectangle (3.1,1.1);
\draw[->,thick] (3.1,1) -- (3.3,1);
\draw[->,thick] (2.9,1) -- (2.7,1);
\draw (3,0.75) node {$\alpha^{(2)}$};

\draw (5,1) circle (0.1);
\draw[->,thick] (5.1,1) -- (5.3,1);
\draw[->,thick] (4.9,1) -- (4.7,1);
\draw (5,0.75) node {$\gamma$};

\draw (0,2) circle (0.1);
\draw[->,thick] (0.1,2) -- (0.3,2);
\draw[->,thick] (-0.1,2) -- (-0.3,2);
\draw (0,1.75) node {1};

\draw[color=red] (-1.0,2.2) -- (1.3, -0.1);
\end{tikzpicture}

Using the Bloch-Kato conjecture proved by Voevodsky \cite{V11},  the $E^2$-term of the motivic spectral sequence converging to $\mathrm{K}_*(K;\mathbb{F}_p)$ may be identified with the part to the right of the red line of the \'etale spectral sequence:

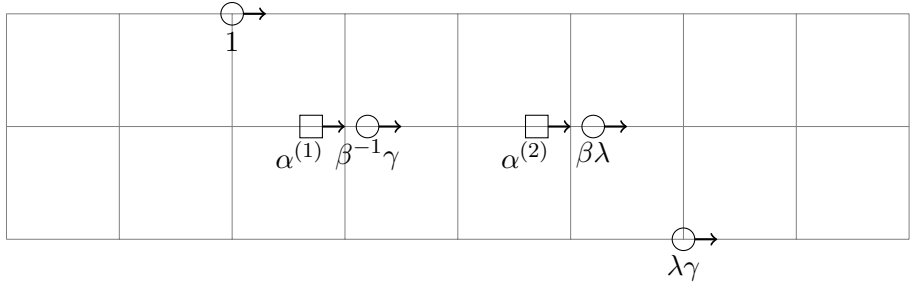
\begin{tikzpicture}[scale  = 1.5]
\draw[help lines] (-2,0) grid (6,2);
\draw  (0,2) circle (0.1);
\draw[->,thick] (0.1,2) -- (0.3,2);
\draw (0,1.75) node {1};

\draw (3.2,1) circle (0.1);
\draw[->,thick] (3.3,1) -- (3.5,1);
\draw (3.2,0.75) node {$\beta\lambda$};

\draw (0.6,0.9) rectangle (0.8,1.1);
\draw[->,thick] (0.8,1) -- (1,1);
\draw (0.6,0.75) node {$\alpha^{(1)}$};

\draw (2.6,0.9) rectangle (2.8,1.1);
\draw[->,thick] (2.8,1) -- (3,1);
\draw (2.6,0.75) node {$\alpha^{(2)}$};

\draw (1.2,1) circle (0.1);
\draw[->,thick] (1.3,1) -- (1.5,1);
\draw (1.2,0.75) node {$\beta^{-1}\gamma$};

\draw (4,0) circle (0.1);
\draw[->,thick] (4.1,0) -- (4.3,0);
\draw (4,-0.25) node {$\lambda\gamma$};
\end{tikzpicture}

One may show that $\lambda$ generates the cokernel of the cyclotomic trace map 
\[
\mathrm{K}(\Zp;\Fp) \rightarrow \tc(\Zp;\Fp).
\]
We thus see the similarity between the descent spectral sequence and motivic spectral sequence. We expect that, for certain algebraic varieties over $\OK$,  one would be able to construct some analogue of the motivic spectral sequence which determines the structure of the algebraic $K$-theory with $\Fp$-coefficients and is compatible with the descent spectral sequence via the cyclotomic trace map.

\appendix
\section{A variant of Hochschild-Kostant-Rosenberg}

The goal of this appendix is to prove the following theorem. 

\begin{thm}\label{T:AHKR}
Let $R$ be a commutative ring over $\mathbb{Z}_p$, and let $I$ be a locally complete intersection ideal of $R$. Let $A=R/I$. Suppose that $R$ is $I$-separated and $A$ is $p$-torsion free. Then  as filtered rings, the periodic cyclic homology $\mathrm{HP}_0(A/R)$ is canonically isomorphic to the completion of $D_R(I)$ with respect to the Nygaard filtration. Moreover, the Tate spectral sequence for $\mathrm{HP}_0(A/R)$ collapses at the $E^2$-term. Consequently,  
there is a canonical isomorphism of graded rings 
\[
\mathrm{HH}_*(A/R) \cong \Gamma_A( I/I^2).
\]
%under the canonical isomorphism $I/I^2\cong \mathrm{HH}_2(A/R)$. 
\end{thm}
In the following, all tensor products are taken in the derived category of $R$-modules.  
\begin{proof}
We first assume $I=(a)$ for a non-zero divisor $a\in R$. Let $D=R[x]/(x^2)$ be the commutative DG-algebra over $R$ with $|x|=1$ and $d(x) = a$. Then $D$ is a flat resolution of $A$ over $R$. Thus 
$\mathrm{HH}(A/R)$ can be computed by the normalized Hochschild complex (cf. \cite[\S1.1.4]{Lo}) as follows. Let $\bar{D} = D/R$. Set the double complex $C_*(D)$ as
$$\dots \rightarrow D\otimes \bar{D}^{\otimes2} \xrightarrow{b} D\otimes\bar{D}\xrightarrow{b} D.$$
The boundary map $b:D\otimes\bar{D}^{\otimes n}\rightarrow D\otimes\bar{D}^{\otimes n-1}$  is given by the  formula
$ \sum_{i=0}^n (-1)^i \bar{d}_i$
with $d_i:D^{\otimes n+1}\rightarrow D^{\otimes n}$, where
 $$d_0 = m\otimes \mathrm{id} \otimes\dots \otimes \mathrm{id},$$
$$d_1 =\mathrm{id} \otimes m \otimes\dots\otimes \mathrm{id},$$
$$\dots$$
$$d_n = (m\otimes \mathrm{id} \otimes\dots\otimes \mathrm{id})\circ t';$$
here $m:D\otimes D\rightarrow D$ is the multiplication and $t': D^{\otimes n+1}\rightarrow D^{\otimes n+1}$ is the cyclic permutation operator sending the last factor to the first.

A short computation shows that $b(1\otimes x^{\otimes n}) = 1\otimes x^{\otimes n-1}+(-1)^{n+n-1}1\otimes x^{\otimes n-1}=0$. 
It follows that for $n\geq0$, $\mathrm{HH}_{2n+1}(A/R) = 0$ and  $\mathrm{HH}_{2n}(A/R) \cong A$, where the latter is generated by the element represented by the cycle $1\otimes x^{\otimes n}$. In particular, the graded ring $\mathrm{HH}_{*}(A/R)$ is $p$-torsion free as $A$ is $p$-torsion free. 

Next we determine $\mathrm{HP}_*(A/R)$.  By \cite[\S 2.1.9]{Lo}, this may be computed by the double complex
\begin{equation}\label{E:hp-double}
\xymatrix{
\dots &\dots \ar[d]_b &\dots \ar[d]_b &\dots\ar[d]_b  \\
\dots & D\otimes \bar{D}^{\otimes 2}\ar[l]_{B}\ar[d]_b& D\otimes\bar{D} \ar[l]_{B}\ar[d]_b & D\ar[l]_{B} \\
\dots & D\otimes\bar{D} \ar[l]_{B}\ar[d]_b & D\ar[l]_{B} \\
\dots & D\ar[l]_{B}\\
\dots
}
\end{equation}
Here $B:D\otimes \bar{D}^{\otimes n} \rightarrow D\otimes 
\bar{D}^{\otimes n+1}$ is given by the formular
$s \circ N$, where
\[
N= 1+t+\dots +t^n
\]
with $t=(-1)^nt'$
 and $s=u\otimes \mathrm{id}:D^{\otimes n+1}\rightarrow D^{\otimes n+2}$ with $u:R\rightarrow D$ being the unit map. Note that the spectral sequence associated to the double complex \eqref{E:hp-double} is the Tate spectral sequence for $\mathrm{HP}(A/R)$, which collapses at the $E^2$-term because everything is concentrated in even degrees. It follows that the associated graded algebra of $\mathrm{HP}_0(A/R)$ with respect to the Nygaard filtration is isomorphic to $\mathrm{HH}_{*}(A/R)$. This yields that $\mathrm{HP}_0(A/R)$ is $p$-torsion free. 
  
We claim that the natural map $R\rightarrow \mathrm{HP}_0(A/R)$ extends uniquely to a map
\begin{equation}\label{E:hp}
D_I(R)\rightarrow \mathrm{HP}_0(A/R)
\end{equation}
as filtered rings. The uniqueness follows from the fact that $\mathrm{HP}_0(A/R)$ is $p$-torsion free.
For the existence, first note that the image of $a^n$ is represented by the same element in the above double complex. Since $d(a^{n-1}x) = a^n$, $b$ is trivial, $B(a^{n-1}x) = a^{n-1}\otimes x$, we get that $a^n$ is homologous to $a^{n-1}\otimes x$ up to a sign. By induction, we deduce that for $0\leq m\leq n$, $a^n$ is homologous to $m!a^{n-m}\otimes x^{m}$ up to a sign. In particular, $a^n$ is homologous to $n!\otimes x^{\otimes n}$ up to a sign. This proves the claim.
 
To proceed, first note that $R$ is $p$-torsion free by the assumption that it is $I$-separated and $A$ is $p$-torsion free. This implies that the associated graded algebra of $D_I(R)$ is isomorphic to $\Gamma_A(I/I^2)$. Now taking the associated graded of the natural map $D_I(R)\rightarrow \mathrm{HP}_0(A/R)$, we get the map
\begin{equation}\label{E:hh}
\Gamma_A(I/I^2)\cong A\langle \bar{a}\rangle\to \mathrm{HH}_*(A/R).
\end{equation}
By what we have proved, one easily checks that $\bar{a}^{[n]}$ maps to the element represented by the cycle $1\otimes x^{\otimes n}$. This implies that \eqref{E:hh} is an isomorphism of graded rings.  It follows that \eqref{E:hp} becomes an isomorphism after taking completion with respect to the Nygaard filtration on the source.

Now suppose that $I$ is generated by a regular sequence $a_1,\dots,a_n$. By the fact that $R$ is $I$-separated and $A$ is $p$-torsion free, we first deduce that $A_i=R/(a_i)$ is $p$-torsion free for all $i$. Using the previous case and the natural isomorphism $A\cong A_1\otimes\dots\otimes A_n$, 
we get the  isomorphism
$$\mathrm{HH}_{*}(A/R) \cong \mathrm{HH}_{*}(A_1/R)\otimes\dots \otimes\mathrm{HH}_{*}(A_n/R)\cong \Gamma_A(I/I^2),$$
where the second isomorphism follows from
\[
\Gamma_{A_1}((a_1)/(a_1)^2)\otimes\dots\otimes \Gamma_{A_n}((a_n)/(a_n)^2)\cong \Gamma_A(I/I^2).
\]
This implies that $\mathrm{HH}_{*}(A/R)$ is $p$-torsion free and the Tate spectral sequence for $\mathrm{HP}(A/R)$ collapses at the $E^2$-term. We thus get that $\mathrm{HP}_0(A/R)$ is $p$-torsion free. Moreover, using the natural map
\begin{equation*}
\mathrm{HP}_0(A_1/R)\otimes\dots\otimes \mathrm{HP}_0(A_n/R) \rightarrow \mathrm{HP}_0(A/R)
\end{equation*}
and the isomorphism
\[
D_{(a_1)}(R)\otimes\dots\otimes D_{(a_n)}(R)\cong D_I(R),
\]
we obtain the map $D_I(R)\to \mathrm{HP}_0(A/R)$, which uniquely extends the natural map $R\rightarrow \mathrm{HP}_0(A/R)$. By the same argument as in the previous case,  we deduce that it becomes an isomorphism after taking completion on the source.

For general $I$,  by Zariski descent and the previous case, we first have that $\mathrm{HH}_{*}(A/R)$ is concentrated on even degrees. It follows that the Tate spectral sequence for $\mathrm{HP}(A/R)$ collapses at the $E^2$-term. Similarly, we get that $\mathrm{HP}_0(A/R)$ is $p$-torsion free and $R\to \mathrm{HP}_0(A/R)$ uniquely extends to $D_I(R)\rightarrow \mathrm{HP}_0(A/R)$. Moreover, by Zariski descent and the previous case, the associated graded map $\Gamma_A(I/I^2)\to \mathrm{HH}_{*}(A/R)$ is an isomorphism. This concludes the proof. 
\end{proof}

\end{document}